\documentclass{article}

\usepackage{amsmath,amssymb,amsthm}
\usepackage{mathtools}
\usepackage[all,knot]{xy}
\usepackage{enumitem}
\usepackage{mymacro}
\usepackage{comment}

\theoremstyle{plain}
\newtheorem{theo}{Theorem}[section]
\newtheorem{lemma}[theo]{Lemma}
\newtheorem{prop}[theo]{Proposition}
\newtheorem{corol}[theo]{Corollary}

\theoremstyle{definition}
\newtheorem*{defin}{Definition}

\theoremstyle{remark}
\newtheorem*{rem}{Remark}
\newtheorem{exam}[theo]{Example}

\begin{document}
\title{A general method to construct cube-like categories and applications to homotopy theory}
\author{Jun Yoshida}
\date{}
\maketitle
\tableofcontents
\section*{Introduction}
\addcontentsline{toc}{section}{Introduction}
\theoremstyle{plain}
\newtheorem*{theo0}{Theorem}
\newtheorem*{prop0}{Proposition}
Nowadays, the term ``homotopy theory'' on a category is understood as a model structure.
The notion of model structure was originally introduced by Quillen in \cite{Qui}, where he claimed that the study of weak homotopy types of topological spaces is equivalent to that of model structure on the category $\mathbf{Top}$ of topological spaces.
Quillen also showed that the category $\mathbf{SSet}$ of simplicial sets also admits a model structure, called Quillen model structure, and the homotopy theory on $\mathbf{SSet}$ is equivalent to that of $\mathbf{Top}$ in some sense.
So many researchers regard $\mathbf{SSet}$ as the most essential category for homotopy theory.
On the other hand, Grothendieck conjectured in \cite{Gro} that we can extend the homotopy theory on  more general presheaf topoi; the basic ideas are as follows:
We call a functor $f:C\to D\in\mathbf{Cat}$ a Thomason equivalence if it induces a weak equivalence in $\mathbf{SSet}$ on the nerves, where the name is after the person who introduced a model structure on $\mathbf{Cat}$ with the weak equivalences.
If $\mathcal{A}$ is a small category, we have a functor $\txtint^{\mathcal{A}}:\mathcal{A}^\wedge\to\mathbf{Cat}$ which sends a presheaf $X$ on $\mathcal{A}$ to the slice category $\mathcal{A}/X$.
It is known that the functor $\txtint^{\mathcal{A}}$ is a left adjoint functor, so there is a right adjoint $N_{\mathcal{A}}$, unit $\eta:\mathrm{Id}\to N_{\mathcal{A}}\int^{\mathcal{A}}$ and counit $\varepsilon:\txtint^{\mathcal{A}}N_{\mathcal{A}}\to\mathrm{Id}$.
A small category $\mathcal{A}$ is called a weak test category if for each small category $C$,  $\varepsilon:\txtint^{\mathcal{A}}N_{\mathcal{A}}C\to C$ is a Thomason equivalence.
A test category is a weak test category $\mathcal{A}$ such that for each $a\in\mathcal{A}$, the slice category $\mathcal{A}/a$ is also a weak test category.
Grothendieck conjectured that if $\mathcal{A}$ is a test category, the composition
\[
\mathcal{A}^\wedge\xrightarrow{\txtint^{\mathcal{A}}}\mathbf{Cat}\xrightarrow{N}\mathbf{SSet}
\]
would derive a homotopy theory on $\mathcal{A}^\wedge$ from $\mathbf{SSet}$; that is, there is a model structure on $\mathcal{A}^\wedge$ on which a morphism $f:X\to Y$ is a weak equivalence if and only if it induces a Thomason equivalence $\mathcal{A}/X\to\mathcal{A}/Y$, and $f$ is a cofibration if and only if it is a monomorphism.
This formalization was completed by Cisinski in \cite{Cis}, where he in fact developed the model structure on $\mathcal{A}^\wedge$ for any test category $\mathcal{A}$ such that $f:X\to Y$ is a weak equivalence if and only if the induced functor $\mathcal{A}/X\to\mathcal{A}/Y$ is a Thomason equivalence, and $f:X\to Y$ is a cofibration if and only if $f$ is a monoomrphism.

Cisinski also showed that the examples of test categories contain the category $\square$.
The category $\square$ is called the cubical category, and a presheaf on it is called a cubical set, which often appeared in the early of the development of homotopy theory, for example see \cite{Kan55} and \cite{Ser51}.
Jordine pointed out in \cite{Jar} that the model structure on $\square\mathbf{Set}$ given by Cisinski agrees with the classical homotopy theory on cubical sets.
More precisely, he showed that there is a functor $|\blankdot|:\square\mathbf{Set}\to\mathbf{SSet}$, called the realization, which gives rise to a Quillen equivalence between model categories.
However, it was revealed that there are some difficulties in the homotopy theory on the category $\square\mathbf{Set}$.
For example, cartesian products of cubes are not contractible; and the canonical monoidal structure $\otimes:\square\mathbf{Set}\times\square\mathbf{Set}\to\square\mathbf{Set}$ is not symmetric.
To avoid them, some varieties of cubical categories have been considered; for example, Brown and Higgins introduced the additional structure called the connections on cubical sets, see \cite{BHS}.
The resulting category is called a cubical category with connections, often denoted by $\square^{\mathsf{c}}$.
They used it to prove the cubical Dold-Kan correspondence.
Maltsiniotis, moreover, showed that the category $\square^{\mathsf{c}}$ is also a test category, and cartesian products in the category of cubical sets with connections have ``right'' homotopy types.
Another example is the extended cubical category $\square_\Sigma$ introduced by Isaacson in \cite{Isa}.
It is a symmetric version of $\square^{\mathsf{c}}$ so that the canonical monoidal structure $\square_\Sigma$ is symmetric.
Isaacson constructed $\square_\Sigma\mathbf{Set}$ enrichments of symmetric monoidal model categories.
For more variations of cubical categories, see \cite{GM03}.

The main purposes of this article is to generalize the construction of these variations and to give model structures on resulting presheaf topoi.
For generalization of cubes, we have the following observation:
Let $\mathcal{Q}=\square$ or $\square^{\mathsf{c}}$.
Then for every morphism $f:[1]^m\to[1]^n\in\mathcal{Q}$, there is a triple $(A,\mu,B)$ of subsets $A\subset m,B\subset n$ and an order preserving map $\mu:A\to n$ such that for each subset $S\subset m$, we have $f(S)=\mu(S\cap A)\cup B$, where we identify natural numbers with the corresponding ordinals.
In other words, there is a category $\mathcal{R}$ together with a functor $P:\mathcal{R}\to\mathbf{Poset}$ such that
\begin{itemize}
  \item $\objof{\mathcal{Q}} = \objof{\mathcal{R}}$;
  \item $P$ sends each object $r\in\mathcal{R}$ to a finite Boolean lattice $P(r)$, called the poser set;
  \item every morphism $r'\to r\mathcal{Q}$ is identified with a triple $(A,\mu,B)$ of $A\in P(r'), B\in P(r)$ and a morphism $\mu\in\mathcal{R}$.
\end{itemize}
In fact, if we take $\mathcal{Q}=\widetilde\Delta^+$ (resp. $\widetilde\Delta$), the category of finite ordinals and order-preserving injections (resp. arbitrary order-preserving maps), then we obtain $\square$ (resp. $\square^{\mathsf{c}}$).
This observation indicates that some variations of cubical categories can arise from small categories which admit ``power sets'' notions.

We will generalize ``cubes'' in this way.
The construction begins with a small category $\mathcal{R}$ equipped with a unique factorization system $(\mathcal{R}^-,\mathcal{R}^+_0)$.
We assume the following:
\[
\newdimen\temptxtwidth%
\temptxtwidth=\textwidth%
\advance\temptxtwidth by -5em%
(\star1)\left\{\begin{minipage}{\temptxtwidth}
\begin{itemize}
  \item $\mathcal{R}^+_0$ has only monomorphisms and no non-trivial isomorphism, so that $\mathcal{R}^+_0\!/(\blankdot):\mathcal{R}\to\mathbf{Poset}$ is a functor.
  \item For each $r\in\mathcal{R}$, $\mathcal{R}^+_0\!/r$ is a finite Boolean lattice.
  \item For each $f:r'\to r\in\mathcal{R}$, the induced map $f_\ast:\mathcal{R}^+_0\!/r'\to\mathcal{R}^+_0\!/r$ admits a right adjoint $f^\ast:\mathcal{R}^+_0\!/r\to\mathcal{R}^+_0\!/r$ which preserves joins.
\end{itemize}
\end{minipage}\right.
\]
We will see these conditions in section \ref{sec:thin-pow-cat} and \ref{sec:pos-repn} more accurately.
We also consider an additional structure, which we call an enlargement.
If $\mathcal{R}$ is a small category satisfying the conditions above, then an enlargement on $\mathcal{R}$ is a functor $c:\mathcal{R}\to\mathcal{R}$ together with a natural transformation $\iota:\mathrm{Id}\to c$ such that
\begin{enumerate}[label={\rm(\arabic*)}]
  \item for each $r\in\mathcal{R}$, there is a natural isomorphism $\mathcal{R}^+_0\!/c(r)\simeq\mathcal{R}^+_0\!/r\times[1]$;
  \item $c$ preserves $(\mathcal{R}^+_0,\mathcal{R}^-)$-factorization; i.e. if $\sigma\in\mathcal{R}^-$ and $\delta\in\mathcal{R}^+_0$, then $c(\sigma)\in\mathcal{R}^-$ and $c(\delta)\in\mathcal{R}^+_0$;
  \item for each $f:r\to r'\in\mathcal{R}$, we have $c(f)^\ast\iota = \iota$.
\end{enumerate}
Our first result is the following:

\begin{theo0}
Let $\mathcal{R}$ be a small category equipped with a unique factorization system $(\mathcal{R}^-,\mathcal{R}^+_0)$ satisfying the condition $(\star1)$.
Then we can construct a small category $\square(\mathcal{R})$.
Moreover, if $\mathcal{R}$ admits an enlargement $c:\mathcal{R}\to\mathcal{R}$, then $\square(\mathcal{R})$ is a test category, so that there is a model structure on $\square(\mathcal{R})^\wedge$ such that $f:X\to Y$ is a weak equivalence if and only if $\square(\mathcal{R})/X\to\square(\mathcal{R})/Y$ is a Thomason equivalence, and $f:X\to Y$ is a cofibration if and only if it is a monomorphism.
\end{theo0}

If $\mathcal{R}$ satisfies the condition $(\star1)$, we say $\mathcal{R}$ is cubicalizable and call $\square(\mathcal{R})$ the cubicalization of $\mathcal{R}$.
The model structure described above is called the standard model structure.

The construction of $\square(\mathcal{R})$ is splitted into two steps.
At first, we will observe spans in $\mathcal{R}$ in section \ref{sec:rel-span}.
In particular, we are interested in spans of the form
\[
\xymatrix@C=2em@R=2ex{
  {} & s \ar[dl]_{\gamma} \ar[dr]^{f} & \\
  r' && r }
\]
with $\gamma\in\mathcal{R}^+_0$.
These spans form the morphism class of a category $\mathcal{V}(\mathcal{R})$, which we will define.
We denote such a span by $f\gamma^\dagger$
It is also verified that the functor $\mathcal{R}^+_0\!/(\blankdot):\mathcal{R}\to\mathbf{Poset}$ extends to $\mathcal{V}(\mathcal{R})$.
Next, we will introduce the notion of crossed modules for categories in section \ref{sec:crsrep}, which is a generalization of classical crossed modules for groupoids.
In this step, we prove that for a crossed $\mathcal{J}$-module $(M,\mu)$, there is a category $M\!\mathcal{J}$ whose objects are those of $\mathcal{J}$ and whose hom-set from $j$ to $k$ is given by
\[
M\!\mathcal{J}(j,k):=\left\{(a,f)\in M(k)\times\mathcal{J}(j,k)\mid \mu(a)f=f\right\}\,.
\]
The composition is given by
\[
(b,g)\circ (a,f) = \left(b\!\cdot\!g_{\ast}a, \mu(g_{\ast}b)gf\right)\,.
\]
One can verify that $M=\mathcal{R}^+_0\!/(\blankdot):\mathcal{V}(\mathcal{R})\to\mathbf{SemiLat}^\vee$ has a crossed $\mathcal{V}(\mathcal{R})$-module structure by setting $\mu_r:\mathcal{R}^+_0\!/r\to\mathrm{End}_{\mathcal{V}(\mathcal{R})}(r)$ to be $\mu_r(\xi):=(\neg\xi)(\neg\xi)^\dagger$, where $\neg(\blankdot)$ denotes the negation in the Boolean lattice.
Finally, we set $\square(\mathcal{R}):=M\!\mathcal{V}(\mathcal{R})$.
We will see $\square(\mathcal{R})$ is actually a test category in section \ref{sec:test-cat}.

The classical cubical categories $\square$ and $\square^{\mathsf{c}}$ are examples of resulting categories $\square(\mathcal{R})$.
Moreover, we can also considered the symmetrized version of them by group operads.
Group operad $G$, which is discussed in \cite{Wah} and \cite{Zha}, is a non-symmetric operad such that each $G(n)$ has a group structure with certain compatibility with the operad operations.
We can define $\square_G$ the cubical category with symmetry of $G$ as follows:
For a goup operad $G$, we can naturally regard it as a crossed $\widetilde\Delta$-group defined in \cite{BM}.
$\square_G$ is the cubicalization of $\widetilde\Delta G$, the total category of crossed $\widetilde\Delta$-group $G$.
If we take $G=\Sigma$ the symmetric group operad, we obtain $\square_\Sigma=\square(\widetilde\Delta\Sigma)$, which is constructed by Isaacson.

The second result is about relations between homotopy theories on the presheaf category over $\square(\mathcal{R})$.
Since $\square(\mathcal{R})$ is a test category, as mentioned above, the category $\square(\mathcal{R})^\wedge$ of presheaves admits a model structure whose weak equivalences are morphisms $f:X\to Y$ which induce Thomason equivalences $\square(\mathcal{R})/X\to \square(\mathcal{R})/Y$.
On the other hand, since the power set functor $\mathcal{R}^+_0\!/(\blankdot):\square(\mathcal{R})\to\mathbf{Cat}$ sends each object to finite Boolean algebra, by taking the nerves, we obtain a functor $|\blankdot|:\square(\mathcal{R})\to\mathbf{SSet}$ which sends each object to cubes $\Delta[1]^n$.
We define the simplicial realization $|\blankdot|:\square(\mathcal{R})^\wedge\to\mathbf{SSet}$ by the left Kan extension along the Yoneda embedding $\square(\mathcal{R})\hookrightarrow\square(\mathcal{R})^\wedge$.
If the unique factorization system $(\mathcal{R}^-,\mathcal{R}^+_0)$ gives rise to a structure of Eilenberg-Zilber category (briefly, EZ category) on $\mathcal{R}$, we can derive a model structure on $\square(\mathcal{R})^\wedge$ from $\mathbf{SSet}$ along the realization functor.
It is non-trivial that the two model structure coincide.
We will give a proof for this in some special cases in section \ref{sec:htp-cube}:

\begin{theo0}
Let $\mathcal{R}$ be an EZ cubicalizable category with an enlargement.
Suppose moreover that $\mathcal{R}$ has no non-trivial isomorphism.
Then a morphism $f:X\to Y\in\square(\mathcal{R})^\wedge$ induces a Thomason equivalence $\square(\mathcal{R})/X\to\square(\mathcal{R})/Y$ if and only if the simplicial realization $|f|:|X|\to |Y|$ is a weak equivalence in $\mathbf{SSet}$.
\end{theo0}

For the proof, the notion of regularity is important, which is related to the homotopy colimit.
A model structure on a presheaf category $\mathcal{A}^\wedge$ is said to be regular if the natural morphism
\[
\hocolim_{(\mathcal{A}[a]\to X)\in\mathcal{A}/X}\mathcal{A}[a]\,\to X
\]
is a weak equivalence for every $X\in\mathcal{A}^\wedge$.
The regularity gives very strong tools for computation of homotopy colimits.
In fact, if a model structure on $\mathcal{A}$ is regular, then every object is a homotopy colimit of representables, so the left Quillen functor is easily computed.
So we will first show that if $\mathcal{R}$ is an EZ cubicalizable category with an enlargement and has no non-trivial isomorphism, the model structure $\mathcal{R}$ induced by the realization is regular.
Then it is easy to verify that if the model structure on $\square(\mathcal{R})^\wedge$ is regular, we have a zigzag of weak equivalences in $\mathbf{SSet}$ connecting $|X|$ to $N(\square(\mathcal{R})/X)$, which implies the result.
Note that the theorem is a generalization of the Antolini's result presented in \cite{Ant02}.
For example, the theorem is available for $\square$ and $\square^{\mathsf{c}}$.

As for more general $\mathcal{R}$, we need another assumption.
We suppose $\mathcal{R}$ satisfies the following condition:
\begin{enumerate}[label={\rm$(\clubsuit)$}]
  \item For each $r\in\square(\mathcal{R})$, there is a zigzag of homotopy equivalences connecting $\square_{\mathcal{R}}[r]$ to the terminal object.
\end{enumerate}
Then we can develop another model structure on $\square(\mathcal{R})^\wedge$, called the spatial model structure:

\begin{theo0}
Let $\mathcal{R}$ be an EZ cubicalizable category with an enlargement.
If $\mathcal{R}$ satisfies the condition $(\clubsuit)$, then there is a cofibrantly generated model structure on $\square(\mathcal{R})^\wedge$ such that
\begin{itemize}
  \item $f:X\to Y$ is a weak equivalence if and only if it is an $\infty$-equivalence;
  \item the class of cofibrations is the saturated class generated by the set
\[
\left\{\partial\square_{\mathcal{R}}[r]\hookrightarrow\square_{\mathcal{R}}[r]\,:\,r\in\square(\mathcal{R})\right\}\,;
\]
  \item the class of fibrations is the right orthogonal class of trivial cofibrations.
\end{itemize}
\end{theo0}

In the spatial model structure, we obtain the criterion below:

\begin{prop0}
Let $\mathcal{R}$ be an EZ cubicalizable category satisfying the condition \ref{cond:club} with an enlargement.
Suppose $X,Y\in\square(\mathcal{R})^\wedge$ are cofibrant in the spatial model structure
Then $f:X\to Y$ is an $\infty$-equivalence if and only if the geometric realization $|f|:|X|\to |Y|$ is a weak equivalence in $\mathbf{SSet}$.
\end{prop0}

It will be verified that, for every group operad $G$, the cubicalizable category $\widetilde\Delta G$ satisfies the condition $(\clubsuit)$.
Hence the category $\square(\widetilde\Delta G)^\wedge$ admits the spatial model structure.
Moreover, it will be also proved that the convolution product on $\square(\widetilde\Delta G)^\wedge$ is compatible with the spatial model structure.
In other words, $\square(\widetilde\Delta G)^\wedge$ is an example of a monoidal model category.
In particular the case $G=\mathcal{B}$ the braid group operad, $\square(\widetilde\Delta\mathcal{B})^\wedge$ admits a braiding, so that it is a braided monoidal model category.

\subsection*{Acknowledgements}
This paper is written under the supervision of Toshitake Kohno.
I would like to thank him for the encouragement and advice.
This work was supported by the Program for Leading Graduate Schools, MEXT, Japan.

\section{Preliminaries}
\label{sec:prelim}

\subsection{Notations}
Throughout this paper, we use the convension that the set of natural numbers contains $0$ and denote it by $\mathbb{N}$.

In this article, we will refer to the book \cite{Lei} for elements of category theory.
The most of terminologies and notations follow the book.
For more details of category theory, we refer the reader to literatures \cite{KS} and \cite{McL}.
We here review a few terminologies:

\begin{defin}
Let $\mathcal{C}$ be a category.
\begin{itemize}
  \item A subcategory $\mathcal{C}_1$ of $\mathcal{C}$ is said to be full if for each pair $(c_1,c'_1)$ of objects in $\mathcal{C}_1$, we have $\mathcal{C}_1(c_1,c'_1)=\mathcal{C}(c_1,c'_1)$.
  \item A subcategory $\mathcal{C}_2$ of $\mathcal{C}$ is said to be wide if $\mathrm{Ob}\,\mathcal{C}_2=\mathrm{Ob}\,\mathcal{C}$.
\end{itemize}
\end{defin}

We often identify full subcategories with subclasses of objects, and wide subcategories with subclasses of morphisms which are closed under compositions.

\begin{defin}
A category $\mathcal{C}$ is said to be complete (resp. cocomplete) if all small limits (resp. colimits) exist in $\mathcal{C}$.
We say $\mathcal{C}$ is bicomplete if it is both complete and cocomplete.
\end{defin}

For a category $\mathcal{C}$, we denote by $\mathcal{C}^\wedge$ the presheaf category on $\mathcal{C}$, and by $\mathcal{C}^\vee$ the category of copresheaves.
We write $\mathcal{C}[\blankdot]:\mathcal{C}\to\mathcal{C}^\wedge$ the Yoneda embedding.
Hence, the co-Yoneda embedding is denoted by $\mathcal{C}^\opposite[\blankdot]^\opposite:\mathcal{C}\to\mathcal{C}^\vee$.
We often identify $\mathcal{C}$ with a full subcategory of $\mathcal{C}^\wedge$.

For a small category $\mathcal{J}$ and a category $\mathcal{C}$, 
We denote by $\mathcal{C}^\mathcal{J}$ the functor category whose objects are functors $\mathcal{J}\to\mathcal{C}$ and whose morphisms are natural transformations.
In particular, if $[1]$ denotes the category of the form $\bullet\to\bullet$, then the category $\mathcal{C}^{[1]}$ is the category of morphisms of $\mathcal{C}$; the objects are morphisms of $\mathcal{C}$, and the morphisms are commutative squares in $\mathcal{C}$.

For functors $F:\mathcal{C}\to\mathcal{E}$ and $G:\mathcal{D}\to\mathcal{E}$, we denote by $\commacat{F}{G}$ the comma category; whose objects are triples $(c,d,u)$ of $c\in\mathcal{C}$, $d\in\mathcal{D}$ and $u:F(c)\to G(d)$, and whose morphisms are pairs $(f:c\to c',g:d\to d')$ which commute the squares:
\[
\xymatrix{
  F(c) \ar[r]^{u} \ar[d]_{F(f)} & G(d) \ar[d]^{G(g)} \\
  F(c') \ar[r]^{v} & G(d') }
\]
As a special case, for a functor $G:\mathcal{D}\to\mathcal{C}^\wedge$ we write $\mathcal{C}/G=\commacat{\mathcal{C}[\blankdot]}{G}$ and $F/\mathcal{C}=\commacat{F}{\mathcal{C}[\blankdot]}$.
Dually, for a functor $F:\mathcal{D}\to\mathcal{C}^\vee$, we write $F/\mathcal{C}=\commacat{F}{\mathcal{C}^\opposite[\blankdot]^\opposite}$.
We call them the slice categories.
In particular the case $\mathcal{D}=\mathrm{Pt}$ the one-point category, we identify $F$ and $G$ with their images, so that we write $\mathcal{C}/x$ and $y/\mathcal{C}$ for $x\in\mathcal{C}^\wedge$ and $y\in\mathcal{C}^\vee$ instead $\mathcal{C}/G$ and $F/\mathcal{C}$ respectively.

Finally, we define the end and the coend.
Let $\mathcal{J}$ be a small category and $\mathcal{C}$ be a bicomplete category.
For a set $S$ and an object $X\in\mathcal{C}$, we write
\[
S\times X:=\coprod_{s\in S}X_s\quad,\qquad X^S:=\prod_{s\in S} X_s
\]
where each $X_s$ is a copy of $X$.
Notice that we have an adjunction
\[
\mathcal{C}(S\times X,Y)\simeq \mathcal{C}(X,Y^S).
\]
Then the end and the coend of a functor $F:\mathcal{J}^\opposite\times\mathcal{J}\to\mathcal{C}$ are defined as
\[
\begin{split}
\int_{j\in\mathcal{J}}F(j,j)
&:=\mathrm{eq}\left(\prod_{j\in\mathcal{J}}F(j,j)\rightrightarrows\prod_{j_1,j_2\in\mathcal{J}}F(j_1,j_2)^{\mathcal{J}(j_1,j_2)}\right)\,\text{and}\\
\int^{j\in\mathcal{J}}F(j,j)
&:=\mathrm{coeq}\left(\coprod_{j_1,j_2\in\mathcal{J}}\mathcal{J}(j_2,j_1)\times F(j_1,j_2)\rightrightarrows\coprod_{j\in\mathcal{J}}F(j,j)\right)\,,
\end{split}
\]
where $\mathrm{eq}$ and $\mathrm{coeq}$ denote the equalizer and the coequalizer respectively.
Using the end, we obtain the formula:
\[
\mathcal{C}^\mathcal{J}(F,G)\simeq\int_{j\in\mathcal{J}}\mathcal{C}(F(j),G(j))
\]
The following is a generalization of the Yoneda lemma:

\begin{theo}
Let $\mathcal{J}$ be a small category, and $\mathcal{C}$ be a bicomplete category.
Then each functor $X:\mathcal{J}\to\mathcal{C}$, we have the formulae
\[
\int_{j\in\mathcal{J}} X(j)^{\mathcal{J}(\blankdot,j)}
\simeq X
\simeq \int^{j\in\mathcal{J}} \mathcal{J}(j,\blankdot)\times X(j).
\]
Dually, for each functor $Y:\mathcal{J}^\opposite\to\mathcal{C}$, we have
\[
\int_{j\in\mathcal{J}} X(j)^{\mathcal{J}(j,\blankdot)}
\simeq X
\simeq \int^{j\in\mathcal{J}} \mathcal{J}(\blankdot,j)\times X(j).
\]
\end{theo}

For more details about ends and coends, see \cite{Kel}.

\subsection{Posets and lattices}
A partially ordered set, poset for short, is an example of categories.
In the section, we observe some properties of posets in a categorical viewpoint.

\begin{defin}
A category $P$ is called a preordered set if each hom-set $P(x,y)$ consists of at most one element.
A poset is a preordered set which is skeletal as a category.
\end{defin}

Indeed, if $P$ is a preordered set in the usual sense, we can see it as a category whose objects are elements of $P$, and for $x,y\in P$, the hom-set is given by
\[
P(x,y) =
\begin{cases}
\quad\{\mathrm{pt}\}&\quad(x\le y) \\
\quad\varnothing&\quad\text{otherwise}
\end{cases}\,.
\]
Then reflexivity and transitivity corresponds to the existence of the identities and composability respectively.
Moreover, $P$ satisfies the antisymmetricity if and only if it is skeletal as a category.
The following is obvious:

\begin{lemma}
Let $f:P\to Q$ be a map between preordered sets.
Then it is order-preserving if and only if it extends to a functor.
Consequently, the category $\mathbf{Poset}$ of posets and order-preserving maps is a full-subcategory of the category $\mathbf{Cat}$ of small categories and functors.
\end{lemma}

As a category, we can consider limits and colimits in a poset $P$.
Let $D:\mathcal{I}\to P$ be a small diagram.
Recall that the colimit is defined as a left adjoint of the constant functor.
For each $x\in P$, we have
\[
P^{\mathcal{I}}(D,x) =
\begin{cases}
\quad\{\mathrm{pt}\}&\quad(\forall i\in\mathcal{I}:D(i)\le x)\\
\quad\varnothing&\quad\text{otherwise}
\end{cases}\,.
\]
Thus we obtain
\[
\colim_{i\in\mathcal{I}} D(i) = \sup\left\{D(i)\mid i\in\mathcal{I}\right\}
\]
if either side exists.
Dually, we obtain
\[
\lim_{i\in\mathcal{I}} D(i) = \inf\left\{D(i)\mid i\in\mathcal{I}\right\}\,.
\]
Notice that the initial (resp. terminal) object is precisely the least (greatest) element in the poset.

\begin{defin}
A poset $P$ is called a join-semilattice if it has the supremum for every finite set.
Equivalently, a join-semilattice is a finite cocomplete poset.

Dually, a poset $P$ is called a meet-semilattice if it has the infimum for every finite set.
Equivalently, a meet-semilattice is a finite complete poset.

$P$ is called a lattice if it is finite bicomplete as a category.
\end{defin}

For $x,y\in P$, we write
\[
\begin{gathered}
x\vee y := x\amalg y = \sup\{x,y\}\,\\
x\wedge y := x\times y = \inf\{x,y\}\,
\end{gathered}
\]
if they exist, and call them the join and meet respectively, instead of the coproduct and product.
Note that $x\le y$ if and only if $x\vee y = y$ if and only if $x\wedge y = x$.

\begin{lemma}
An order-preserving map $f:P\to Q$ of posets preserves joins if and only if it preserves finite colimits as a functor.
Dually, $f$ preserves meets if and only if it preserves finite limits.
\end{lemma}

Recall that, in the category theory, many examples of cocontinuous functors, i.e. functors preserve colimits, come from adjunctions.
Let $f_\ast:P\xrightleftarrows{} Q:f^\ast$ be a pair of order-preserving maps of posets.
They form an adjunction if for every $p\in P$ and $q\in Q$, we have
\[
f_\ast f^\ast(q) \le q\quad,\quad\text{and}\quad p\le f^\ast f_\ast(p)\,.
\]
Notice that in that case, the triangle identities are automatic; indeed, the compositions
\[
\begin{gathered}
f_\ast \le f_\ast f^\ast f_\ast \le f_\ast\,,\\
f^\ast \le f^\ast f_\ast f^\ast \le f^\ast\,
\end{gathered}
\]
are the identities.

\begin{defin}
A pair $f_\ast:P\rightleftarrows{} Q:f^\ast$ is called a Galois connection if it forms an adjunction.
\end{defin}

Since every left adjoint functors preserves colimits, we obtain the following:

\begin{corol}\label{cor:galois-conti}
Suppose $f_\ast:P\rightleftarrows{} Q:f^\ast$ is a Galois connection.
Then $f_\ast$ and $f^\ast$ preserves joins and meets respectively.
\end{corol}

Now consider a lattice $L$.
As a well-known consequence of the lattice theory, we have the following:

\begin{prop}
For a lattice $L$, the following are equivalent:
\begin{enumerate}[label={\rm(\roman*)}]
  \item\label{cond:dist-meet} $x\wedge(y\vee z)=(x\wedge y)\vee (x\wedge z)$ for every $x,y,z\in L$.
  \item\label{cond:dist-join} $x\vee(y\wedge z)=(x\vee y)\wedge(x\vee z)$ for every $x,y,z\in L$.
\end{enumerate}
\end{prop}
\begin{proof}
Suppose \ref{cond:dist-meet}.
Then since $x\wedge z\le x\le (x\vee y)$, we have
\[
(x\vee y)\wedge (x\vee z)
= ((x\vee y)\wedge x)\vee ((x\vee y)\wedge z)
= x\vee (x\wedge z)\vee (y\wedge z)
= x\vee(y\wedge z)\,,
\]
which is \ref{cond:dist-join}.

The converse is the dual, so apply the above argument to the lattice $L^\opposite$.
\end{proof}

\begin{defin}
A lattice $L$ is said to be distributive if it satisfies the above equivalent conditions.
Equivalently, $L$ is distributive if and only if cartesian products in $L$ preserve all finite colimits.
\end{defin}

Recall that, in the category theory, a typical situation that cartesian products preserves colimits is that $\mathcal{C}$ is a cartesian closed category.
The cartesian product in a lattice is nothing but the meet, so we have the following definition:

\begin{defin}
A lattice $L$ is called a Heyting algebra if it is cartesian closed as a category; i.e. there is a functor $(\to):L^\opposite\times L\to L$ such that for every $x,y,z\in L$, we have
\[
x\wedge y\le z \iff x\le (y\to z)\,.
\]
\end{defin}

Notice that if $L$ is a Heyting algebra, we have $L(x\wedge y,z)\simeq L(x,y\to z)$, so that the functor $(\blankdot)\wedge y:L\to L$ has a right adjoint $y\to(\blankdot):L\to L$.
By corollary \ref{cor:galois-conti}, every Heyting algebra is a distributive lattice.

In this article, Boolean lattices, which are examples of Heyting algebras, are most important.

\begin{defin}
Let $L$ be a lattice with the greatest element $1$ and the least element $0$.
Then $L$ is said to be Boolean if there is an (anti-)operation $\neg:L^\opposite\to L$ such that
\[
x\wedge\neg x = 0\,,\quad x\vee\neg x = 1\,.
\]
\end{defin}

A Boolean lattice is a Heyting algebra;
indeed, if $L$ is a Boolean lattice, setting
\[
(y\to z) := (\neg y)\vee z\,
\]
gives rise to the right adjoint of meets.
Therefore, every Boolean lattice is distributive.

Notice that if there is an element $w\in L$ such that $w\wedge x=0$ and $w\vee x=1$, then we have
\[
w = w\wedge(x\vee\neg x) = (w\wedge x)\vee (w\wedge\neg x) = w\wedge\neg x\,,
\]
and
\[
w = w\vee(x\wedge\neg x) = (w\vee x)\wedge(w\vee\neg x) = w\vee\neg x\,.
\]
These imply $\neg x\le w\le \neg x$, so that we obtain $w=\neg x$.
Thus, the operation $\neg:L^\opposite\to L$ is uniquely determined if it exists.

The following are easily verified by the definition:

\begin{lemma}
If $L$ is a Boolean lattice and $x,y\in L$, then the following hold:
\begin{enumerate}[label={\rm(\arabic*)}]
  \item $\neg\neg x = x$.
  \item $\neg(x\wedge y)=(\neg x)\vee (\neg y)$.
  \item $\neg(x\vee y) = (\neg x)\wedge (\neg y)$.
\end{enumerate}
\end{lemma}

For finite Boolean lattices, the following result is known:

\begin{prop}
If $S$ is a set, then the powerset $2^S$ forms a Boolean lattice.
Conversely, if $L$ is a finite Boolean lattice, then there is a finite set $S$ such that $L\simeq 2^S$.
Consequently, the cardinality of a finite Boolean lattice is a power of $2$.
\end{prop}

For more details about lattice theory, see for example \cite{DP}.

\section{Thin-powered categories}
\label{sec:thin-pow-cat}
Recall that for a set $S$, we can consider its subset $A\subset S$.
In general, for an object $X$ of a category $\mathcal{C}$, we can consider the class $\mathrm{Sub}(X)$ of isomorphism classes of monomorphisms with codomain $X$ in $\mathcal{C}$.
It is a generalization of the notion of subsets.
$\mathrm{Sub}(X)$ is, however, too large in general to treat combinatorially.
In this section, we consider a more ideal situation.

\subsection{Distinguished injections}
\begin{defin}
Let $\mathcal{R}$ be a small category.
Then a thin-powered structure on $\mathcal{R}$ is a class $\mathcal{R}^+_0$ of morphisms which satisfies the following conditions:
\begin{enumerate}[label={(DI\arabic*)},leftmargin=\widthof{\indent(DI0)}]
\item\label{cond:DImono} $\mathcal{R}^+_0$ contains only monomorphisms.
\item\label{cond:DIcomp} $\mathcal{R}^+_0$ is closed under compositions.
\item\label{cond:DIfact} For each morphism $f:s\to r\in\mathcal{R}$, there is a factorization $f=\delta\sigma$ with $\delta\in\mathcal{R}^+_0$ and $\sigma\pitchfork\mathcal{R}^+_0$, and the factorization is strictly unique.
\end{enumerate}
In the case, we call an element of $\mathcal{R}^+_0$ a distinguished injection.
A small category which is equipped with thin-powered structure is called a thin-powered category.
\end{defin}

Note that if $\mathcal{R}^+_0$ is a thin-powered structure, then it cannot contain non-trivial isomorphisms.
Indeed, if $\pi\in\mathcal{R}^+_0$ is an isomorphism, then clearly we have $\pi\pitchfork\mathcal{R}^+_0$.
Then $\mathrm{id}\pi=\pi\mathrm{id}$ are factorizations, and the uniqueness in the condition \ref{cond:DIfact} implies $\pi=\mathrm{id}$.

On the other hand, $\mathcal{R}^+_0$ contains the all identities.
To see this, factor the identity as $\mathrm{id}=\delta\sigma$ with $\delta\in\mathcal{R}^+_0$ and $\sigma\pitchfork\mathcal{R}^+_0$ by the condition \ref{cond:DIfact}.
This factorization implies that $\delta$ is not only a monomorphism, which follows from the condition \ref{cond:DImono}, but also a split epimorphism.
Hence $\delta$ is an isomorphism, so that by the condition \ref{cond:DImono}, $\delta$ itself needs to be the identity.
That is why we often regard $\mathcal{R}^+_0$ as a wide subcategory of $\mathcal{R}$.

We introduce some notations.
Let $\mathcal{R}$ be a thin-powered category with thin-powered structure $\mathcal{R}^+_0$.
We denote by $\mathcal{R}^-$ the class of morphisms which are left orthogonal to $\mathcal{R}^+_0$ whose elements are called $\mathcal{R}^+_0$-surjections.
Then the condition \ref{cond:DIfact} is equivalent to that there is a $(\mathcal{R}^-,\mathcal{R}^=_0)$-factorization.
One should notice that this is a special case of weak factorization systems.
In particular, we have the following results:

\begin{lemma}\label{lem:thinpow-ortho-sk}
In the above situation, every morphism $f:r'\to r$ with $\mathcal{R}^-\pitchfork f$ factors as $f=\delta\pi$ such that $\pi$ is an isomorpshim and $\delta\in\mathcal{R}^+_0$.
\end{lemma}
\begin{proof}
Let $f=\delta\sigma$ be the factorization with $\delta\in\mathcal{R}^+_0$ and $\sigma\pitchfork\mathcal{R}^+_0$, and consider the following commutative diagram:
\[
\xymatrix{
  r' \ar@{=}[r] \ar[d]_{\sigma} & r' \ar[d]^{f} \\
  s \ar[r]^{\delta} & r }
\]
By the assumption, we have $\sigma\pitchfork f$, so that the square has a lift $\mu:s\to r'$ with $\mu\sigma=\mathrm{id}$ and $f\mu=\delta$.
Since $\delta$ is a monomorphism, these implies that $\mu$ is both a split epimorpihsm and a monomorphism.
Hence $\mu$ is an isomorphism, so by setting $\pi=\mu^{-1}$, we obtain the required factorization $f=\delta\pi$.
\end{proof}

\begin{corol}\label{cor:thinpow-wk-fact}
If $\mathcal{R}$ is a thin-powered category, then there is a weak factorization system $(\mathcal{R}^-,\mathcal{R}^+)$ such that
\begin{itemize}
  \item $\mathcal{R}^-$ is the class of morphisms left orthogonal to all distinguished injections;
  \item Every morphisms in $\mathcal{R}^+$ is isomorphic to a distinguished injections.
\end{itemize}
\end{corol}

Recall that for a class $S$ of morphisms, a morphism $f$ is said to be $S$-projective if $f\pitchfork S$.
Hence we can say that $\mathcal{R}^-$ is the class of $\mathcal{R}^+_0$-projections.
In particular, $f$ is called a strong epimorphism if it is an epimorphism and left orthogonal to all monomorphisms.
Note that if $S'\subset S$, then $S$-projectivity implies $S'$-projectivity.
Consequently, we obtain the following:

\begin{lemma}\label{lem:thinpow-ortho-split}
Let $\mathcal{R}$ be a thin-powered category.
Then every strong epimorphism in $\mathcal{R}$ belongs to $\mathcal{R}^-$.
In particular, $\mathcal{R}^-$ contains all split epimorphisms.
\end{lemma}

As shown in lemma \ref{cor:thinpow-wk-fact}, a thin-powered structure gives rise to an analogy of the notion of images of maps in the set theory.
Indeed, if $\mathcal{R}^+_0$ is a thin-powered structure on a category $\mathcal{R}$, and $f$ is a morphism of $\mathcal{R}$, then we have a unique factorization $f=\delta\sigma$ with $\sigma\pitchfork\mathcal{R}^+_0$ and $\delta\in\mathcal{R}^+_0$.
We write $\coim(f):=\sigma$ and $\image(f):=\delta$.
If we regard the weak factorization system as a generalization of coimage-image factorization in the set theory, then by corollary \ref{cor:thinpow-wk-fact}, $\image(f)$ should be seen as the ``image'' of $f$.

\begin{rem}
Notice that we have another notion of images of morphisms in an arbitrary category.
Recall that for a morphism $f:r'\to r\in\mathcal{R}$, its image, which is sometimes called the regular image, is defined to be the representation of the inclusion of presheaves $\operatorname{Im} f\hookrightarrow\mathcal{R}[r]$, where
\[
(\operatorname{Im} f)(s)
:= \left\{ g:s\to r\mid \forall h_1,h_2:r\rightrightarrows w:h_1 f=h_2 f\Rightarrow h_1 g=h_2 g \right\}\,.
\]
It is known that images are strong monomorphisms, i.e. they are right orthogonal to all epimorphisms, so that if $f$ admits the image, it factors as an epimorpihsm followed by a strong monomorphism.
We can also define the coimages of morphisms in the dual method.
For more details, see \cite{KS}.

Now, our notation of $\image(f)$ and $\coim(f)$ are a bit different from the notion described above.
However, we can see a thin-powered category as a generalized situation, so we use the notation.
\end{rem}

\subsection{Pullbacks and stability}
We next consider the notion of inverse images.
Recall that, in the set theory, they are categorically defined by pullbacks in the category $\mathbf{Set}$.
So as an analogy, we also need pullbacks in our situation.

\begin{defin}
A thin-powered category $\mathcal{R}$ with thin-powered structure $\mathcal{R}^+_0$ is said to be semicomplete if every pair of morphisms $r'\xrightarrow{f} r \xleftarrow{\delta} s$ with $\delta\in\mathcal{R}^+_0$ has a pullback in $\mathcal{R}$; i.e. there is a pullback square
\begin{equation}
\label{eq:thinpow-pb}
\begin{gathered}
\xymatrix{
  s' \ar[r]^{f'} \ar[d]_{\delta'} \ar@{}[dr]|(.4){\pbcorner} & s \ar[d]^{\delta} \\
  r' \ar[r]^{f} & r }
\end{gathered}
\end{equation}
in $\mathcal{R}$.
\end{defin}

We denote by $\mathcal{R}^+$ the class of morphisms right orthogonal to $\mathcal{R}^-$.
By corollary \ref{cor:thinpow-wk-fact}, $(\mathcal{R}^-,\mathcal{R}^+)$ forms a weak factorization system on $\mathcal{R}$.
Hence, the class $\mathcal{R}^+$ is closed under (existing) pullbacks.
Now since $\mathcal{R}^+_0\subset\mathcal{R}^+$, the morphism $\delta'$ in the diagram \eqref{eq:thinpow-pb} belongs to $\mathcal{R}^+$.
Actually, we can take $\delta'$ so that $\delta'\in\mathcal{R}^+_0$.

\begin{lemma}\label{lem:thinpow-pb}
Let $\mathcal{R}$ be a semicomplete thin-powered category with thin-powered structure $\mathcal{R}^+_0$.
Then for every pair of morphisms $r'\xrightarrow{f} r \xleftarrow{\delta} s$ with $\delta\in\mathcal{R}^+_0$, there is a (strictly) unique pullback square
\[
\xymatrix{
  s' \ar[r]^{f'} \ar[d]_{\delta'} \ar@{}[dr]|(.4){\pbcorner} & s \ar[d]^{\delta} \\
  r' \ar[r]^{f} & r }
\]
such that $\delta'\in\mathcal{R}^+_0$.
\end{lemma}
\begin{proof}
The existence follows from the assumption that $\mathcal{R}$ is semicomplete and lemma \ref{lem:thinpow-ortho-sk}.
So we only show the uniqueness.
Suppose we have another pullback square
\[
\xymatrix{
  s'_1 \ar[r]^{f'_1} \ar[d]_{\delta'_1} \ar@{}[dr]|(.4){\pbcorner} & s \ar[d]^{\delta} \\
  r' \ar[r]^{f} & r }
\]
with $\delta'_1\in\mathcal{R}^+_0$.
Then by the uniqueness of pullbacks, there is an isormophism $\theta:s'_1\to s'$ such that $\delta'_1=\delta'\theta$.
It implies $\theta=\mathrm{id}$ and $\delta'_1=\delta'$ by the uniqueness of the factorization in \ref{cond:DIfact}.
Moreover we have $\delta f'_1=\delta f'$.
Since $\delta$ is a monomorphism by \ref{cond:DImono}, we obtain $f'_1=f'$.
Therefore, two pullbacks are strictly equal.
\end{proof}

For $\delta:s\to r\in\mathcal{R}^+_0$ and $f:r'\to r\in\mathcal{R}$, by lemma \ref{lem:thinpow-pb}, there is a unique pullback $\delta'$ of $\delta$ by $f$ such that $\delta'\in\mathcal{R}^+_0$.
We write $f^\ast\delta:=\delta'$, which is well-defined thanks to the uniqueness.
It is straightfoward by the definition that we have $(gf)^\ast\delta=f^\ast g^\ast\delta$.
In what follows, we always assume that such pullbacks are taken in the way.

Lemma \ref{lem:thinpow-pb} also says that every morphism in $\mathcal{R}$ has pullbacks by distinguished injections.
Notice that while $\mathcal{R}^+$ is closed under pullbacks, $\mathcal{R}^-$ is not necessarily closed under pullbacks.
This is a similar case to the categories $\mathbf{Poset}$ or $\mathbf{Top}$, where pullbacks of regular epimorphisms are not necessarily regular epimorphisms.
So we need the notion of stability.

\begin{defin}
Let $\mathcal{R}$ be a semicomplete thin-powered category, and $\mathcal{R}^+_0$ be a thin-powered structure.
Then an $\mathcal{R}^+_0$-surjection $\sigma$ is said to be stable if all its pullbacks by distinguished injections are again $\mathcal{R}^+_0$-surjections.
In general, a morphism $f$ in $\mathcal{R}$ is said to be stable if the $\mathcal{R}^+_0$-surjection $\coim(f)$ is stable.

A thin-powered category is said to be stable if it is semicomplete and all $\mathcal{R}^+_0$-surjections are stable.
\end{defin}

Some morphisms are necessarily stable.
For example, an isomorphism is stable.
More generally, every split epimorpihsm is stable since the class of split epimorphisms is closed under existing pullbacks.
It also follows that all distinguished injections are stable.

Note that if $\mathcal{R}$ is stable, then the $(\mathcal{R}^-,\mathcal{R}^+_0)$-factorization is stable under pullbacks by dinstinguished injections.
This is an analogy of the situation that if $f:S\to T$ is a map of sets and $B\subset T$, then the factorization
\[
f^{-1}(B) \xrightarrow{f} f(S)\cap B \hookrightarrow B
\]
is the image-factorization.

\subsection{Examples}
We here give some examples.
Before this, we prove the following results:

\begin{lemma}\label{lem:thinpow-embed}
Let $\mathcal{R}$ be a thin-powered category with thin-powered structure $\mathcal{R}^+_0$ such that every morphism $f$ with $f\pitchfork\mathcal{R}^+_0$ is a split epimorphism.
Suppose there is an embedding $j:\mathcal{R}\hookrightarrow\mathcal{S}$ of categories such that
\begin{enumerate}[label={\rm(\alph*)}]
  \item $j$ is bijective on objects;
  \item $j$ preserves monomorphisms;
  \item $\mathcal{S}$ admits a weak factorization system $(\mathcal{S}',j\mathcal{R})$.
\end{enumerate}
Then the class $j\mathcal{R}^+_0$ is a thin-powered structure on $\mathcal{S}$.
\end{lemma}
\begin{proof}
We have to verify the conditions.
The conditions \ref{cond:DImono} and \ref{cond:DIcomp} are straightfoward from the assumptions.
So it suffices to verify the condition \ref{cond:DIfact}.

Let $\mathcal{R}^-$ and $\mathcal{S}^-$ denote the class of morphisms in $\mathcal{R}$ and $\mathcal{S}$ which are left orthogonal to $\mathcal{R}^+_0$ and $j\mathcal{R}^+_0$ respectively.
By the assumptions, morphisms in $j\mathcal{R}^+_0$ are monomorphisms while morphisms in $j\mathcal{R}^-$ are split epimorphisms in $\mathcal{S}$, so we have $j\mathcal{R}^-\pitchfork j\mathcal{R}^+_0$ in $\mathcal{S}$.
On the other hand, since $(\mathcal{S}',j\mathcal{R})$ is a weak factorization system, we have $\mathcal{S}'\pitchfork j\mathcal{R}$, and hence $\mathcal{S}'\pitchfork j\mathcal{R}^+_0$.
These imply that $\mathcal{S}',j\mathcal{R}^-\subset\mathcal{S}^-$.

Next, suppose $f$ is a morphism in $\mathcal{S}$.
By the assumption, there is a factorization $f=j(\mu)\tau$ with $\tau\in\mathcal{S}'$ and $\mu\in\mathcal{R}$.
Moreover, since $\mathcal{R}^+_0$ is a thin-powered structure on $\mathcal{R}$, there is a factorization $\mu=\delta\sigma$ with $\sigma\in\mathcal{R}^-$ and $\delta\in\mathcal{R}^+_0$.
We now have $f=j(\delta) j(\sigma)\tau$.
The factorization is actually unique; indeed, if $f=j(\delta') j(\sigma')\tau'$ is another factorization, we have a commutative square:
\[
\xymatrix{
  s \ar[r]^{j(\sigma)\tau} \ar[d]_{j(\sigma')\tau'} & s' \ar[d]^{j(\delta)} \\
  s'' \ar[r]^{j(\delta')} & r }
\]
Since $j(\sigma)\tau,j(\sigma')\tau'\in\mathcal{S}^-$ and $\mathcal{S}^-\pitchfork j\mathcal{R}^+_0$, the square has lifts $\theta:s''\to s'$ and $\theta':s'\to s''$.
$j(\delta)$ and $j(\delta')$ are monoomrphisms, hence $\theta$ and $\theta'$ are inverse to each other.
Then $\theta\in\mathcal{S}'\cap j\mathcal{R}$, and we may assume $\theta=j(\bar\theta)$.
We have $\delta=\bar\theta\delta'$, and the uniqueness of the factorization in \ref{cond:DIfact} implies $\theta=\mathrm{id}$.
Thus we obtain $j(\sigma)\tau=j(\sigma')\tau'$ and $\delta=\delta'$ as required.
The condition \ref{cond:DIfact} now follows.
\end{proof}

We also have another version.
The proof is similar to that of lemma \ref{lem:thinpow-embed}.

\begin{lemma}\label{lem:ufact-thinpow}
Suppose $\mathcal{R}$ is a small category such that it has no non-trivial isomorphism, and there is a class $\mathcal{R}^-$ of morphisms in $\mathcal{R}$ such that every morphism $f$ of $\mathcal{R}$ uniquely factors as $f=\delta\sigma$ with $\sigma\in\mathcal{R}^-$ and $\delta$ monic.
Then the class $\mathcal{R}^+_0$ of monomorphisms in $\mathcal{R}$ is a thin-powered structure on $\mathcal{R}$.
\end{lemma}

\begin{exam}
The most trivial example is given as follows:
Let $\mathcal{R}$ be an arbitrary small category, and consider $\objof{\mathcal{R}}$ as the discrete wide subcategory of $\mathcal{R}$.
Then $\objof{\mathcal{R}}$ is a thin-powered structure on $\mathcal{R}$.
It is clear that the structure is semicomplete and stable.
\end{exam}

\begin{exam}
If $\mathcal{R}^+_0$ is a thin-powered structure on $\mathcal{R}$, then $\mathcal{R}^+_0$ itself is, as a category, a thin-powered category with thin-powered structure $\mathcal{R}^+_0$.
\end{exam}

\begin{exam}
Let $\kappa$ be a regular cardinal.
We denote by $\mathbf{Ord}^\kappa$ the full subcategory of $\mathbf{Set}$ spanned by ordinals smaller than $\kappa$.
Let $(\mathbf{Ord}^\kappa)^+_{0}$ be the class of all order-preserving injections in $\mathbf{Ord}^\kappa$.
Then $(\mathbf{Ord}^\kappa)^+_0$ is a thin-powered structure of $\mathbf{Ord}^\kappa$.
Indeed, \ref{cond:DImono} and \ref{cond:DIcomp} are straightforward by the definition.
For a map $f:\lambda_1\to\lambda_2$ between cardinals, its set-theoretical image $f(\lambda_1)\subset\lambda_2$ is well-ordered as a subset.
Since every well-ordered set is isomorphic to a unique ordinal by a unique order-preserving bijection, we can choose a unique ordinal $\lambda_f$ with $\lambda_f\simeq f(\lambda_1)$.
Now we obtain a unique factorization $\lambda_1\to\lambda_f\hookrightarrow\lambda_2$ as required in \ref{cond:DIfact}.
In fact, $\mathbf{Ord}^\kappa$ is a stable thin-powered category.

In particular, the case $\kappa=\omega$, the smallest infinite cardinal, is important.
We write $\mathbf{FinOrd}:=\mathbf{Ord}^\omega$ and $\widetilde\Delta^+:=(\mathbf{Ord}^\omega)^+_0$.
\end{exam}

\begin{exam}
We denote by $\widetilde\Delta$ the category of all (possibly empty) finite ordinals and order-preserving maps.
Then $\widetilde\Delta$ satisfies the assumption in lemma \ref{lem:ufact-thinpow}.
Hence if we denote by $\widetilde\Delta^+$ the class of monomorphisms in $\widetilde\Delta$, $\widetilde\Delta$ is a thin-powered category with thin-powered structure $\widetilde\Delta^+$.
This example is clearly stable.
\end{exam}

\begin{exam}
We define the category $\mathcal{P}$ as a subcategory of $\mathbf{Poset}$ whose objects are posets of the form $(\underline{n},\preceq)$ such that
\begin{itemize}
  \item $\underline{n}$ is a finite ordinal $\{0<1<\dots<n-1\}$;
  \item $\preceq$ is an ordering on $\underline{n}$ of the form
\[
\{0\prec 1 \prec \dots \prec i_1-1, i_1\prec i_1+1\prec\dots\prec i_2-1,\dots,i_k\prec\dots\prec n-1\}\,;
\]
\end{itemize}
and whose morphisms are order-preserving maps $f:(\underline{m},\preceq)\to(\underline{n},\preceq)$ which reflect comparability; i.e. $f(k)\preceq f(l)$ implies either $k\preceq l$ or $k\succeq l$.
We also define a wide subcategory $\mathcal{P}_\Delta$ of $\mathcal{P}$ whose morphisms are maps $f:(\underline{m},\preceq)\to(\underline{n},\preceq)$ which also preserves the natural orderings of ordinals.

Note that every morphism of $\mathcal{P}$ factors as a permutation of components followed by a morphism in $\mathcal{P}_\Delta$, while every morphism of $\mathcal{P}_\Delta$ factors as a split epimorphism followed by a monomorphism.
We denote by $\mathcal{P}^+_\Delta$ the class of monomorphisms in $\mathcal{P}_\Delta$.
Then $\mathcal{P}^+_\Delta$ is a stable thin-powered structure on both $\mathcal{P}_\Delta$ and $\mathcal{P}$.
Note also that $\widetilde\Delta$ is a full subcategory of $\mathcal{P}_\Delta$ (and $\mathcal{P}$).
We have a forgetful functor $\mathcal{P}_\Delta\to\widetilde\Delta$, which is a retraction of the inclusion $\widetilde\Delta\hookrightarrow\mathcal{P}_\Delta$.

We have another description for these categories.
For a tuple $(k_1,\dots,k_n)$ of positive integers, we define a poset
\[
[\![k_1,\dots,k_n]\!]
:= \left\{1<\dots<k_1, k_1+1<\dots<k_1+k_2, \dots, \sum_{i=1}^{n-1}k_i + 1<\dots<\sum_{i=1}^n k_i\right\}\,.
\]
Hence, as a poset, we have an isomorphism
\[
[\![k_1,\dots,k_n]\!]\simeq [\![k_1]\!]\amalg\dots\amalg[\![k_n]\!]\,.
\]
We denote by $[\![0]\!]$ the empty poset.
Then these posets are precisely the objects of $\mathcal{P}$ and $\mathcal{P}_\Delta$.

Each morphism of $\mathcal{P}_\Delta$ uniquely factors as a morphism of the form
\[
[\![\mu_1,\dots,\mu_n]\!]:[\![k_1,\dots,k_n]\!]
\simeq [\![k_1]\!]\amalg\dots\amalg[\![k_n]\!]
\xrightarrow{\mu_1\amalg\dots\amalg\mu_n} [\![l_1]\!]\amalg\dots\amalg[\![l_n]\!]
\simeq [\![l_1,\dots,l_n]\!]
\]
for morphisms $\mu_i:[\![k_i]\!]\to[\![l_i]\!]$ in $\widetilde\Delta$, followed by a morphism of the form
\[
\delta^m_j:[\![l_1,\dots,l_n]\!] \hookrightarrow [\![l_1,\dots,l_j,m,l_{j+1},\dots,l_n]\!]
\]
for $0\le j\le n$ and a positive integer $m$.
Summarizingly, every morphism of $\mathcal{P}_\Delta$ is uniquely written in the form
\[
\delta^m_j\circ[\![\delta_1,\dots,\delta_n]\!]\circ[\![\sigma_1,\dots,\sigma_1]\!]
\]
such that $\delta_1,\dots,\delta_n$ are monomorphisms and $\sigma_1,\dots,\sigma_n$ are split epimorphisms in $\widetilde\Delta$.

As for a morphism in $\mathcal{P}$, it is a composition of the above morphisms and
\[
\pi_\theta:[\![k_1,\dots,k_n]\!]\to [\![k_{\theta(1)},\dots,k_{\theta(n)}]\!]
\]
for a permutation $\theta\in\Sigma_n$.
Hence every morphism of $\mathcal{P}$ is uniquely written in the fom
\[
\delta^m_j\circ[\![\delta_1,\dots,\delta_n]\!]\circ[\![\sigma_1,\dots,\sigma_n]\!]\circ\pi_\theta\,.
\]
\end{exam}

We will also give more examples in the section \ref{sec:crsgrp}.

\section{Poset representations of thin-powered categories}
\label{sec:pos-repn}
In this section, we see that a thin-powered structure behaves as a skeleton of ``subsets'' with some good properties.
Most of the properties observed in the section are analogous to those in the ordianlly set-theory, and they are described categorically and combinatorially.

\subsection{Standard poset representation}
First we recall the following general lemma:

\begin{lemma}\label{lem:sliceposet}
Let $\mathcal{C}$ be a small category and suppose every morphism of $\mathcal{C}$ is monic.
Then for each $c\in\mathcal{C}$, the category $\mathcal{C}/c$ is a preordered set with maximum.
Moreover, if $\mathcal{C}$ is skeletal, then $\mathcal{C}/c$ is a poset for each $c\in\mathcal{C}$.
\end{lemma}
\begin{proof}
Suppose $f:v\to c$ and $g:w\to c$ are morphisms in $\mathcal{C}$, and suppose we have a diagram
\[
\renewcommand{\labelstyle}{\textstyle}
\xymatrix@C=2em{ v \ar@<.35ex>[rr]^{\alpha} \ar@<-.35ex>[rr]_{\beta} \ar[dr]_f & {} & w \ar[dl]^g \\ {} & c & {} }
\]
in $\mathcal{C}$ with $g\alpha=g\beta=f$.
Since $g$ is monic, it implies $\alpha=\beta$.
Hence the hom-set $(\mathcal{C}/c)(f,g)$ consists of at most one element.
This shows that $\mathcal{C}/c$ is a preordered set.
It is obvious that the identity $\mathrm{id}_c$ is the maximum of $\mathcal{C}/c$.

Now suppose $\mathcal{C}$ is skeletal; i.e. if $c,c'\in\mathcal{C}$ are isomorphic, then $c=c'$.
In particular this implies that any isomorphism in $\mathcal{C}$ is an endomorphism which is, by the first part, the identity morphism.
Thus $\mathcal{C}$ has no non-trivial isomorphism, and the result follows.
\end{proof}

\begin{corol}\label{cor:r/poset}
If $\mathcal{R}^+_0$ is a thin-powered structure on $\mathcal{R}$, then for each $r\in\mathcal{R}$, the category $\mathcal{R}^+_0/r$ is a poset.
\end{corol}

By corollary \ref{cor:r/poset}, we have an assignment $r\mapsto\mathcal{R}^+_0/r\in\mathbf{Poset}$, where $\mathbf{Poset}$ denotes the category of posets and order-preserving maps.
On the other hand, for a morphism $f:r\to r'\in\mathcal{R}$, we can define a map $f_{\ast}:\mathcal{R}^+_0/r\to\mathcal{R}^+_0/r'$ by setting
\[
f_{\ast}(\delta) := \image(f\delta) \in \mathcal{R}^+_0/r'
\]
for $\delta\in\mathcal{R}^+_0/r$.
By the uniqueness of the factorization, $f_{\ast}$ is well-defined as a map.

\begin{prop}\label{prop:ordpres}
Suppose $\mathcal{R}$ is a thin-powered category.
Then for every morphism $f:r\to r'$ of $\mathcal{R}$, the map $f_{\ast}:\mathcal{R}^+_0/r\to\mathcal{R}^+_0/r'$ defined above preserves orderings.
\end{prop}
\begin{proof}
Suppose we have the following commutative diagram in $\mathcal{R}^+_0$:
\[
\xymatrix@C=2em {
  s_1 \ar[rr]^{\varepsilon} \ar[dr]_{\delta_1} & {} & s_2 \ar[dl]^{\delta_2}\\
  {} & r & {} }
\]
In other words, we have $\delta_1\le\delta_2$ in the poset $\mathcal{R}^+_0/r$.
Put $\sigma_2:=\coim(f\delta_2)$ and $\sigma_1:=\coim(\sigma_2\varepsilon)$, hence we have $f\delta_2=f_\ast(\delta_2)\sigma_2$ and $\sigma_2\varepsilon=\sigma_{2\ast}(\varepsilon)\sigma_1$.
Then we have
\[
f\delta_1
= f\delta_2\varepsilon
= f_\ast(\delta_2)\sigma_2\varepsilon
= f_\ast(\delta_2)\sigma_{2\ast}(\varepsilon)\sigma_1,
\]
so that $f_\ast(\delta_1)=f_\ast(\delta_2)\sigma_{2\ast}(\varepsilon)$.
Now we obtain the following commutative diagram in $\mathcal{R}^+_0$:
\[
\xymatrix@C=2em {
  s'_1 \ar[rr]^{\sigma_{2\ast}(\varepsilon)} \ar[dr]_{f_{\ast}(\delta_1)} & {} & s'_2 \ar[dl]^{f_{\ast}(\delta_2)}\\
  {} & r' & {} }
\]
This implies that $f_{\ast}(\delta_1)\le f_{\ast}(\delta_2)$ in $\mathcal{R}^+_0/r'$, which is the required result.
\end{proof}

Moreover, we have the following:

\begin{prop}
Let $\mathcal{R}$ be a thin-powered category.
Then the assignment $r\mapsto\mathcal{R}^+_0/r$ gives rise to a functor $\mathcal{R}\to\mathbf{Poset}$.
\end{prop}
\begin{proof}
By proposition \ref{prop:ordpres}, it suffices to verify that the assignment $f\mapsto f_\ast$ is functorial.
For a composition $gf$ in $\mathcal{R}$ and $\delta\in\mathcal{R}^+_0\!/r$, we have
\[
gf\delta = g\image(f\delta)\coim(f\delta) = \image(g\image(f\delta))\coim(g\image(f\delta))\coim(f\delta)\,,
\]
which implies that
\[
(gf)_\ast(\delta)=g_\ast f_\ast(\delta)\,.
\]
Thus $f\mapsto f_\ast$ is functorial, and the result follows.
\end{proof}

In other words, a thin-powered structure induces a representation of $\mathcal{R}$ in the category $\mathbf{Poset}$.

Next, we are interested in the semicompleteness.
Using proposition \ref{prop:ordpres}, we obtain a combinatorial paraphrase of the semicompleteness:

\begin{lemma}\label{lem:pbposet}
Let $\mathcal{R}$ be a thin-powered category.
Suppose we have a commutative square
\begin{equation}
\label{eq:pbposet}
\begin{gathered}
\xymatrix{
  s \ar[r]^{\delta} \ar[d]_{\bar{f}} & r \ar[d]^f \\
  s' \ar[r]^{\delta'} & r'}
\end{gathered}
\end{equation}
in $\mathcal{R}$ with $\delta,\delta'\in\mathcal{R}^+_0$.
Then it is a pullback in $\mathcal{R}$ if and only if $\delta$ is greatest among $\gamma\in\mathcal{R}^+_0/r$ with $f_{\ast}(\gamma)\le\delta'$; i.e. $f_{\ast}(\gamma)\le\delta'$ implies $\gamma\le\delta$.
\end{lemma}
\begin{proof}
First suppose the square is a pullback.
If $\gamma\in\mathcal{R}^+_0/r$ satisfies $f_{\ast}(\gamma)\le\delta'$, say $f\gamma=(f_{\ast}(\gamma))\sigma$ with $\sigma=\coim{f\gamma}$, then there is a distinguished injection $\varepsilon'$ with $f_{\ast}(\gamma)=\delta'\varepsilon'$, and we have the following diagram:
\[
\xymatrix{
  s_1 \ar[r]^{\gamma} \ar[d]_{\varepsilon'\sigma} & r \ar[d]^{f} \\
  s' \ar[r]^{\delta'} & r' }
\]
Since \eqref{eq:pbposet} is a pullback, there is a morphism $\varepsilon:s_1\to s$ with $\gamma=\delta\varepsilon$, which implies $\gamma\le\delta$.

Conversely, suppowe $\delta$ is greatest among $\gamma\in\mathcal{R}^+_0/r$ with $f_{\ast}(\gamma)\le\delta'$, and suppose we have a diagram:
\[
\xymatrix{
  s_1 \ar[r]^{g} \ar[d]_{h} & r \ar[d]^{f} \\
  s' \ar[r]^{\delta'} & r' }
\]
We have $f_{\ast}(\mathrm{im}(g))=\mathrm{im}(fg)=\mathrm{im}(\delta'h)\le\delta'$, so that the assumption implies $\mathrm{im}(g)\le\delta$.
Hence $g=\delta g_1$ for some morphism $g_1:s_1\to s$.
We have $\delta'\bar{f}g_1=f\delta g_1=fg=\delta' h$, and since $\delta'$ is a monomorphism, we obtain $\bar{f}g_1=h$.
Therefore the square \eqref{eq:pbposet} is a pullback.
\end{proof}

\begin{corol}\label{cor:pos-semicomp}
A thin-powered category $\mathcal{R}$ is semicomplete if and only if it satisfies the following condition: 
For every morphism $f:r\to r'$ of $\mathcal{R}$ and every $\delta'\in\mathcal{R}^+_0/r'$, the subset $f_{\ast}^{-1}(\ordideal\delta')\subset\mathcal{R}^+_0/r$ has the greatest element, where $\ordideal\delta'$ is the lower subset of $\mathcal{R}^+_0/r'$ generated by $\delta'$.
\end{corol}
\begin{proof}
If $\mathcal{R}$ is semicomplete, the condition immediately follows from lemma \ref{lem:pbposet}.
Conversely, suppose $\mathcal{R}$ satisfies the condition, and let $f:r\to r'$ be a morphism and $\delta'\in\mathcal{R}^+_0/r'$.
Let $\delta$ be the greatest element of $f_{\ast}^{-1}(\Lambda^{\delta'})\subset\mathcal{R}^+_0/r$.
Since $f_{\ast}(\delta)\le\delta'$, we have a square
\[
\xymatrix{
  s \ar[r]^{\delta} \ar[d] & r \ar[d]^{f} \\
  s' \ar[r]^{\delta'} & r' }
\]
Notice that the condition on $\delta$ is equivalent to that $\delta$ is greatest among $\gamma\in\mathcal{R}^+_0/r$ with $f_{\ast}(\gamma)\le\delta'$.
Thus the square is a pullback by lemma \ref{lem:pbposet}.
Of course $\delta\in\mathcal{R}^+_0$, so $\mathcal{R}$ is semicomplete.
\end{proof}

Now suppose $\mathcal{R}$ is a semicomplete thin-powered category.
Then because of the pullbacks, $\mathcal{R}^+_0\!/r$ is not only a poset, but also a meet-semilattice.
For $\delta_1,\delta_2\in\mathcal{R}^+_0\!/r$, take a pullback square:
\[
\xymatrix{
  t \ar[r]^{\varepsilon_1} \ar[d]_{\varepsilon_2} \ar@{}[dr]|(.4){\pbcorner} & s_1 \ar[d]^{\delta_1} \\
  s_2 \ar[r]^{\delta_2} & r }
\]
Then the meet is given by the formula
\[
\delta_1\wedge\delta_2 = \delta_1\varepsilon_1 = \delta_2\varepsilon_2.
\]
In general, for a morphism $f:r\to r'$ and $\delta'\in\mathcal{R}^+_0\!/r'$, we define $f^\ast\delta'\in\mathcal{R}^+_0\!/r$ by the pullback square:
\[
\xymatrix{
  s \ar[r]^{f^\ast\delta'} \ar[d] \ar@{}[dr]|(.4){\pbcorner} & r \ar[d]^{f} \\
  s' \ar[r]^{\delta'} & r' }
\]
By lemma \ref{lem:thinpow-pb}, $f^\ast\delta'$ is uniquely determined, and the resulting $f^\ast:\mathcal{R}^+_0\!/r'\to\mathcal{R}^+_0\!/r$ is a well-defined map.
Clearly the map $f^\ast$ preserves meets, so that it is order-preserving.
Moreover, we have the following result:

\begin{lemma}\label{lem:galois}
Let $\mathcal{R}$ be a semicomplete thin-powered category.
Then for a morphism $f:r\to r'\in\mathcal{R}$, the pair
\[
f_\ast:\mathcal{R}^+_0\!/r\rightleftarrows\mathcal{R}^+_0\!/r':f^\ast
\]
of order-preserving maps forms a Galois connection.
\end{lemma}
\begin{proof}
We have to verify the inequalities $\delta\le f^\ast f_\ast\delta$ and $f_\ast f^\ast\delta'\le\delta'$ for every $\delta\in\mathcal{R}^+_0\!/r$ and $\delta'\in\mathcal{R}^+_0\!/r'$.
Since $f^\ast$ is defined by the pullback square, the inequality $\delta\le f^\ast f_\ast\delta$ can be easily verified.
So consider the following pullback square:
\[
\xymatrix{
  s \ar[r]^{f^\ast\delta} \ar[d]_{f_1} & r \ar[d]^{f} \\
  s' \ar[r]^{\delta} & r' }
\]
Then we have $f_\ast f^\ast\delta=\delta \image(f')\le \delta$ as required.
\end{proof}

\begin{corol}\label{cor:supremum}
In the situation above, for each morphism $f:r\to r'\in\mathcal{R}$, the map $f_\ast:\mathcal{R}^+_0/r \to \mathcal{R}^+_0/r'$ preserves the existing supremums of subsets of $\mathcal{R}^+_0/r$.
In particular, $f_\ast$ preserves the least element (if it exists).
\end{corol}

As special cases of Galois connections, we have the following:

\begin{lemma}\label{lem:thinpow-galois}
Let $\mathcal{R}$ be a semicomplete thin-powered category.
Then if $\delta:s\to r\in\mathcal{R}^+_0$, we have $\delta^\ast\delta_\ast=\mathrm{id}$ on $\mathcal{R}^+_0/s$.
\end{lemma}
\begin{proof}
Let $\gamma\in\mathcal{R}^+_0/s$, hence since $\delta\gamma\in\mathcal{R}^+_0$, we have $\delta_\ast(\gamma)=\delta\gamma$.
Consider the following commutative diagram:
\[
\xymatrix{
  t \ar[r]^{\gamma} \ar@{=}[d] \ar@{}[dr]|(.4){\pbcorner} & s \ar@{=}[r] \ar@{=}[d] \ar@{}[dr]|(.4){\pbcorner} & s \ar[d]^{\delta} \\
  t \ar[r]^{\gamma} & s \ar[r]^{\delta} & r }
\]
Each square is a pullback, so the outer square is also a pullback.
Then we obtain $\delta^\ast\delta_\ast(\gamma)=\delta^\ast(\delta\gamma)=\gamma$.
This is the required result.
\end{proof}

\begin{corol}\label{cor:thinpow-meets}
Let $\mathcal{R}$ be a semicomplete thin-powered category.
Then if $\delta:s\to r\in\mathcal{R}^+_0$, the map $\delta_\ast:\mathcal{R}^+_0\!/s\to\mathcal{R}^+_0\!/r$ preserves meets.
\end{corol}
\begin{proof}
Let $\gamma_1,\gamma_2\in\mathcal{R}^+_0\!/s$.
Then by lemma \ref{lem:thinpow-galois}, we have
\[
\delta_\ast(\gamma_1\wedge\gamma_2)
= \delta_\ast(\gamma_1)_\ast(\gamma_1)^\ast\gamma_2
= \delta_\ast(\gamma_1)_\ast(\gamma_1)^\ast\delta^\ast\delta_\ast\gamma_2
= \delta_\ast\gamma_1\wedge\delta_\ast\gamma_2\,.
\]
\end{proof}

\begin{defin}
A thin-powered category $\mathcal{R}$ is said to be latticed if it is semicomplete and each $\mathcal{R}^+_0/r$ is a lattice.
Moreover, $\mathcal{R}^+_0$ is said to be Boolean if it is semicomplete and each $\mathcal{R}^+_0/r$ is a Boolean lattice.
\end{defin}

By corollary \ref{cor:supremum}, we have the following:

\begin{lemma}\label{lem:semilat-repn}
If $\mathcal{R}$ is a latticed thin-powered category, then for each morphism $f:r\to r'\in\mathcal{R}$, the induced map $f_\ast:\mathcal{R}^+_0/r\to\mathcal{R}^+_0/r'$ preserves joins and the least element.
In particular, the poset representation factors through $\mathcal{R}^+_0/(\blankdot):\mathcal{R}\to\mathbf{SemiLat}^\vee$, where $\mathbf{SemiLat}^\vee$ denotes the category of join-semilattices and join-preserving maps.
\end{lemma}

\subsection{Stability of $\mathcal{R}^+_0$-surjections}
Next, we observe the stability condition.

\begin{lemma}\label{lem:pbstable}
Let $\mathcal{R}$ be a semicomplete thin-powered category.
Then for every $\mathcal{R}^+_0$-surjection $\sigma:r\to r'$, the following conditions are equivalent:
\begin{enumerate}[label={\rm (\arabic*)}]
  \item\label{cond:pbstable} $\sigma$ is stable.
  \item\label{cond:rsurjsplit} For every $\gamma'\in\mathcal{R}^+_0\!/r'$, we have $\sigma_\ast\sigma^\ast(\gamma')=\gamma'$.
  \item\label{cond:rsurjinj} The map $\sigma^\ast:\mathcal{R}^+_0\!/r'\to\mathcal{R}^+_0\!/r$ is injective.
\end{enumerate}
\end{lemma}
\begin{proof}
\ref{cond:pbstable}$\Rightarrow$\ref{cond:rsurjsplit}: Suppose $\sigma$ is stable, and let $\gamma'\in\mathcal{R}^+_0\!/r'$.
Take the pullback, and we obtain the square:
\[
\xymatrix{
  s \ar[d]_{\mu} \ar[r]^{\sigma^\ast(\gamma')} \ar@{}[dr]|(.4){\pbcorner} & r \ar[d]^{\sigma} \\
  s' \ar[r]^{\gamma'} & r' }
\]
The condition \ref{cond:pbstable} implies that $\mu$ is $\mathcal{R}^+_0$-surjective.
We have $\sigma\sigma^\ast(\gamma')=\gamma'\mu$, and by the uniqueness of the factorization, we obtain $\sigma_\ast\sigma^\ast(\gamma')=\gamma'$.
Thus the condition \ref{cond:rsurjsplit} follows.

\ref{cond:rsurjsplit}$\Rightarrow$\ref{cond:rsurjinj}: Obvious.

\ref{cond:rsurjinj}$\Rightarrow$\ref{cond:pbstable}: Suppose $\sigma^\ast$ is injective.
Suppose we have a pullback square
\[
\xymatrix{
  s \ar[d]_{\mu} \ar[r]^{\sigma^\ast(\gamma')} \ar@{}[dr]|(.4){\pbcorner} & r \ar[d]^{\sigma} \\
  s' \ar[r]^{\gamma'} & r' }
\]
with $\gamma'$ being a distinguished injection.
Then we have the following diagram:
\[
\xymatrix{
  s \ar@{=}[r] \ar[d]_{\coim(\mu)} & s \ar[d]_{\mu} \ar[r]^{\sigma^\ast(\gamma')} \ar@{}[dr]|(.4){\pbcorner} & r \ar[d]^{\sigma} \\
  s'' \ar[r]^{\image(\mu)} & s' \ar[r]^{\gamma'} & r' }
\]
Since $\image(\mu)$ is a monomorphism, the outer square is a pullback.
This implies that $\sigma^\ast(\gamma'\image(\mu))=\sigma^\ast(\gamma')$, and since $\sigma^\ast$ is injective by the assumption, we obtain $\gamma'\image(\mu)=\gamma'$.
Now $\gamma'$ is a monomorphism, so that $\image(\mu)=1$ and $\mu=\tau$ is $\mathcal{R}^+_0$-surjective.
\end{proof}

\begin{corol}\label{cor:ffmeet}
Let $\mathcal{R}$ be a semicomplete thin-powered category, and suppose $f:r\to r'\in\mathcal{R}$ is stable.
Then for each $\gamma'\in\mathcal{R}^+_0\!/r'$, we have $f_\ast f^\ast(\gamma')=\gamma'\wedge\image(f)$.
\end{corol}
\begin{proof}
Let $\gamma'\in\mathcal{R}^+_0\!/r'$, and consider the following pullback square:
\[
\xymatrix{
  s \ar[r]^{\tau} \ar[d]_{f^\ast(\gamma')} \ar@{}[dr]|(.4){\pbcorner} & s'_0 \ar[r]^{\varepsilon} \ar[d] \ar@{}[dr]|(.4){\pbcorner} & s' \ar[d]^{\gamma'} \\
  r \ar[r]^{\coim(f)} & s'_1 \ar[r]^{\mathrm{im}(f)} & r' }
\]
Since $\coim(f)$ is $\mathcal{R}^+_0$-surjective, by the definition of stablity, $\tau$ is $\mathcal{R}^+_0$-surjective.
It follows that $f_\ast f^\ast(\gamma')=\gamma'\varepsilon=\gamma'\wedge\image(f)$.
This is the required result.
\end{proof}

If $\mathcal{R}$ is a stable thin-powered category, by lemma \ref{lem:pbstable}, every $\mathcal{R}^+_0$-surjection $\sigma:r\to r'$ induces a split epimorphism $\sigma_\ast:\mathcal{R}^+_0/r\to\mathcal{R}^+_0/r'$.

\subsection{Locally finiteness}
\label{subsec:loc-fin}
We consider the following situation:

\begin{defin}
A thin-powered category $\mathcal{R}$ is said to be locally finite if each poset $\mathcal{R}^+_0\!/r$ is finite.
\end{defin}

Note that if a meet-semilattice is finite and have a maximum element, then it is also a join-semilattice with unit by setting:
\[
x\vee y = \bigwedge_{w\ge x, y} w
\]
In particular, it is a complete lattice.
By lemma \ref{lem:sliceposet}, each $\mathcal{R}^+_0/r$ has a maximum element.
Hence if $\mathcal{R}$ is semicomplete and locally finite, then each $\mathcal{R}^+_0/r$ is a complete lattice, so that $\mathcal{R}$ is latticed.

Let $\mathcal{R}$ be a semicomplete locally finite thin-powered category.
We define a degree function $\deg:\mathcal{R}\to\mathbb{N}$ by
\[
\deg_{\mathcal{R}^+_0}(r):= |\mathcal{R}^+_0\!/r|,
\]
where $|\blankdot|$ denotes the cardinality of sets.
Then by lemma \ref{lem:thinpow-galois} and lemma \ref{lem:pbstable}-\ref{cond:rsurjsplit}, the following hold:
\begin{itemize}
  \item Every isomorphism preserves the degree.
  \item Every non-trivial $\mathcal{R}^+_0$-morphism raises the degree; in other words, $\deg$ induces a functor $\deg:\mathcal{R}^+_0\to\mathbb{N}$ which reflects the identities.
  \item Every stable $\mathcal{R}^+_0$-surjection does not raises the degree.
\end{itemize}

\subsection{Saturated subobjects}
We next consider a saturated subobject, which is an analogy of the notion of saturated subsets in the set theory:

\begin{defin}
Let $\mathcal{R}$ be a semicomplete thin-powered category, and let $f:r\to r'$ be a morphism of $\mathcal{R}$.
Then a morphism $\delta\in\mathcal{R}^+_0\!/r$ is said to be saturated with respect to $f$ if the square
\[
\xymatrix{
  s \ar[r]^{\delta} \ar[d] & r \ar[d]^{f} \\
  s' \ar[r]^{f_{\ast}(\delta)} & r' }
\]
is a pullback in $\mathcal{R}$.
Equivalently, we say $\delta$ is saturated with respect to $f$ if $\delta=f^{\ast}f_{\ast}(\delta)$.
\end{defin}

Some similar results in the set-theory hold for stable morphisms.

\begin{lemma}\label{lem:stabsatcrit}
Let $\mathcal{R}$ be a semicomplete thin-powered category.
Then the following hold:
\begin{enumerate}[label={\rm (\arabic*)}]
  \item\label{sublem:injsat} If $\delta:r\to r'\in\mathcal{R}^+_0$, every element of $\mathcal{R}^+_0\!/r$ is saturated with respect to $\delta$.
  \item\label{sublem:rightsat} Suppose $f=\delta\mu:r\to r'$ is a morphism of $\mathcal{R}$ with $\delta\in\mathcal{R}^+_0$.
Then $\gamma\in\mathcal{R}^+_0\!/r$ is saturated with respect to $f$ if and only if it is saturated with respect to $\mu$.
In particular, $\gamma$ is saturated with respect to $f$ if and only if it is saturated with respect to $\coim(f)$.
  \item\label{sublem:stabsat} Suppose $f:r\to r'\in\mathcal{R}$ is stable.
Then an element of $\mathcal{R}^+_0\!/r$ is saturated with respect to $f$ if and only if it is of the form $f^\ast(\gamma')$ for some $\gamma'\in\mathcal{R}^+_0\!/r'$.
\end{enumerate}
\end{lemma}
\begin{proof}
The part \ref{sublem:injsat} immediately follows from lemma \ref{lem:thinpow-galois}.
Then, if $\delta$ is a distinguished injection, then we have $(\delta\mu)^\ast(\delta\mu)_\ast=\mu^\ast\delta^\ast\delta_\ast\mu_\ast=\mu^\ast\mu_\ast$.
Hence we obtain \ref{sublem:rightsat}.

Finally we show \ref{sublem:stabsat}.
By part \ref{sublem:rightsat}, it suffices to show \ref{sublem:stabsat} only for a stable $\mathcal{R}^+_0$-surjection $\sigma:r\to r'$.
If $\gamma\in\mathcal{R}^+_0\!/r$ is saturated with respect to $\sigma$, then we have $\gamma=\sigma^\ast\sigma_\ast(\gamma)$, which is the required form.
Conversely, by lemma \ref{lem:pbstable}, we have $\sigma_\ast\sigma^\ast=\mathrm{id}$, so that $\sigma^\ast\sigma_\ast\sigma^\ast=\sigma^\ast$.
Thus $\sigma^\ast(\gamma')$ is saturated with respect to $\sigma$ for every $\gamma'\in\mathcal{R}^+_0\!/r'$.
\end{proof}

\begin{lemma}\label{lem:satwedge}
Let $\mathcal{R}$ be a stable latticed thin-powered category, and suppose $\delta\in\mathcal{R}^+_0\!/r$ is saturated with respect to a morphism $f:r\to r'\in\mathcal{R}$.
Then for every $\gamma\in\mathcal{R}^+_0\!/r$, we have $f_\ast(\delta)\wedge f_\ast(\gamma)=f_\ast(\delta\wedge\gamma)$.
\end{lemma}
\begin{proof}
If $f\in\mathcal{R}^+_0$, the equation is obviuos.
So factorizing $f$, we may assume $f$ is $\mathcal{R}^+_0$-surjective.
Consider the diagram
\[
\xymatrix@-2.5ex{
  {} & s \ar[rr]^{\delta} \ar[dd]|{\hole} && r \ar[dd]^{f} \\
  u \ar[rr] \ar[ur]^{\varepsilon} \ar[dd]_{\tau} && t \ar[ur]^{\gamma} \ar[dd]^<<(.2){\sigma} & {} \\
  {} & s' \ar[rr]|{\hole}_<<(.1){f_\ast\delta} && r' \\
  u' \ar[rr] \ar[ur]^{\varepsilon'} && t' \ar[ur]_{f_\ast\gamma} & {} }
\]
where $\sigma=\coim(f\gamma)$ and the top and bottom faces are pullbacks, so that we have $\delta\wedge\gamma=\delta\varepsilon$ and $f_\ast(\delta)\wedge f_\ast(\gamma)=f_\ast(\delta)\varepsilon'$.
Since $\delta$ is saturated with respect to $f$, the back face is a pullback.
It follows that the front face is also a pullback.
Since $\sigma$ is $\mathcal{R}^+_0$-surjective and $\mathcal{R}^+_0$ is stable, $\tau$ is $\mathcal{R}^+_0$-surjective.
Now we have
\[
f(\delta\wedge\gamma)=f\delta\varepsilon=(f_\ast\delta)\varepsilon'\tau=(f_\ast\delta\wedge f_\ast\gamma)\tau\,.
\]
Then we obtain $f_\ast(\delta\wedge\gamma)=\image f(\delta\wedge\gamma)=f_\ast\delta\wedge f_\ast\gamma$ as required.
\end{proof}

Next we observe relations to joins.
Recall that in the set theory, if $S_1,S_2\subset A$ are saturated subsets with respect to a map $f:A\to B$, their union $S_1\cup S_2$ is also a saturated subset.
However, even if $\delta_1,\delta_2\in\mathcal{R}^+_0/r$ is saturated with respect to $\sigma:r\to r'\in\mathcal{R}$, their join $\delta_1\vee\delta_2$ need not be saturated in general.
But we have the following criterion.

\begin{lemma}\label{lem:distjoin}
Let $\mathcal{R}$ be a latticed thin-powered category.
Then for $\delta:s\to r\in\mathcal{R}^+_0$, the map $\delta^{\ast}:\mathcal{R}^+_0\!/r\to\mathcal{R}^+_0\!/s$ preserves joins if and only if for each $\gamma_1,\gamma_2\in\mathcal{R}^+_0\!/r$, we have the identity $(\gamma_1\vee \gamma_2)\wedge\delta=(\gamma_1\wedge\delta)\vee(\gamma_2\wedge\delta)$.
\end{lemma}
\begin{proof}
By the definition of meets, we have $\gamma\wedge\delta=\delta(\delta^\ast\gamma)=\delta_{\ast}\delta^{\ast}(\gamma)$ for each $\gamma\in\mathcal{R}^+_0\!/r$.
Hence if $\delta^\ast$ preserves joins, we obtain the required identity.
Conversely, the identity implies $\delta_\ast\delta^\ast(\gamma_1\vee\gamma_2)=\delta_\ast\delta^\ast(\gamma_1)\vee\delta_\ast\delta^\ast(\gamma_2)=\delta_\ast(\delta^\ast(\gamma_1)\vee\delta^\ast(\gamma_2)$.
Since $\delta$ is a distinguished injection, $\delta_\ast:\mathcal{R}^+_0\!/s\to\mathcal{R}^+_0\!/r$ is injective.
Thus we obtain $\delta^\ast(\gamma_1\vee\gamma_2)=\delta^\ast(\gamma_1)\vee\delta^\ast(\gamma_2)$ for each $\gamma_1,\gamma_2\in\mathcal{R}^+_0\!/r$, which implies that $\delta^\ast$ preserves joins.
\end{proof}

\begin{corol}\label{cor:Booljoin}
Let $\mathcal{R}$ be a Boolean thin-powered category.
Then for every $\delta:s\to r$, the induced map $\delta^\ast:\mathcal{R}^+_0/r\to\mathcal{R}^+_0/s$ preserves joins.
\end{corol}
\begin{proof}
By lemma \ref{lem:distjoin}, it suffices to show that each $\mathcal{R}^+_0/r$ is a distributive lattice.
Since we assumed $\mathcal{R}^+_0/r$ is a Boolean lattice, and every Boolean lattice is distributive, the result follows.
\end{proof}

\begin{lemma}\label{lem:vpres}
Let $\mathcal{R}$ be a latticed thin-powered category.
Suppose $\sigma:r\to r'$ is a stable $\mathcal{R}^+_0$-surjection.
Then the following are equivalent:
\begin{enumerate}[label=\rm{(\alph*)}]
  \item\label{cond:vpres} $\sigma^{\ast}:\mathcal{R}^+_0\!/r'\to\mathcal{R}^+_0\!/r$ preserves joins.
  \item\label{cond:vsatur} Whenever $\delta_1,\delta_2\in\mathcal{R}^+_0\!/r$ are saturated with respect to $\sigma$, so is $\delta_1\vee\delta_2$.
\end{enumerate}
\end{lemma}
\begin{proof}
Since $\sigma$ is an stable $\mathcal{R}^+_0$-surjection, the map $\sigma^{\ast}:\mathcal{R}^+_0\!/r'\to\mathcal{R}^+_0\!/r$ is a full-embedding of posets by lemma \ref{lem:pbstable}; i.e. $\sigma^\ast$ is injective and reflects the ordereings.
Moreover, by lemma \ref{lem:stabsatcrit}, its image coincides with the subposet consisting of elements of $\mathcal{R}^+_0\!/r$ saturated with respect to $\sigma$.
Hence the result follows.
\end{proof}

\begin{defin}
Let $\mathcal{R}$ be a stable latticed thin-powered category.
A morphism $f$ of $\mathcal{R}$ is said to be coherent if the order-preserving map $f^\ast$ preserves joins and the least element.
$\mathcal{R}$ is said to be coherent in $\mathcal{R}$ if it is stable and latticed, and every morphism in $\mathcal{R}$ is coherent.
\end{defin}

\begin{lemma}\label{lem:regcrit}
Let $\mathcal{R}$ be a stable latticed thin-powered category.
Then the following hold:
\begin{enumerate}[label={\rm(\arabic*)}]
  \item\label{sublem:isomreg} If $f:r'\to r$ induces an isomorphism $f_\ast:\mathcal{R}^+_0/r'\to\mathcal{R}^+_0/r$, then $f$ is coherent.
In particular, every isomorphism is coherent.
  \item\label{sublem:compreg} If $f:r'\to r$ and $g:r\to r''$ are coherent, then so is $gf$.
\end{enumerate}
\end{lemma}
\begin{proof}
Let $f:r'\to r\in\mathcal{R}$ be a morphism such that $f_\ast:\mathcal{R}^+_0/r'\to \mathcal{R}^+_0/r$ is an isomorhism of posets.
Then since $f^\ast$ is a right adjoint to $f_\ast$, it follows that $f^\ast$ is also an isomorphism of posets.
Hence $f^\ast$ preserves all supremums, and this is \ref{sublem:isomreg}.

\ref{sublem:compreg} is obvious by the functoriality.
\end{proof}

Finally, we note that saturated subobjects have always the ``maximum'' under certaing assumptions.

\begin{lemma}\label{lem:maxsatur}
Let $\mathcal{R}$ be a stable Boolean thin-powered category.
Suppose $f:r\to r'\in\mathcal{R}$ is coherent.
Then for every $\delta\in\mathcal{R}^+_0\!/r$, the following hold:
\begin{itemize}
  \item $\delta_0:=\neg f^\ast f_\ast(\neg\delta)$ is saturated with respect to $f$.
  \item $\delta_0\le\delta$.
  \item If $\gamma\in\mathcal{R}^+_0\!/r$ is saturated with respect to $f$ and $\gamma\le\delta$, then $\gamma\le\delta_0$.
\end{itemize}
In other words, $\delta_0$ is the maximum saturated subobject among $\delta_0\le\delta$.
\end{lemma}
\begin{proof}
Notice that $\neg x=\bigwedge_{w\vee x=1}w$ in a Boolean lattice.
Hence by lemma \ref{lem:vpres}, we have $\delta_0=f^\ast(\neg f_\ast(\neg\delta))$, which implies $\delta_0$ is saturated by \ref{sublem:stabsat} in lemma \ref{lem:stabsatcrit}.
We also have $\delta_0=\neg f^\ast f_\ast(\neg\delta)\le \neg\neg\delta=\delta$ because $\mathrm{id}\le f^\ast f_\ast$.

Suppose $\gamma\in\mathcal{R}^+_0\!/r$ is saturated with respect to $f$ and $\gamma\le\delta$.
Then by lemma \ref{lem:satwedge}, we have
\[
\gamma\wedge f^\ast f_\ast(\neg\delta)
= f^\ast f_\ast(\gamma)\wedge f^\ast f_\ast(\neg\delta)
= f^\ast f_\ast(\gamma\wedge\neg\delta)
= f^\ast f_\ast(0)
= 0
\]
since $f$ is coherent.
This implies that $\gamma\le\neg f^\ast f_\ast(\neg\delta)=\delta_0$, which is the required result.
\end{proof}

\section{Crossed groups on thin-powered categories}
\label{sec:crsgrp}

\subsection{Crossed groups}
Crossed groups are defined by \cite{FL} and \cite{Kra} in the simplicial case.
Berger and Moerdijk used crossed groups to make examples of generalized Reedy categories in \cite{BM}.
We here recall the definition.

\begin{defin}
Let $\mathcal{R}$ be a small category.
Then a crossed $\mathcal{R}$-group is a presheaf $G:\mathcal{R}^{\mathrm{op}}\to\mathbf{Set}$ together with structures
\begin{itemize}
  \item a group structure on each $G(r)$ for $r\in\mathcal{R}$;
  \item a left action $G(r)\times\mathcal{R}(s,r)\to\mathcal{R}(s,r)$ for each $r,s\in\mathcal{R}$;
\end{itemize}
which satisfies the following conditions:
For $x,y\in G(r)$, $f:s\to r$ and $g:t\to s$ of $\mathcal{R}$, we have
\begin{gather}
\tag{CG1}\label{cond:CGgdist} x\cdot(fg) = (x\cdot f)(f^{\ast}x\cdot g), \\
\tag{CG2}\label{cond:CGcdist} f^{\ast}(xy) = ((y\cdot f)^{\ast}x)(f^{\ast}y), \\
\tag{CG3}\label{cond:CGcunit} x\cdot1 = 1, \\
\tag{CG4}\label{cond:CGgunit} f^{\ast}1 =1.
\end{gather}
\end{defin}

We can describe the conditions \eqref{cond:CGgdist} and \eqref{cond:CGcdist} more categorically:
If $G$ is a presheaf on $\mathcal{R}$, and each $G(r)$ have a left action on $\mathcal{R}(s,r)$, then we can define a map
\[
\mathrm{crs}:G(r)\times\mathcal{R}(s,r)\ni(f,x)\mapsto (x\cdot f,f^\ast(x))\in\mathcal{R}(s,r)\times G(s).
\]
Then the conditions \eqref{cond:CGgdist} and \eqref{cond:CGcdist} are equivalent to the commutativity of the diagrams
\begin{equation}\label{eq:CGgcomm}
\begin{gathered}
\xymatrix{
  G(r)\times\mathcal{R}(s,r)\times\mathcal{R}(t,s) \ar[r]^-{1\times\mathrm{comp}} \ar[d]_{\mathrm{crs}\times1} & G(r)\times\mathcal{R}(t,r) \ar[dd]^{\mathrm{crs}} \\
  \mathcal{R}(s,r)\times G(s)\times\mathcal{R}(t,s) \ar[d]_{1\times\mathrm{crs}} & {} \\
  \mathcal{R}(s,r)\times\mathcal{R}(t,s)\times G(t) \ar[r]^-{\mathrm{comp}\times1} & \mathcal{R}(t,r)\times G(t) }
\end{gathered}
\end{equation}
and
\begin{equation}\label{eq:CGccomm}
\begin{gathered}
\xymatrix{
  G(r)\times G(r)\times\mathcal{R}(s,r) \ar[r]^-{\mathrm{mult}\times 1} \ar[d]_{1\times\mathrm{crs}} & G(r)\times\mathcal{R}(s,r) \ar[dd]^{\mathrm{crs}} \\
  G(r)\times\mathcal{R}(s,r)\times G(s) \ar[d]_{\mathrm{crs}\times1} & {} \\
  \mathcal{R}(s,r)\times G(s)\times G(s) \ar[r]^-{1\times\mathrm{mult}} & \mathcal{R}(s,r)\times G(s) }
\end{gathered}
\end{equation}
respectively.

A crossed $\mathcal{R}$-group defines a new category $\mathcal{R}\!G$ called the total category.
An object of $\mathcal{R}G$ is that of $\mathcal{R}$, and its morphism set is given by
\[
\mathcal{R}G(s,r):=\mathcal{R}(s,r)\times G(s).
\]
The composition map is given as the following:
\[
\begin{gathered}
\begin{multlined}
\mathcal{R}G(s,r)\times\mathcal{R}G(t,s) = \mathcal{R}(s,r)\times G(s)\times\mathcal{R}(t,s)\times G(t) \\
\xrightarrow{1\times\mathrm{crs}\times1} \mathcal{R}(s,r)\times\mathcal{R}(t,s)\times G(t)\times G(t) \\
\xrightarrow{\mathrm{comp}\times\mathrm{mult}} \mathcal{R}(t,r)\times G(t) = \mathcal{R}G(t,r)
\end{multlined}
\end{gathered}
\]
Explicitly we have $(g,y)\circ(f,x) = (g\circ(y\cdot f),(f^{\ast}y)\cdot x)$.
The associativity of the composition follows from the diagrams \eqref{eq:CGgcomm} and \eqref{eq:CGccomm}.
The identity morphism on an object $r\in\mathcal{R}G$ is $(1_r,1_r)$.
Note that $\mathcal{R}$ is naturally seen as a wide subcategory of $\mathcal{R}G$ by
\[
\mathcal{R}(s,r)\ni f\hookrightarrow (f,1)\in\mathcal{R}G(s,r).
\]
Moreover, each $G(r)$ is contained in the automorphism group $\mathrm{Aut}_{\mathcal{R}G}(r)$ by
\[
G(r)\ni x\hookrightarrow (1,x)\in\mathrm{Aut}_{\mathcal{R}G}(r).
\]
That is why we see a morphism $f$ of $\mathcal{R}$ or an element $x\in G(r)$ as a morphism of $\mathcal{R}G$.
Hence every morphism of $\mathcal{R}G$ is of the form $f x$ with unique $f$ of $\mathcal{R}$ and $x\in G(r)$.

The category $\mathcal{R}G$ is an example of a thickening of a category (\cite{Cis}).
We here give a bit stronger definition by Isaacson \cite{Isa}:

\begin{defin}
Let $\mathcal{S}$ be a category and $\mathcal{R}$ be a wide subcategory of $\mathcal{S}$.
Then $\mathcal{S}$ is said to be a thickening of $\mathcal{R}$ if the composition map induces an isomorphism
\[
\mathcal{R}(s,r)\times\mathrm{Aut}_{\mathcal{S}}(s)\overset{\sim}{\to}\mathcal{S}(s,r).
\]
\end{defin}

Equivalently $\mathcal{S}$ is a thickening of $\mathcal{R}$ if each morphism $s\to r$ of $\mathcal{S}$ uniquely factors as an automorphism $\pi\in\mathrm{Aut}_{\mathcal{S}}(s)$ followed by a morphism $f:s\to r$ in $\mathcal{R}$.
Thanks to the uniqueness of the factorization, thickenings preserve some classes of morphisms:

\begin{lemma}\label{lem:thickmono}
Let $\mathcal{R}$ be a small category and $\mathcal{S}$ be a thickening of $\mathcal{R}$.
Then the embedding $\mathcal{R}\hookrightarrow\mathcal{S}$ preserves monomorphisms.
Consequently, a morphism $s\to r$ of $\mathcal{S}$ is a monomorphism if and only if it is of the form $f\pi$ with $f$ being a monomorphism of $\mathcal{R}$ and $\pi\in\mathrm{Aut}_{\mathcal{S}}(s)$.
\end{lemma}
\begin{proof}
Suppose $f:s\to r$ is a monomorphism in $\mathcal{R}$, and we have $g_1\rho_1, g_1\rho_2:t\to s$ of $\mathcal{S}$ such that $f g_1\rho_1=f g_2\rho_2$.
Then by the uniqueness of the factorization, we obtain $fg_1=fg_2$ and $\rho_1=\rho_2$.
Since $f$ is a monomorphism in $\mathcal{R}$, the former implies $g_1=g_2$.
Hence we obtain $g_1\rho_1=g_2\rho_2$.
This implies the first part of the lemma.
The second part is now obvious.
\end{proof}

\subsection{Compatibility with thin-powered structures}
We next consider a situation that a thickening derives a thin-powered structure.

Suppose $\mathcal{R}\hookrightarrow\mathcal{S}$ is a thickening.
For each $s,r\in\mathcal{S}$, we can consider the following composition:
\[
\mathrm{Aut}_{\mathcal{S}}(s)\times\mathcal{R}(r,s)
\xrightarrow{\mathrm{comp}} \mathcal{S}(r,s)
\xleftarrow{\substack{\mathrm{comp}\cr\sim}} \mathcal{R}(r,s)\times\mathrm{Aut}_{\mathcal{S}}(r)
\xrightarrow{\mathrm{proj}} \mathcal{R}(r,s)
\]
Since $\mathcal{R}$ is seen to be a subcategory of $\mathcal{S}$, the map gives rise to an action of $\mathrm{Aut}_{\mathcal{S}}(s)$ on the set $\mathcal{R}(r,s)$.

\begin{defin}
Let $\mathcal{R}^+_0$ be a thin-powered structure on $\mathcal{R}$.
A thickening $\mathcal{S}$ of $\mathcal{R}$ is said to be compatible with the thin-powered structure if the action $\mathrm{Aut}_{\mathcal{S}}(s)$ on $\mathcal{R}(r,s)$ preserves subsets $\mathcal{R}^+_0(r,s)$ and $\mathcal{R}^-(r,s)$.
A crossed $\mathcal{R}$-group $G$ is said to be compatible with the thin-powered structure if the thickening $\mathcal{R}\hookrightarrow\mathcal{R}\!G$ is compatible.
\end{defin}

\begin{prop}
Let $\mathcal{R}$ be a thin-powered category with thin-powered structure $\mathcal{R}^+_0$.
If $\mathcal{R}\hookrightarrow\mathcal{S}$ is a thickening compatible with $\mathcal{R}^+_0$, then $\mathcal{R}^+_0$ is also a thin-powered structure on $\mathcal{S}$.
In the case, moreover, if $\mathcal{R}$ is semicomplete (resp. stable, coherent), then so is $\mathcal{S}$.
\end{prop}
\begin{proof}
We first show the first statement.
The condition \ref{cond:DImono} follows from lemma \ref{lem:thickmono}, and \ref{cond:DIcomp} is obvious.
So we verify the condition \ref{cond:DIfact}.

We denote by $\mathcal{S}^-$ the class of morphisms in $\mathcal{S}$ of the form $\sigma\theta:s\to r$ with $\sigma\in\mathcal{R}^-$ and $\theta\in\mathrm{Aut}_{\mathcal{S}}(s)$.
Since the thikening is compatible with $\mathcal{R}^+_0$, $\mathcal{S}^-$ is closed under compositions.
Moreover, it is obvious that every morphism $f:s\to r\in\mathcal{S}$ uniquely factors as $f=\delta\sigma\theta$ with $\theta\in\mathrm{Aut}_{\mathcal{S}}(s)$, $\sigma\in\mathcal{R}^-$ and $\delta\in\mathcal{R}^+_0$.
Now, suppose we have a square
\[
\xymatrix{
  s' \ar[r]^{f} \ar[d]_{\sigma\theta} & s \ar[d]^{\delta} \\
  r' \ar[r]^{g} & r }
\]
in $\mathcal{S}$ with $\theta\in\mathrm{Aut}_{\mathcal{S}}(s')$, $\sigma\in\mathcal{R}^-$ and $\delta\in\mathcal{R}^+_0$.
Factor $f=\beta\rho\varphi$ and $g=\gamma\tau\psi$ as above.
Then the uniqueness of the factorization implies that $\delta\beta=\gamma$.
Since $\delta$ is a monomorphism in $\mathcal{S}$ by lemma \ref{lem:thickmono}, the morphism $\beta\tau\psi$ is actually a lift for the square.
Thus, we obtain \ref{cond:DIfact}.

Finally, the last part follows from corollary \ref{cor:pos-semicomp}.
\end{proof}

\begin{corol}\label{cor:crs-grp-thinpow}
Let $\mathcal{R}$ be a thin-powered category with thin-powered structure $\mathcal{R}^+_0$, and let $G$ be a crossed $\mathcal{R}$-group.
If $G$ is compatible with $\mathcal{R}^+_0$, then $\mathcal{R}^+_0$ is also a thin-powered structure on $\mathcal{R}\!G$.
In the case, if $\mathcal{R}$ is semicomplete (resp. stable, coherent), then so is $\mathcal{R}\!G$.
\end{corol}

\begin{exam}
Let $\Sigma=(\Sigma_n)_{n\ge0}$ is the sequence of the symmetric groups (we assume $\Sigma_0=\left\{\mathrm{pt}\right\}$).
For a map $\mu:m\to n$ of $\widetilde\Delta$, we define $\mu^\ast:\Sigma_n\to \Sigma_m$ by setting $\mu^\ast(\sigma)$ to be the permutation which moves the elements in $\mu^{-1}\left\{i\right\}$ into $\mu^{-1}\left\{\sigma(i)\right\}$ with preserving orderings in these sets.
Moreover, each $\Sigma_n$ naturally acts on $\widetilde\Delta(m,n)$.
It is easy to verify that $\Sigma$ and $\widetilde\Delta$ satisfy the conditions \eqref{cond:CGgdist} $\sim$ \eqref{cond:CGgunit}, so that $\Sigma$ is a crossed $\widetilde\Delta$-group.
Clearly it is compatible with the thin-powered structure $\widetilde\Delta^+$ on $\widetilde\Delta$, hence $\widetilde\Delta^+$ is a thin-powered structure on $\widetilde\Delta\Sigma$.
Thus $\widetilde\Delta\Sigma$ is a coherent Boolean thin-powered category.
\end{exam}

\begin{exam}\label{ex:group-operad}
Some typical examples of crossed $\widetilde\Delta$-group come from group operads.
We know the sequence $\Sigma=(\Sigma_n)_{n\ge 0}$ of symmetric groups form a non-symmetric operad.
The notion of group operads is a generalization of it, which appears in \cite{Wah} and an axiomatic definition is given in \cite{Zha}.
A group operad is a ($\mathbf{Set}$-valued) non-symmetric operad $G=(G_n)_{n\ge 0}$ together with an operad map $\pi:G\to\Sigma$ such that
\begin{enumerate}[label={\rm(GO\arabic*)},leftmargin=\widthof{\indent(GO0)}]
  \item\label{cond:GOhomo} each $G_n$ is a group and $\pi_n:G_n\to\Sigma_n$ is a group homomorphism;
  \item\label{cond:GOunit} the unit $1\in G_1$ is the identity of the operad $G$;
  \item\label{cond:GOmult} for each $a,a'\in G_k$ and $b_i,b'_i\in G_{m_i}$, we have
\[
\gamma(aa';b_1b'_1,\dots,b_kb'_k) = \gamma(a;b_{\pi(a')(1)},\dots,b_{\pi(a')(k)})\gamma(a';b'_1,\dots,b'_n)
\]
where $\gamma$ denotes the composition in the operad $G$.
\end{enumerate}
We denote by $(k)\in G_k$ the unit of the group.
Then for a map $\mu:m\to n$ of $\widetilde\Delta$, we can define a map $\mu^{\ast}:G_n\to G_m$ by setting
\[
\mu^{\ast}(a) = \gamma\left(a;(\#\mu^{-1}\left\{1\right\}),\dots,(\#\mu^{-1}\left\{n\right\})\right).
\]
Moreover, each $G_n$ acts on $\widetilde\Delta(m,n)$ via $\pi_n:G_n\to\Sigma_n$.
Then $G$ is a crossed $\widetilde\Delta$-group as well.
By lemma \ref{lem:ufact-thinpow}, $G$ is compatible with the thin-powered structure $\widetilde\Delta^+$ on $\widetilde\Delta$.
Thus by corollary \ref{cor:crs-grp-thinpow}, $\widetilde\Delta^+$ is a coherent Boolean thin-powered structure on $\widetilde\Delta G$.

For example, a sequence $\mathcal{B}=(\mathcal{B}_n)_{n\ge0}$ of braid groups or $\mathcal{P}=(\mathcal{P}_n)$ of pure braid groups defines a crossed $\widetilde\Delta$-group.
\end{exam}

\begin{exam}
Let $G$ be a crossed $\widetilde\Delta$-group.
Then we naturally regard $G$ as a crossed $\mathcal{P}$-group or $\mathcal{P}_\Delta$-group as follows:
For $[\![k_1,\dots,k_n]\!]$, we define
\[
G([\![k_1,\dots,k_n]\!]) := G(k_1)\times G(k_n)\,.
\]
For morphisms, we set
\[
\begin{gathered}
[\![\mu_1,\dots,\mu_n]\!]^\ast:=\mu_1^\ast\times\dots\times\mu_n^\ast:G(k_1)\times\dots\times G(k_n)\to G(l_1)\times\dots\times G(l_n)\,,\\
(\delta^m_j)^\ast:G(l_1)\times\dots\times G(l_j)\times G(m)\times G(l_{j+1})\times\dots\times G(l_n)\xrightarrow{\mathrm{proj}} G(l_1)\times\dots\times G(l_n)\,,\\
\pi_{\theta}^\ast:G(k_1)\times\dots\times G(k_n)\to G(k_{\theta^{-1}(1)})\times\dots\times G(k_{\theta^{-1}(n)})\,.
\end{gathered}
\]
Hence we have an extension $G:\mathcal{P}^\opposite\to\mathbf{Set}$.
Moreover, the crossed $G$-action on $\mathcal{P}$ also extends in the obvious way.
\end{exam}

\section{Confluent degeneracy systems}
\label{sec:conflu}

Recall that in the simplex category $\Delta$, if $0\le i<j\le n$, the square
\[
\xymatrix{
  {[n+1]} \ar[r]^{\sigma_i} \ar[d]_{\sigma_j} & {[n]} \ar[d]^{\sigma_{j-1}} \\
  {[n]} \ar[r]^{\sigma_i} & {[n-1]} }
\]
is an absolute pushout square.
As a result, for every simplicial set $X$, each cell $x:\Delta[n]\to X$ factors through a degeneracy map $\Delta[n]\twoheadrightarrow\Delta[n_0]$ followed by a non-degenerate cell $x_0:\Delta[n_0]\to X$.
In this section, we observe this phenomenon from an axiomatic viewpoint.

\subsection{Definition}
\begin{defin}
Let $\mathcal{A}$ be a small category.
A confluent degeneracy system on $\mathcal{A}$ is a wide subcategory $\mathcal{A}^-\subset\mathcal{A}$ together with a functor $\deg:\mathcal{A}^-\to\mathbb{N}^\opposite$, called the degree function, such that
\begin{enumerate}[label={\rm(CDS\arabic*)},leftmargin=\widthof{\indent(CDS0)}]
  \item\label{cond:CDSisom} every isomorphism preserves the degree and belongs to $\mathcal{A}^-$;
  \item\label{cond:CDSorth} if $\delta$ is a monomorphism of $\mathcal{A}$ such that $\mathcal{A}^-\pitchfork\delta$, then $\delta$ is either an isomorphism or a degree-raising morphism;
  \item\label{cond:CDSfact} every morphism $f$ of $\mathcal{A}$ factors as $f=\delta\sigma$ such that $\sigma\in\mathcal{A}^-$ and $\delta$ is a monomorphism such that $\mathcal{A}^-\pitchfork\delta$;
  \item\label{cond:CDSconf} for any diagram $a_1\xleftarrow{\sigma_1}a_0\xrightarrow{\sigma_2}a_2$ with $\sigma_1,\sigma_2\in\mathcal{A}^-$, there is an absolute pushout square
\[
\xymatrix{
  a_0 \ar[r]^{\sigma_1} \ar[d]_{\sigma_2} \ar@{}[dr]|(.6){\pocorner} & a_1 \ar[d]^{\tau_1} \\
  a_2 \ar[r]^{\tau_2} & a }
\]
in $\mathcal{A}$ with $\tau_1,\tau_2\in\mathcal{A}^-$.
\end{enumerate}
We call an element of $\mathcal{A}^-$ a degeneracy morphism.
A morphism $f:X\to Y\in\mathcal{A}^\wedge$ is said to be non-degenerate if $\mathcal{A}^-\pitchfork f$ holds.
\end{defin}

\begin{rem}
The terminology comes from the theory of term rewriting systems.
\end{rem}

\begin{exam}\label{exam:EZ-cat}
A small category $\mathcal{A}$ is called an Eilenberg-Zilber category (EZ category briefly) if there is a function $\deg:\objof\mathcal{A}\to\mathbb{N}$ which satisfies the following:
\begin{enumerate}[label={\rm(EZ\arabic*)},leftmargin=\widthof{\indent(EZ0)}]
  \item\label{cond:EZdegr} Every monomorphism raises the degree.
  \item\label{cond:EZfact} Every morphism $f$ factors as $f=\delta\sigma$ with $\sigma$ a split epimorphism and $\delta$ a monomorphism.
  \item\label{cond:EZconf} For any diagram $a_1\xleftarrow{\sigma_1}a_0\xrightarrow{\sigma_2}a_2$ such that $\sigma_1,\sigma_2$ are split epimorphisms, there is an absolute pushout square
\[
\xymatrix{
  a_0 \ar[r]^{\sigma_1} \ar[d]_{\sigma_2} \ar@{}[dr]|(.6){\pocorner} & a_1 \ar[d]^{\tau_1} \\
  a_2 \ar[r]^{\tau_2} & a }
\]
in $\mathcal{A}$ such that $\tau_1,\tau_2$ are split epimorphisms.
\end{enumerate}
The notion of EZ categories was introduced by Berger and Moerdijk in \cite{BM}.
The conditions above was originally pointed out for the simplicial category $\Delta$ by Eilenberg and Zilber in \cite{EZ}.
So EZ categories are generalizations of $\Delta$ in this sense.
Since split epimorphisms are left orthogonal to monomorphisms in any category, it is clear that a small category $\mathcal{A}$ is an EZ category if and only if the wide subcategory $\mathcal{A}^-$ of split epimorphisms form a confluent degeneracy system on $\mathcal{A}$.
\end{exam}

Some examples come from thin-powered categories.
For this, we need an additional assumption.

\begin{defin}
A locally finite stable thin-powered category $\mathcal{R}$ is said to be confluent if it satisfies the following condition:
If $r_1\xleftarrow{\sigma_1}r_0\xrightarrow{\sigma_2}r_2$ are $\mathcal{R}^+_0$-surjections, then there is an absolute pushout square
\[
\xymatrix{
  r_0 \ar[r]^{\sigma_1} \ar[d]_{\sigma_2} \ar@{}[dr]|(.6){\pocorner} & r_1 \ar[d]^{\tau_1} \\
  r_2 \ar[r]^{\tau_2} & r }
\]
such that $\tau_1,\tau_2$ are again $\mathcal{R}^+_0$-surjections.
\end{defin}

Now fix a locally finite stable confluent thin-powered category $\mathcal{R}$.
We denote by $\mathcal{R}^-$ the class of $\mathcal{R}^+_0$-surjections.
Since $\mathcal{R}^+_0$ is locally finite, so as discussed in the section \ref{subsec:loc-fin}, we have a degree function $\deg:\objof{\mathcal{R}}\to\mathbb{N}$.
Then $\mathcal{R}^-$ is a confluent degeneracy system on $\mathcal{R}$ with the degree function.
Indeed, the condition \ref{cond:CDSisom} is obvious, and the conditions \ref{cond:CDSorth} and \ref{cond:CDSfact} follow from corollary \ref{cor:thinpow-wk-fact}.
Finally the condition \ref{cond:CDSconf} directly follows form the definition of the confluence of thin-powered structure.

\subsection{Skeletons}
Let $\mathcal{A}^-$ be a confluent degeneracy system on $\mathcal{A}$.
We denote by $\mathcal{A}_{\le n}$ the full subcategory of $\mathcal{A}$ spanned by objects of degree $\le n$.
Let $j_n:\mathcal{A}_{\le n}\hookrightarrow\mathcal{A}$ be the embedding.
By the left Kan extension, we obtain an adjuction pair:
\[
(j_n)_!:\mathcal{A}^\wedge_{\le n} \xrightleftarrows{} \mathcal{A}^\wedge:j_n^\ast
\]
We write $\operatorname{sk}_n:=(j_n)_!j_n^\ast:\mathcal{A}^\wedge\to\mathcal{A}^\wedge$.
The counit of the adjunction gives rise to a natural transformation $\operatorname{sk}_nX\to X$ for $X\in\mathcal{A}^\wedge$.

Note that using the coend, for $X'\in\mathcal{A}^\wedge_{\le n}$, we can write the left Kan extension $(j_n)_!X'\in\mathcal{A}^\wedge$ as follows:
\[
(j_n)_!X'
\simeq \int^{a\in\mathcal{A}_{\le n}} \mathcal{A}(\blankdot,j_n(a))\times X'(a)
\]
Since $j_n:\mathcal{A}_{\le n}\hookrightarrow\mathcal{A}$ is fully faithful, we have
\[
\begin{multlined}
j_n^\ast(j_n)_!X'
\simeq \int^{a\in\mathcal{A}_{\le n}} \mathcal{A}(j_n(\blankdot),j_n(a))\times X'(a) \\
\simeq \int^{a\in\mathcal{A}_{\le n}} \mathcal{A}_n(\blankdot,a)\times X'(a)
\,\simeq\, X'\,,
\end{multlined}
\]
here the last isormophism is induced by the $\mathcal{A}_n$-action
\[
\mathcal{A}_n(\blankdot,a)\times X'(a)\to X'(\blankdot),
\]
which in fact induces an isomorphism by co-Yoneda lemma.
It follows that the unit $X'\to j_n^\ast(j_n)_!X'$ of the adjunction is an isomorphism.
In particular, if $a\in\mathcal{A}$ is an object of degree $\le n$, then the natural transformation $\operatorname{sk}_n\mathcal{A}[a]\to\mathcal{A}[a]$ is an isomorphism.

On the other hand, we can give an explicit description for $\operatorname{sk}_nX$.
For this, we shall need the following lemma:

\begin{lemma}\label{lem:sk-injective}
Let $\mathcal{A}^-$ be a confluent degeneracy system on a small category $\mathcal{A}$.
Then the natural transformation $\operatorname{sk}_nX\to X$ is a monomorphism for every $X\in\mathcal{A}^\wedge$.
\end{lemma}
\begin{proof}
It suffices to show that for each $a\in\mathcal{A}$, the map $\operatorname{sk}_nX(a)\to X(a)$ is injective.
Recall that we have
\[
\operatorname{sk}_nX \simeq \int^{a'\in\mathcal{A}_{\le n}} \mathcal{A}(\blankdot,a')\times X(a')\,,
\]
so that $\operatorname{sk}_nX(a)$ can be identified with the set
\[
\left(\coprod_{\substack{a'\in\mathcal{A}\\\deg a'\le n}} \mathcal{A}(a,a')\times X(a')\right)\bigg/ \sim\,,
\]
where the equivalence relation $\sim$ is generated by $(\mu f,x)\sim (f,x\mu)$ for $\mu:a'\to a''\in\mathcal{A}_{\le n}$.
The image of $(f,x)$ by $\operatorname{sk}_nX(a)\to X(a)$ is the composition $xf:\mathcal{A}[a]\to X$.

Suppose we have $(f_i,x_i)\in\mathcal{A}(a,a_i)\times X(a_i)$ for $i=1,2$ such that $x_1f_1=x_2f_2$.
Say $f_i=\delta_i\sigma_i$ for $i=1,2$ such that $\sigma_i\in\mathcal{A}^-$ and $\delta_i:a'_i\to a_i$ is a non-degenerate monomorphism in $\mathcal{A}$.
Since $\delta_i$ does not raise the degree by the condition \ref{cond:CDSorth}, we have $\deg a'_i\le n$.
Hence $(f_i,x_i)\sim (\sigma_i,x_i\delta_i)$.
Now we have $x_1\delta_1\sigma_1=x_2\delta_2\sigma_2$.
By the condition \ref{cond:CDSconf}, there is a diagram
\[
\xymatrix@R-2ex{
  {} & \mathcal{A}[a'_1] \ar[dr]_{\tau_1} \ar@/^/[drr]^{x_1\delta_1} && \\
  \mathcal{A}[a] \ar[ur]^{\sigma_1} \ar[dr]_{\sigma_2} && \mathcal{A}[a_0] \ar[r]^{x_0} & X \\
  {} & \mathcal{A}[a'_2] \ar[ur]^{\tau_2} \ar@/_/[urr]_{x_2\delta_2} && }
\]
with $\tau_1,\tau_2\in\mathcal{A}^-$.
Then $\deg a_0\le n$ and we obtain
\[
(\sigma_1, x_1\delta_1)
\sim (\tau_1\sigma_1, x_0)
= (\tau_2\sigma_2, x_0)
\sim (\sigma_2, x_2\delta_2).
\]
Finally we have $(f_1,x_1)\sim (f_2,x_2)\in\coprod_{\substack{a'\in\mathcal{A}\\\deg a'\le n}} \mathcal{A}(a,a')\times X(a')$.
It follows that $\operatorname{sk}_nX(a)\to X(a)$ is injective as required.
\end{proof}

\begin{corol}\label{cor:explicit-sk}
For $X\in\mathcal{A}^\wedge$, $\operatorname{sk}_nX$ can be identified with a subpresheaf of $X$ consisting of cells $x:\mathcal{A}[a]\to X$ which factor as $\mathcal{A}[a]\to\mathcal{A}[a']\to X$ for some $a'\in\mathcal{A}$ with $\deg a'\le n$.
\end{corol}
\begin{proof}
By lemma \ref{lem:sk-injective}, $\operatorname{sk}_nX$ can be identified with its image by $\operatorname{sk}_nX\to X$.
For each $a\in\mathcal{A}$, the image of $\operatorname{sk}_nX(a)$ coincides with that of the set
\[
\coprod_{\substack{a'\in\mathcal{A}\\\deg a'\le n}} \mathcal{A}(a,a')\times X(a')\,.
\]
Hence the result follows.
\end{proof}

Another important consequence is the following:

\begin{corol}
For each $X\in\mathcal{A}^\wedge$, the morphism $\operatorname{sk}_nX\to X$ is a non-degenerate monomorphism.
\end{corol}
\begin{proof}
We have to show that the morphism is left orthogonal to morphisms in $\mathcal{A}^-$.
So suppose we have a square
\[
\xymatrix{
  \mathcal{A}[a] \ar[r]^{x} \ar[d]_{\sigma} & \operatorname{sk}_nX \ar[d] \\
  \mathcal{A}[a'] \ar[r] & X }
\]
with $\sigma\in\mathcal{A}^-$.
By corollary \ref{cor:explicit-sk}, there is a factorization $\mathcal{A}[a]\xrightarrow{\mu}\mathcal{A}[a_1]\to\operatorname{sk}_nX$ with $\deg a_1\le n$.
By considering, if it is needed, the factorization of $\mu$, we may assume $\mu\in\mathcal{A}^-$.
Taking the absolute pushout, we obtain the following diagram
\[
\xymatrix{
  \mathcal{A}[a] \ar[r]^{\mu} \ar[d]_{\sigma} \ar@{}[dr]|(.6){\pocorner} & \mathcal{A}[a_1] \ar[r] \ar[d]^{\sigma_1} & \operatorname{sk}_nX \ar[d] \\
  \mathcal{A}[a'] \ar[r] & \mathcal{A}[a'_1] \ar[r] & X }
\]
with $\sigma_1\in\mathcal{A}^-$.
Then since $\deg a'_1\le n$, by corollary \ref{cor:explicit-sk}, there is a lift $\mathcal{A}[a'_1]\to\operatorname{sk}_nX$ for the right square.
Finally $\mathcal{A}[a']\to\mathcal{A}[a'_1]\to\operatorname{sk}_nX$ is the required lift for the first square.
\end{proof}

\subsection{Reduced cells}
Next, we see that if $\mathcal{A}$ has a confluent degenerasy system, then each $X\in\mathcal{A}^\wedge$ consists of ``recuced'' cells.

Let $\mathcal{A}^-$ be a confluent degeneracy system on $\mathcal{A}$, and let $X\in\mathcal{A}^\wedge$.
Then we can consider the comma category $\commacat{\mathcal{A}^-}{X}$, which is a small category.
Objects of $\commacat{\mathcal{A}^-}{X}$ are cells $\mathcal{A}[a]\to X$, and morphisms are diagrams of the form
\[
\xymatrix@C-1em{
  \mathcal{A}[a] \ar[rr]^{\sigma} \ar[dr]_{x} && \mathcal{A}[a'] \ar[dl]^{x'} \\
  {} & X & {} }
\]
with $\sigma\in\mathcal{A}^-$.

\begin{lemma}\label{lem:cds-connect}
Let $\mathcal{A}^-$ be a confluent degeneracy system on $\mathcal{A}$, and let $X\in\mathcal{A}^\wedge$.
Then two cells $x_1:\mathcal{A}[a_1]\to X$ and $x_2:\mathcal{A}[a_2]\to X$ belong to the same path-component of $\commacat{\mathcal{A}^-}{X}$ if and only if there is a diagram
\[
\xymatrix@R-2ex{
  \mathcal{A}[a_1] \ar[dr]_{\sigma_1} \ar@/^/[drr]^{x_1} && \\
  {} & \mathcal{A}[a] \ar[r]^{x} & X \\
  \mathcal{A}[a_2] \ar[ur]^{\sigma_2} \ar@/_/[urr]_{x_2} && }
\]
\end{lemma}
\begin{proof}
If there is such a diagram, clearly $x_1$ and $x_2$ belong to the same component.
To see the converse, it suffices to show that we can find the diagram above whenever we have a diagram
\[
\xymatrix@R-2ex{
  {} & \mathcal{A}[a_1] \ar[dr]^{x_1} & {} \\
  \mathcal{A}[a_0] \ar[ur]^{\tau_1} \ar[dr]_{\tau_2} && X \\
  {} & \mathcal{A}[a_2] \ar[ur]_{x_2} & {} }
\]
with $\tau_1,\tau_2\in\mathcal{A}^-$.
But this immediately follows from the condition \ref{cond:CDSconf}.
\end{proof}

We denote by $\pi_0(\commacat{\mathcal{A}^-}{X})$ the set of path-components of $\commacat{\mathcal{A}^-}{X}$.
A morphism $f:X\to Y\in\mathcal{A}^\wedge$ induces a functor $f_\ast:\commacat{\mathcal{A}^-}{X}\to\commacat{\mathcal{A}^-}{Y}$, so that for $\alpha\in\pi_0(\commacat{\mathcal{A}^-}{X})$, the restriction $f_\ast:\alpha\to\commacat{\mathcal{A}^-}{Y}$ is a connected diagram.
It follows that  there is a component $\beta\in\pi_0(\commacat{\mathcal{A}^-}{Y})$ such that $f_\ast:\alpha\to\beta$.
We set $f_\ast(\alpha):=\beta$.

\begin{corol}\label{cor:non-deg-isom}
In the above situation, the following hold:
\begin{enumerate}[label={\rm(\arabic*)}]
  \item If $f:X\to Y\in\mathcal{A}^\wedge$ is non-degenerate, then the functor $f_\ast:\alpha\to f_\ast(\alpha)$ is a surjection on objects for each $\alpha\in\pi_0(\commacat{\mathcal{A}^-}{X})$.
  \item If $f:X\to Y\in\mathcal{A}^\wedge$ is a non-degenerate monomorphism, then the functor $f_\ast:\alpha\to f_\ast(\alpha)$ is an isomorphism of categories for each $\alpha\in\pi_0(\commacat{\mathcal{A}^-}{X})$.
\end{enumerate}
\end{corol}
\begin{proof}
Suppose $f:X\to Y$ is non-degenerate, and $(\mathcal{A}[a]\xrightarrow{y} Y)\in f_\ast(\alpha)$.
Then by lemma \ref{lem:cds-connect}, there is a cell $(\mathcal{A}[a']\xrightarrow{x} X)\in\alpha$ and a diagram
\[
\xymatrix@R-2ex@C-1em{
  \mathcal{A}[a'] \ar[rr]^{x} \ar[dr]_{\sigma'} && X \ar[dr]^{f} & {} \\
  {} & \mathcal{A}[a_0] \ar[rr]^{y_0} && Y \\
  \mathcal{A}[a] \ar[ur]^{\sigma} \ar@/_/[urrr]_{y} &&& }
\]
with $\sigma,\sigma'\in\mathcal{A}^-$.
Since $\sigma'\pitchfork f$, there is a cell $x_0:\mathcal{A}[a_0]\to X$ such that $x=x_0\sigma'$ and $fx_0=y_0$.
Thus we obtain $y=fx_0\sigma=f_\ast(x_0\sigma)$, and $y$ belongs to the image of $f_\ast$.
This implies that $f_\ast:\alpha\to f_\ast(\alpha)$ is a surjection on objects.

Now we assume that $f:X\to Y$ is a non-degenerate monomorphism.
Since $f$ is a monomorphism, the functor $f_\ast:\commacat{\mathcal{A}^-}{X}\to\commacat{\mathcal{A}^-}{Y}$ is fully faithful and injective on objects.
By the first part, $f_\ast:\alpha\to f_\ast(\alpha)$ is also surjective on objects.
It follows that $f_\ast:\alpha\to f_\ast(\alpha)$ is an isomorphism of categories.
\end{proof}

Now we define ``reduced'' cells for $X\in\mathcal{A}^\wedge$.
For $\alpha\in\pi_0(\commacat{\mathcal{A}^-}{X})$, we set
\[
\deg\alpha:=\min\left\{ \deg a \mid (\mathcal{A}[a]\to X)\in\alpha\right\}.
\]
By corollary \ref{cor:non-deg-isom}, if $f:X\to Y$ is non-degenerate, we have $\deg\alpha=\deg f_\ast(\alpha)$.
Let $X\in\mathcal{A}^\wedge$ and $\alpha\in\pi_0(\commacat{\mathcal{A}^-}{X})$.
Then the forgetful functor $\alpha\to\mathcal{A}\hookrightarrow\mathcal{A}^\wedge$ gives rise to a diagram in $\mathcal{A}^\wedge$.
We set
\[
\mathcal{A}[\alpha] := \colim_{(\mathcal{A}[a]\to X)\in\alpha} \mathcal{A}[a]\,.
\]
If $\deg\alpha=n$, we set $\partial\mathcal{A}[\alpha]:=\operatorname{sk}_{n-1}\mathcal{A}[\alpha]$.

\begin{lemma}\label{lem:cell-core}
Let $\mathcal{A}$ be a small category equipped with a confluent degeneracy system $\mathcal{A}^-$.
Let $X\in\mathcal{A}^\wedge$ and $\alpha\in\pi_0(\commacat{\mathcal{A}^-}{X})$.
We denote by $\bar\alpha\subset\alpha$ the full subcategory spanned by cells $\mathcal{A}[a]\to X$ in $\alpha$ with $\deg a=\deg\alpha$.
Then the embedding $\bar\alpha\hookrightarrow\alpha$ is final.
\end{lemma}
\begin{proof}
Let $(x_0:\mathcal{A}[a_0]\to X)\in\alpha$ and consider the category $\commacat{x_0}{\bar\alpha}$.
We have to show it is connected.
Suppose $x_1,x_2$ are objects in the category.
In other words, we have a commutative diagram
\[
\xymatrix{
  \mathcal{A}[a_0] \ar[r]^{\sigma_1} \ar[dr]^{x_0} \ar[d]_{\sigma_2} & \mathcal{A}[a_1] \ar[d]^{x_1} \\
  \mathcal{A}[a_2] \ar[r]^{x_2} & X }
\]
with $\sigma_1,\sigma_2\in\mathcal{A}^-$ and $\deg a_1=\deg a_2=\deg\alpha$.
By taking an absolute pushout, we have
\[
\xymatrix@R-2ex{
  {} & \mathcal{A}[a_1] \ar[dr]_{\tau_1} \ar@/^/[drr]^{x_1} && \\
  \mathcal{A}[a_0] \ar[ur]^{\sigma_1} \ar[dr]_{\sigma_2} && \mathcal{A}[a] \ar[r]^{x} & X \\
  {} & \mathcal{A}[a_2] \ar[ur]^{\tau_2} \ar@/_/[urr]_{x_2} && }
\]
with $\tau_1,\tau_2\in\mathcal{A}^-$.
Since $\tau_1$ and $\tau_2$ does not raise the degree, we have $\deg a=\deg\alpha$.
Hence $x_1\xrightarrow{\tau_1}x\xleftarrow{\tau_2}x_2$ is a zig-zag in $\commacat{x_0}{\bar\alpha}$.
It follows that the category $\commacat{x_0}{\bar\alpha}$ is connected.
\end{proof}

\begin{corol}\label{cor:cell-skeletal}
In the above situation, if $\deg\alpha=n$, the natural transformation $\operatorname{sk}_n\mathcal{A}[\alpha]\to\mathcal{A}[\alpha]$ is an isomorphism.
\end{corol}
\begin{proof}
Note that since the functor $\operatorname{sk}_k:\mathcal{A}^\wedge\to\mathcal{A}^\wedge$ has a right adjoint, it commutes with arbitrary small colimits.
Hence, by lemma \ref{lem:cell-core}, the natural transformation is isormophic to the following composition:
\[
\begin{split}
\operatorname{sk}_n\mathcal{A}[\alpha]
&\simeq \operatorname{sk}_n\left(\colim_{\substack{(\mathcal{A}[a]\to X)\in\alpha\cr\deg a=n}}\mathcal{A}[a]\right)
\simeq \colim_{\substack{(\mathcal{A}[a]\to X)\in\alpha\cr\deg a=n}} \operatorname{sk}_n\mathcal{A}[a] \\
&\xrightarrow{\sim} \colim_{\substack{(\mathcal{A}[a]\to X)\in\alpha\cr\deg a=n}} \mathcal{A}[a]
\simeq \mathcal{A}[\alpha]
\end{split}
\]
\end{proof}

By the definition of $\mathcal{A}[\alpha]$, the family $\left\{(\mathcal{A}[a]\to X)\in\alpha\right\}$ induces a morphism $\xi_\alpha:\mathcal{A}[\alpha]\to X$.
We have some essential results:

\begin{lemma}\label{lem:cell-intersect}
Let $\mathcal{A}$ be a small category with a confluent degeneracy system $\mathcal{A}^-$.
Suppose $X\in\mathcal{A}^-$ and $\alpha\in\pi_0(\commacat{\mathcal{A}^-}{X})$ with $\deg\alpha=n$.
Then the following hold:
\begin{enumerate}[label={\rm(\arabic*)}]
  \item\label{sublem:pres-nondeg} If $v:\mathcal{A}[a]\to\mathcal{A}[\alpha]$ is a cell such that $\xi_\alpha v:\mathcal{A}[a]\to X$ factors through $\mathcal{A}[a']$ with $\deg a'\le n-1$, then $v$ itself factors through some $\mathcal{A}[a'_1]$ with $\deg a'_1\le n-1$.
  \item\label{sublem:monic-nondeg} Suppose $v_1,v_2:\mathcal{A}[a]\rightrightarrows\mathcal{A}[\alpha]$ are cells which do not factor through any $\mathcal{A}[a']$ with $\deg a'<n$.
Then $\xi_\alpha v_1=\xi_\alpha v_2$ implies $v_1=v_2$.
  \item\label{sublem:cell-disjoint} If $\beta\in\pi_0(\commacat{\mathcal{A}^-}{X})$ is another component with $\deg\beta=n$, then every cell of the presheaf $\mathcal{A}[\alpha]\times_X\mathcal{A}[\beta]$ factors through $\mathcal{A}[a']$ with $\deg a'\le n-1$.
\end{enumerate}
\end{lemma}
\begin{proof}
\ref{sublem:pres-nondeg}: 
Suppose we have a square
\[
\xymatrix{
  \mathcal{A}[a] \ar[r]^{v} \ar[d]_{\sigma} & \mathcal{A}[\alpha] \ar[d]^{\xi_\alpha} \\
  \mathcal{A}[a'] \ar[r]^{x'} & X }
\]
with $\sigma\in\mathcal{A}^-$ and $\deg a'\le n-1$.
Since $\mathcal{A}[\alpha](a)\simeq\colim_\alpha \mathcal{A}[a'']$, there is a cell $x'':\mathcal{A}[a'']\to X$ and $v'':\mathcal{A}[a]\to\mathcal{A}[a'']$ such that the following diagram is commutative:
\[
\xymatrix{
  \mathcal{A}[a] \ar[r]^{v''} \ar[dr]_{v} & \mathcal{A}[a''] \ar[r]^{x''} \ar[d]^{i^{a''}} & X \\
  {} & \mathcal{A}[\alpha] \ar[ur]_{\xi_\alpha} & {} }
\]
Hence we have the following square:
\[
\xymatrix{
  \mathcal{A}[a] \ar[r]^{v''} \ar[d]_{\sigma} & \mathcal{A}[a''] \ar[d]^{x''} \\
  \mathcal{A}[a'] \ar[r]^{x'} & X }
\]
Let $v''=\delta\sigma''$ be the factorization such that $\sigma'':a\to a'_1\in\mathcal{A}^-$ and $\delta:a'_1\to a''$ is a non-degenerate monomorphism.
Since $\deg a'\le n-1 < \deg\alpha$, $x'$ cannot belong to $\alpha$, and hence $\delta$ cannot be an isomorphism.
This implies $\deg a'_1\le n-1$, and the part \ref{sublem:pres-nondeg} follows.

\ref{sublem:monic-nondeg}: 
Let $v_1,v_2:\mathcal{A}[a]\rightrightarrows\mathcal{A}[\alpha]$ be as in the assumption, that is, we have $\xi_\alpha v_1=\xi_\alpha v_2$.
The similar discussion above implies that we can take a commutative diagram
\[
\xymatrix@R-2ex{
  {} & \mathcal{A}[a_1] \ar[dr]_{x_1} \ar[rr]^{i^1} && \mathcal{A}[\alpha] \ar[dl]^{\xi_\alpha} \\
  \mathcal{A}[a] \ar[ur]^{\mu_1} \ar[dr]_{\mu_2} && X & \\
  {} & \mathcal{A}[a_2] \ar[ur]^{x_2} \ar[rr]^{i^2} && \mathcal{A}[\alpha] \ar[ul]_{\xi_\alpha} }
\]
such that $i^k\mu_k=v_k$ for $k=1,2$, and $\deg a_1=\deg a_2=n$.
Since $v_k$ does not factor through any cells of degree $\le n-1$, we have $\mu_k\in\mathcal{A}^-$.
Thus $\mu_1x_1,\mu_2x_2\in\alpha$.
This implies that $v_1=i^1\mu_1=i^2\mu_2=v_2$ are natural injections $\mathcal{A}[a]\to\mathcal{A}[\alpha]$ at the same cell $x_1\mu_1=x_2\mu_2$.
Hence the part \ref{sublem:monic-nondeg} follows.

\ref{sublem:cell-disjoint}: 
Let $\beta\in\pi_0(\commacat{\mathcal{A}^-}{X})$ such that $\deg\beta=n$ and $\beta\neq\alpha$.
Notice that each cell $\mathcal{A}[a]\to\mathcal{A}[\alpha]\times_X\mathcal{A}[\beta]$ corresponds to a commutative square
\[
\xymatrix{
  \mathcal{A}[a] \ar[r]^{v} \ar[d]_{w} & \mathcal{A}[\alpha] \ar[d]^{\xi_\alpha} \\
  \mathcal{A}[\beta] \ar[r]^{\xi_\beta} & X }
\]
By the same argument as above, we may choose cells $(\mathcal{A}[a_1]\xrightarrow{x_1}X)\in\alpha$, $(\mathcal{A}[a_2]\xrightarrow{x_2}X)\in\beta$ and factorizations
\[
\begin{split}
v:\mathcal{A}[a] &\xrightarrow{\mu} \mathcal{A}[a_1] \to \mathcal{A}[\alpha]\,,\\
w:\mathcal{A}[a] &\xrightarrow{\nu} \mathcal{A}[a_2] \to \mathcal{A}[\beta]\,.
\end{split}
\]
Then we have the following square:
\[
\xymatrix{
  \mathcal{A}[a] \ar[r]^{\mu} \ar[d]_{\nu} & \mathcal{A}[a_1] \ar[d]^{x_1} \\
  \mathcal{A}[a_2] \ar[r]^{x_2} & X }
\]
Since $\alpha\neq\beta$, either $\mu$ or $\nu$ factors through some $\mathcal{A}[a']$ with $\deg a'\le n-1$.
But in the case, by the part \ref{sublem:monic-nondeg}, the other also factors through some $\mathcal{A}[a'']$ with $\deg a''\le n-1$.
By the condition \ref{cond:CDSconf}, we may take $a'=a''$.
Thus we obtain \ref{sublem:cell-disjoint}.
\end{proof}

Now we are ready to prove a key result:

\begin{prop}\label{prop:sk-cell-extend}
Let $\mathcal{A}$ be a small category with a confluent degeneracy system $\mathcal{A}^-$.
Then for every $X\in\mathcal{A}$, the following square is bicartesian:
\[
\xymatrix@R-1ex{
  {\displaystyle\raisedunderop{\coprod}{\substack{\alpha\in\pi_0(\commacat{\mathcal{A}^-}{X})\\ \deg\alpha=n}} \partial\mathcal{A}[\alpha]} \ar[r] \ar[d] & \operatorname{sk}_{n-1}X \ar[d] \\
  {\displaystyle\raisedunderop{\coprod}{\substack{\alpha\in\pi_0(\commacat{\mathcal{A}^-}{X})\\ \deg\alpha=n}} \mathcal{A}[\alpha]} \ar[r] & \operatorname{sk}_n X }
\]
\end{prop}
\begin{proof}
By lemma \ref{lem:sk-injective}, the left edge is a monomorphism, so that it suffices to show the square is a pushout, since $\mathcal{A}^\wedge$ is a presheaf topos.
Moreover, since all edges belong to the image of the left Kan extension $(j_n)_!:\mathcal{A}_{\le n}^\wedge\to\mathcal{A}^\wedge$ as shown in the previous subsection, it suffices to show that the square is a pushout at objects $a\in\mathcal{A}$ with $\deg a = n$.

Suppose $a\in\mathcal{A}$ is an object of degree $n$.
First, notice that the map
\[
\left(\raisedunderop{\coprod}{\substack{\alpha\in\pi_0(\commacat{\mathcal{A}^-}{X})\\\deg\alpha=n}}\mathcal{A}[\alpha](a)\right)\amalg\operatorname{sk}_{n-1}X(a)
\to \operatorname{sk}_nX(a)
\]
is surjective.
Indeed, if a cell $x:\mathcal{A}[a]\to X$ factors through a cell of degree $\le n$, then $x$ is a cell of $\operatorname{sk}_{n-1}X$ by corollary \ref{cor:explicit-sk}.
Otherwise, put $\alpha=[x]\in\pi_0(\commacat{\mathcal{A}^-}{X})$.
Then we have $\deg\alpha=n$, and $x$ factors through $\mathcal{A}[\alpha]$.
Hence the map is surjective.
This implies the map
\[
\left(\raisedunderop{\coprod}{\substack{\alpha\in\pi_0(\commacat{\mathcal{A}^-}{X})\\\deg\alpha=n}}\mathcal{A}[\alpha](a)\right)\cup\operatorname{sk}_{n-1}X(a)
\to \operatorname{sk}_nX(a)
\]
is also surjective, where the symbol $\cup$ denotes the pushout over the left corner of the square.
On the other hand, it is injective by \ref{lem:cell-intersect}.
So the result follows.
\end{proof}

\subsection{Generators of non-degenerate monomorphisms}
In this section, we give a ``good generator'' of the class of non-degenerate monomorphisms.
We fix a confluent degeneracy class $\mathcal{A}^-$ on a small category $\mathcal{A}$, and denote by $\mathsf{M}$ the class of non-degenerate monomorphisms in $\mathcal{A}$.
Recall that a class $\mathcal{M}$ of morphisms in a bicomplete category $\mathcal{C}$ is said to be saturated if it is closed under pushouts, retracts and transfinite compositions.

\begin{prop}
Let $\mathcal{A}^-$ be a confluent degeneracy system on a small category $\mathcal{A}$.
Then the class of non-degenerate monomorphisms in $\mathcal{A}^\wedge$ is saturated.
\end{prop}
\begin{proof}
By definition, the class of non-degenerate morphisms equals to the right orthogonal class of $\mathcal{A}^-$.
Hence it is clearly closed under retracts.
On the other hand, the class of monomorphisms is closed under retracts.
These imply that $\mathsf{M}$ is closed under retracts.

Suppose that $f:X\to Y\in\mathcal{A}^\wedge$ is a non-degenerate monomorphism, and we have a diagram
\begin{equation}\label{eq:polift}
\xymatrix{
  X \ar[r]^{p} \ar[d]_{f} \ar@{}[dr]|(.6){\pocorner} & V \ar[d]^{g} & \mathcal{A}[a] \ar[l]_{v} \ar[d]^{\sigma} \\
  Y \ar[r]^{q} & W & \mathcal{A}[a'] \ar[l]_{w} }
\end{equation}
where $\sigma\in\mathcal{A}^-$ and the left square is a pushout.
We have to find a lift for the right square.
Note that colimits in $\mathcal{A}^\wedge$ is pointwise colimits, hence the left square gives rise to a pushout in $\mathbf{Set}$ at $a'\in\mathcal{A}$.
This implies that $w:\mathcal{A}[a']\to W$ factors through either $q:Y\to W$ or $g:V\to W$ in the diagram.
We asserts that it in fact factors through $g$.
Indeed, suppose there is a cell $y:\mathcal{A}[a']\to Y$ such that $qy=w$.
Then we have $hv=w\sigma=qy\sigma$.
Since $\mathcal{A}^\wedge$ is a presheaf topos and $f$ is a monomorphism, every pushout square of $f$ is bicartesian (i.e. it is simultaneously a pushout and a pullback).
Thus there is $x:\mathcal{A}[a]\to X$ such that $v=px$ and the following diagram is commutative:
\[
\xymatrix{
  \mathcal{A}[a] \ar[r]^{x} \ar[d]_{\sigma} & X \ar[r]^{p} \ar[d]_{f} \ar@{}[dr]|{\text{bicart.}} & V \ar[d]^{g} \\
  \mathcal{A}[a'] \ar[r]^{y} & Y \ar[r]^{q} & W }
\]
Since $f$ is non-degenerate, we have $\sigma\pitchfork f$, so the left square has a lift $x':\mathcal{A}[a']\to X$, which gives rise to a factorization $w=gpx'$ through $g$.

Now say $w=gv'$ for $v':\mathcal{A}[a']\to V$.
Then we have $gv'\sigma=w\sigma=gv$.
Since the class of monomorphisms is closed under pushouts in the presheaf topos $\mathcal{A}^\wedge$, $g$ is monomorphism.
Hence we have $v'\sigma=v$.
This implies that $v':\mathcal{A}[a']\to V$ is a lift for the square in the diagram \ref{eq:polift}.
Therefore the class $\mathsf{M}$ is closed under pushouts.

Finally, we show that $\mathsf{M}$ is closed under transfinite compositions.
Let $\lambda$ be a limit ordinal and $X_\bullet:\lambda\to\mathcal{A}^\wedge$ be a $\lambda$-sequence such that each $X_\alpha\to X_{\alpha+1}$ is a non-degenerate monomorphism.
Set $X_\lambda:=\colim_{\alpha<\lambda} X_\alpha$ and $i^\alpha:X_\alpha\to X_\lambda$ to be the natural injection.
Suppose we have a diagram
\[
\xymatrix{
  \mathcal{A}[a] \ar[r]^{x_0} \ar[d]_{\sigma} & X_0 \ar[d]^{i^0} \\
  \mathcal{A}[a'] \ar[r]^{x'_\lambda} & X_\lambda }
\]
with $\sigma\in\mathcal{A}^-$.
Since $X_\lambda(a')$ is the colimit of the sequence $\{X_\alpha(a')\}_{\alpha<\lambda}$, $x'_\lambda$ factors through $x'_\alpha:\mathcal{A}[a']\to X_\alpha$ for some ordinal $\alpha<\lambda$ followed by $i^\alpha:X_\alpha\to X_\lambda$.
Let $\alpha_0$ be the smallest among such ordinals.
Then $\alpha_0$ cannot be a successor ordinal, since each $X_\alpha\to X_{\alpha+1}$ is right orthogonal to $\sigma$.
On the other hand $\alpha_0$ cannot be a non-zero limit ordinal, since if $\alpha$ is a non-zero limit ordinal, $X_\alpha(a')$ should be defined as the colimit of the sequence $\{X_\beta(a')\}_{\beta<\alpha}$.
Therefore we obtain $\alpha_0=0$, and $x'_\lambda=i^0 x'_0$.
Now we have $i^0x'_0\sigma=x'_\lambda\sigma=i^0 x_0$.
Since the class of monomorphisms is closed under transfinite compositions in the presheaf topos $\mathcal{A}^\wedge$, $i^0$ is a monomorphism, hence we obtain $x'_0\sigma=x_0$.
This implies that $x'_0$ is in fact a lift for the square above, which completes the proof.
\end{proof}

The aim of this section is finding a ``good'' class $S$ of non-degenerate monomorphisms which generates $\mathsf{M}$ as a saturated class; i.e. $\mathsf{M}$ is the smallest saturated class containing $S$.

We define a class $S$ of morphisms in $\mathcal{A}^-$ by
\[
S:=
\left\{\partial\mathcal{A}[\alpha]\to\mathcal{A}[\alpha]\,\middle|\,\alpha\in\pi_0(\commacat{\mathcal{A}^-}{X}),X\in\mathcal{A}^\wedge\right\}\,.
\]

\begin{prop}
Let $\mathcal{A}$ be a small category with a confluent degeneracy system $\mathcal{A}^-$.
Then the class $\mathsf{M}$ of non-degenerate monomorphisms in $\mathcal{A}$ is the smallest saturated class containing $S$ defined above.
More precisely, every non-degenerate monomorphism is isomorphic to a transfinite composition of pushouts of morphisms in $S$.
\end{prop}
\begin{proof}
The first part follows from the last part.
Suppose $f:X\to Y$ is a non-degenerate monomorphism.
Let $X^f_0:=X$ and define $X^f_n\in\mathcal{A}^\wedge$ for $n\ge 1$ by the following pushout square:
\[
\xymatrix{
  \operatorname{sk}_n X \ar[r]^{\operatorname{sk}_n f} \ar[d] \ar@{}[dr]|(.6){\pocorner} & \operatorname{sk}_n Y \ar[d] \\
  X \ar[r]^{f} & X^f_n }
\]
There is a morphism $X^f_n\to X^f_{n+1}$ induced by the diagram below:
\begin{equation}
\label{eq:DGSrel-cell}
\begin{gathered}
\xymatrix@-3ex{
  {} & \operatorname{sk}_{n+1}X \ar[rr] \ar[dd]|{\hole} && \operatorname{sk}_{n+1}Y \ar[dd] \\
  \operatorname{sk}_nX \ar[ur] \ar[rr] \ar[dd] && \operatorname{sk}_nY \ar[ur] \ar[dd] & {} \\
  {} & X \ar[rr]|{\hole} && X^f_{n+1} \\
  X \ar@{=}[ur] \ar[rr] && X^f_n \ar[ur] & {} }
\end{gathered}
\end{equation}
It is clear that if $a\in\mathcal{A}[a]$ with $\deg a\le n$, the natural maps $X^f_n(a)\to Y(a)$ and $X^f_n(a)\to X^f_{n+1}(a)$ are isormorphisms.
Hence we have an isomorphism $\colim_{n\to\infty} X^f_n\simeq Y$.
Since $f:X\to Y$ factors as $X\to X^f_n\to Y$, $f$ is isomorphic to the transfinite composition $X^f_0\to \colim_{n\to\infty}X^f_n$.
So we show each $X^f_n\to X^f_{n+1}$ is a pushout of a direct coproduct of morphisms in $S$.

Now since the front and back faces of the commutative cube \eqref{eq:DGSrel-cell} are pushouts, we obtain the pushout square:
\begin{equation}
\label{eq:DGSrel-sk}
\let\objectstyle=\displaystyle
\begin{gathered}
\xymatrix{
  \operatorname{sk}_{n+1} X \mathop{\amalg}_{\operatorname{sk}_nX} \operatorname{sk}_n Y \ar[r] \ar[d] \ar@{}[dr]|(.6){\pocorner} & X^f_n \ar[d] \\
  \operatorname{sk}_{n+1}Y \ar[r] & X^f_{n+1} }
\end{gathered}
\end{equation}
On the other hand, we have the commutative cube below:
\[
\let\objectstyle=\displaystyle
\xymatrix@-5ex{
  {} & \raisedunderop\coprod{\substack{\beta\in\pi_0(\commacat{\mathcal{A}^-}{Y})\\\deg\beta=n+1}} \partial\mathcal{A}[\beta] \ar[rr] \ar[dd]|{\hole} && \operatorname{sk}_n Y \ar[dd] \\
  \raisedunderop\coprod{\substack{\alpha\in\pi_0(\commacat{\mathcal{A}^-}{X})\\\deg\alpha=n+1}} \partial\mathcal{A}[\alpha] \ar[rr] \ar[dd] \ar[ur] && \operatorname{sk}_n X \ar[dd] \ar[ur] & {} \\
  {} & \raisedunderop\coprod{\substack{\beta\in\pi_0(\commacat{\mathcal{A}^-}{Y})\\\deg\beta=n+1}} \mathcal{A}[\beta] \ar[rr]|-(.53){\hole} && \operatorname{sk}_{n+1} Y \\
  \raisedunderop\coprod{\substack{\alpha\in\pi_0(\commacat{\mathcal{A}^-}{X})\\\deg\alpha=n+1}} \mathcal{A}[\alpha] \ar[rr] \ar[ur] && \operatorname{sk}_{n+1} X \ar[ur] & {} }
\]
Since $f$ is a non-degenerate monomorphism, the functor $f_\ast:\commacat{\mathcal{A}^-}{X}\to\commacat{\mathcal{A}^-}{Y}$ is an embedding and component-wise isomorphism by corollary \ref{cor:non-deg-isom}.
Hence $f_\ast:\pi_0(\commacat{\mathcal{A}^-}{X})\to\pi_0(\commacat{\mathcal{A}^-}{Y})$ is injective and each $\partial\mathcal{A}[\alpha]\to\partial\mathcal{A}[f_\ast\alpha]$ and $\mathcal{A}[\alpha]\to\mathcal{A}[f_\ast\alpha]$ are isomorphisms.
We denote by $C_{n+1}$ the set of $\beta\in\pi_0(\commacat{\mathcal{A}^-}{X})$ with $\deg\beta=n+1$ which does not belong to the image of $f_\ast$.
Then we obtain the following pushout square:
\begin{equation}
\label{eq:DGSrel-cell-sub}
\let\objectstyle=\displaystyle
\xymatrix{
  \raisedunderop\coprod{\beta\in C_{n+1}} \partial\mathcal{A}[\beta] \ar[r] \ar[d] \ar@{}[dr]|(.6){\pocorner}& \operatorname{sk}_{n+1}X \mathop{\amalg}_{\operatorname{sk}_nX} \operatorname{sk}_n Y \ar[d] \\
  \raisedunderop\coprod{\beta\in C_{n+1}} \mathcal{A}[\beta] \ar[r] & \operatorname{sk}_{n+1}Y }
\end{equation}
Combining squares \eqref{eq:DGSrel-sk} and \eqref{eq:DGSrel-cell-sub}, we obtain the following pushout square:
\[
\let\objectstyle=\displaystyle
\xymatrix{
  \raisedunderop\coprod{\beta\in C_{n+1}} \partial\mathcal{A}[\beta] \ar[r] \ar[d] \ar@{}[dr]|(.6){\pocorner} & X^f_n\ar[d] \\
  \raisedunderop\coprod{\beta\in C_{n+1}} \mathcal{A}[\beta] \ar[r] & X^f_{n+1} }
\]
This is the required result.
\end{proof}

It is natural to ask whether the class $\mathcal{S}$ is a small set.
For EZ categories, it is true.
To see this, notice that if $\mathcal{A}$ is an EZ category, then every degree-preserving split epimorphism is an isomorphism.
Hence for every $X\in\mathcal{A}^\wedge$ and every $\alpha\in\pi_0(\commacat{\mathcal{A}^-}{X})$, the full subcategory $\bar\alpha$ of $\alpha$ spanned by cells $(x:\mathcal{A}[a]\to X)\in\alpha$ with $\deg a=\deg\alpha$ is a non-empty connected groupoid, which is equivalent to a group.

\begin{prop}
\label{prop:cell-stab}
Let $\mathcal{A}$ be an EZ category, and let $X\in\mathcal{A}^\wedge$ and $\alpha\in\pi_0(\commacat{\mathcal{A}}{X})$.
Choose a cell $(x:\mathcal{A}[a]\to X)\in\alpha$ with $\deg a=\deg\alpha$, and set
\[
\mathrm{Stab}(x):=\{\pi\in\mathrm{Aut}_{\mathcal{A}}(a)\mid x\pi=x\}\,.
\]
Then we have an isomorphism
\[
\mathcal{A}[\alpha] \simeq \mathrm{Stab}(x)\backslash\mathcal{A}[a]\,.
\]
\end{prop}
\begin{proof}
We have the embedding:
\[
\mathrm{Stab}(x)\xrightarrow{x}\bar\alpha\,,
\]
where $\bar\alpha$ is the full subcategory of $\alpha$ spanned by cells $(x':\mathcal{A}[a']\to X)\in\alpha$ with $\deg a'=\deg\alpha$.
The embedding is full, so since $\bar\alpha$ is a groupoid, it is an equivalence of categories.
It follows that we have isomorphisms
\[
\mathcal{A}[\alpha]
\simeq \colim_{(x':\mathcal{A}[a']\to X)\in\bar\alpha}\mathcal{A}[a']
\simeq \colim_{\pi\in\mathrm{Stab}(x)} \mathcal{A}[a]
\simeq \mathrm{Stab}(x)\backslash\mathcal{A}[a]\,,
\]
which is the required result.
\end{proof}

Notice that, by corollary \ref{cor:explicit-sk}, the subpresheaf $\partial\mathcal{A}[a]$ of $\mathcal{A}[a]$ is stable under the left action of $\mathrm{Aut}_{\mathcal{A}}(a)$.
It follows that $\mathrm{Aut}_{\mathcal{A}}(a)$ acts on $\partial\mathcal{A}[a]=\operatorname{sk}_{n}\mathcal{A}[a]$.
Since $\operatorname{sk}_n$ commutes with colimits, for every subgroup $H<\mathrm{Aut}_{\mathcal{A}}(a)$, there is a natural isomorphism
\[
H\backslash\partial\mathcal{A}[a]\simeq \mathrm{sk}_n(H\backslash\mathcal{A}[a])\,.
\]
Finally we obtain the following result:

\begin{corol}[\cite{Isa}]
\label{cor:ez-cellular}
Let $\mathcal{A}$ is an EZ category, and set
\[
S_0 := \{H\backslash\partial\mathcal{A}[a]\hookrightarrow H\backslash\mathcal{A}[a]\mid a\in\mathcal{A},H<\mathrm{Aut}_{\mathcal{A}}(a)\}\,.
\]
Then $S_0$ is a (small) set, and every monomorphism in $\mathcal{A}^\wedge$ is a transfinite composition of pushouts of morphisms in $S_0$.
\end{corol}

\section{Relative Spans}
\label{sec:rel-span}
The notion of spans in a category was first introduced by B\'enabou in \cite{Ben} as an example of bicategories.
It is a kind of generalization of binary relations, and closely related to I-categories, see \cite{Kaw}.
Dawson, Par\'e and Pronk investigate spans in a more categorical viewpoint in \cite{DPP1} and \cite{DPP2}.
We here introduce a little bit general notion here.

\subsection{The category of relative spans}
\begin{defin}
Let $\mathcal{A}$ be a category and $\mathcal{A}_0$ be a wide subcategory of $\mathcal{A}$ which is essentially closed under pullbacks; i.e. any pullback of an $\mathcal{A}_0$-morphism is isomorphic to an $\mathcal{A}_0$-morpihsm.
Then an $(\mathcal{A},\mathcal{A}_0)$-span from an object $x$ to an object $y$ in $\mathcal{A}$ is a diagram in $\mathcal{A}$ of the form
\[
\xymatrix@=4ex{
  {} & a \ar[dl]_{\gamma} \ar[dr]^f & {} \\
  x & {} & y }
\]
with $\gamma\in\mathcal{A}_0$.
We denote by $\mathrm{Span}(\mathcal{A},\mathcal{A}_0)(x,y)$ the set of all $(\mathcal{A},\mathcal{A}_0)$-spans from $x$ to $y$.
\end{defin}

We can compose $(\mathcal{A},\mathcal{A}_0)$-spans.
The composition is given by the following diagram
\[
\xymatrix@=4ex{
  {} & {} & c \ar[dl]_{\delta'} \ar@{}[dd]|(.4){\rotatebox{-45}\pbcorner} \ar[dr]^{f'} & {} & {} \\
  {} & a \ar[dl]_{\gamma} \ar[dr]^{f} & {} & b \ar[dl]_{\delta} \ar[dr]^{g} \\
  x & {} & y & {} & z }
\]
With this composition, we obtain a category $\mathrm{Span}(\mathcal{A},\mathcal{A}_0)$ under some conditions.
In particular, we are interested in the case where the following condition is satisfied:
\begin{enumerate}[label=$(\bigstar)$]
\item\label{cond:star} $\begin{cases}
\quad\text{$\mathcal{A}_0$ has no non-trivial isomorphism.}\\
\quad\text{For morphisms $\delta,\gamma$ of $\mathcal{A}$, if $\delta$ and $\delta\gamma$ are $\mathcal{A}_0$-morphisms, then so is $\gamma$.}
\end{cases}$
\end{enumerate}

\begin{lemma}
Let $\mathcal{A}$ be a category and $\mathcal{A}_0$ be a subcategory.
Suppose that $\mathcal{A}_0$ satisfies the condition \ref{cond:star}.
Then $\mathrm{Span}(\mathcal{A},\mathcal{A}_0)$ forms a (strict) category.
\end{lemma}
\begin{proof}
First notice that if $\mathcal{A}_0$ satisfies the condition \ref{cond:star}, then for every $f:a\to y$ of $\mathcal{A}$ and $gamma:b\to y$ of $\mathcal{A}_0$, the pullback square
\[
\xymatrix{
  c \ar[r]^{\delta'} \ar@{}[dr]|(.4){\pbcorner} \ar[d]_{f'} & b \ar[d]^{f} \\
  a \ar[r]^{\delta} & y }
\]
is unique; cf. the proof of lemma \ref{lem:thinpow-pb}.
Since $\mathcal{A}_0$ is closed under pullback, $\delta'$ is a $\mathcal{A}_0$ morphism.
Thus the composition
\[
\xymatrix@=4ex{
  {} & c \ar[dl]_{\gamma\delta'} \ar[dr]^{gf'} & {} \\
  x & {} & z }
\]
is well-defined element of $\mathrm{Span}(\mathcal{A},\mathcal{A}_0)(x,z)$.
Moreover, the uniqueness of the pullback implies the associativity of the composition:
Indeed the left-associative composition is given by
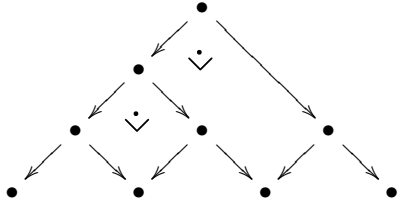
\[
\xymatrix@=3ex{
  {} & {} & {} & \bullet \ar[dl] \ar[ddrr] \ar@{}[dd]|(.4){\rotatebox{-45}{\pbcorner}} & {} & {} & {} \\
  {} & {} & \bullet \ar[dl] \ar[dr] \ar@{}[dd]|(.4){\rotatebox{-45}{\pbcorner}} & {} & {} & {} & {} \\
  {} & \bullet \ar[dl] \ar[dr] & {} & \bullet \ar[dl] \ar[dr] & {} & \bullet \ar[dl] \ar[dr] & {} \\
  \bullet & {} & \bullet & {} & \bullet & {} & \bullet }
\]
while the right-associative composition is given by
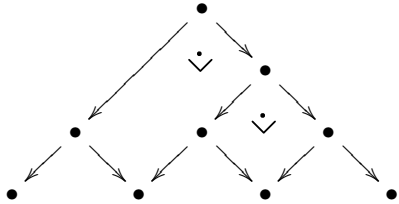
\[
\xymatrix@=3ex{
  {} & {} & {} & \bullet \ar[ddll] \ar[dr] \ar@{}[dd]|(.4){\rotatebox{-45}{\pbcorner}} & {} & {} & {} \\
  {} & {} & {} & {} & \bullet \ar[dl] \ar[dr] \ar@{}[dd]|(.4){\rotatebox{-45}{\pbcorner}} & {} & {} \\
  {} & \bullet \ar[dl] \ar[dr] & {} & \bullet \ar[dl] \ar[dr] & {} & \bullet \ar[dl] \ar[dr] & {} \\
  \bullet & {} & \bullet & {} & \bullet & {} & \bullet }
\]
By the uniqueness, they coincide with the composition given by the diagram
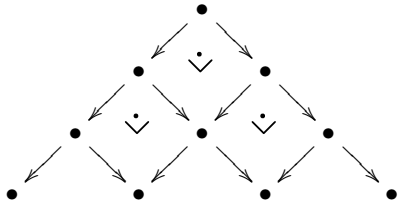
\[
\xymatrix@=3ex{
  {} & {} & {} & \bullet \ar[dl] \ar[dr] \ar@{}[dd]|(.4){\rotatebox{-45}{\pbcorner}} & {} & {} & {} \\
  {} & {} & \bullet \ar[dl] \ar[dr] \ar@{}[dd]|(.4){\rotatebox{-45}{\pbcorner}} & {} & \bullet \ar[dl] \ar[dr] \ar@{}[dd]|(.4){\rotatebox{-45}{\pbcorner}} & {} & {} \\
  {} & \bullet \ar[dl] \ar[dr] & {} & \bullet \ar[dl] \ar[dr] & {} & \bullet \ar[dl] \ar[dr] & {} \\
  \bullet & {} & \bullet & {} & \bullet & {} & \bullet }
\]
Therefore the composition is associative.
\end{proof}

\begin{rem}
Note that in general cases, $\mathrm{Span}(\mathcal{A},\mathcal{A}_0)$ is not a category but a bicategory.
A 2-morphism of $\mathrm{Span}(\mathcal{A},\mathcal{A}_0)$ is a diagram of the form:
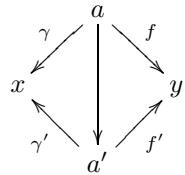
\[
\xymatrix@=4ex{
  {} & a \ar[dl]_{\gamma} \ar[dd] \ar[dr]^f & {} \\
  x & {} & y \\
  {} & a' \ar[ul]^{\gamma'} \ar[ur]_{f'} & {} }
\]
In the bicategory, the associativity $(\alpha_3\alpha_2)\alpha_1\simeq\alpha_3(\alpha_2\alpha_1)$ of compositions of 1-morphisms is guaranteed only up to isomorphisms.
See \cite{DPP1} and \cite{DPP2} for more details.
\end{rem}

Now suppose $\mathcal{A}_0$ satisfies the condiotion \ref{cond:star}.
Then we have an embedding of $\mathcal{A}$ into $\mathrm{Span}(\mathcal{A},\mathcal{A}_0)$:
Indeed, a morphism $f:x\to y$ of $\mathcal{A}$ is embedded as the following diagram in $\mathrm{Span}(\mathcal{A},\mathcal{A}_0)$:
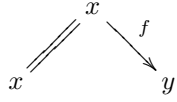
\[
\xymatrix@=4ex{
  {}& x \ar@{=}[dl] \ar[dr]^{f} & {} \\
  x & {} & y }
\]
Therefore, we do not ditinguish $\mathcal{A}$ with its image in $\mathrm{Span}(\mathcal{A},\mathcal{A}_0)$.
Similarly, we have an embedding $(\mathchar`-)^{\dagger}:\mathcal{A}_0^{\mathrm{op}}\to\mathrm{Span}(\mathcal{A},\mathcal{A}_0)$ which assigns to $\gamma:y\to x$ of $\mathcal{A}_0$ the diagram
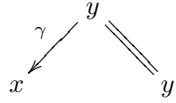
\[
\xymatrix@=4ex{
  {}& y \ar[dl]_{\gamma} \ar@{=}[dr] & {} \\
  x & {} & y }
\]
We denote by $\mathcal{A}_0^\dagger$ the embedded image of $\mathcal{A}_0^{\mathrm{op}}$.
Then each morphism of $\mathrm{Span}(\mathcal{A},\mathcal{A}_0)$ depicted as
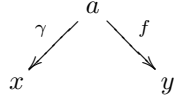
\[
\xymatrix@=4ex{
  {} & a \ar[dl]_{\gamma} \ar[dr]^f & {} \\
  x & {} & y }
\]
uniquely factors as $f\gamma^{\dagger}$.

The fundamental property of the operator $(\ \cdot\ )^\dagger$ is the following:

\begin{lemma}\label{lem:c0dagger}
Let $\mathcal{A}$ be a category and $\mathcal{A}_0$ a subcategory with the condition \ref{cond:star}.
Then every morphism $\gamma$ in $\mathcal{A}_0$ is a split monomorphism, and $\gamma^\dagger$ is a split epimorphism in $\mathrm{Span}(\mathcal{A},\mathcal{A}_0)$.
More precisely, we have $\gamma^\dagger\gamma = \mathrm{id}$.
\end{lemma}
\begin{proof}
We have the following diagram:
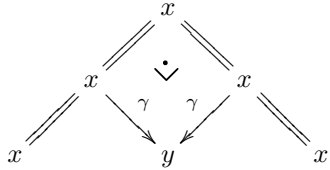
\[
\xymatrix@=4ex{
  {} & {} & x \ar@{=}[dl] \ar@{=}[dr] \ar@{}[dd]|(.4){\rotatebox{-45}{\pbcorner}} & {} & {} \\
  {} & x \ar@{=}[dl] \ar[dr]^{\gamma} & {} & x \ar[dl]_{\gamma} \ar@{=}[dr] & {} \\
  x & {} & y & {} & x }
\]
This implies that $\gamma^\dagger\gamma$ is the identity.
\end{proof}

The uniqueness of the factorization implies that the inclusion $\mathcal{A}\hookrightarrow\mathrm{Span}(\mathcal{A},\mathcal{A}_0)$ preserves monomorphisms.
More generally, we have the following:

\begin{prop}\label{prop:monocrit}
Let $\mathcal{A}$ be a category and $\mathcal{A}_0$ a subcategory with condition \ref{cond:star}.
Then a morphism in $\mathrm{Span}(\mathcal{A},\mathcal{A}_0)$ is a monomorphism if and only if it is a monomorphism in $\mathcal{A}$.
\end{prop}
\begin{proof}
Suppose $\alpha$ is a monomorphism in $\mathrm{Span}(\mathcal{A},\mathcal{A}_0)$, say $\alpha=f\gamma^\dagger$ is a unique factorization.
Then, by lemma \ref{lem:c0dagger}, we have $\alpha\gamma\gamma^\dagger=f\gamma^\dagger\gamma\gamma^\dagger=f\gamma^\dagger$.
Since $\alpha$ is a monomorphism, we obtain $\gamma\gamma^\dagger=\mathrm{id}$, which, by the uniqueness of factorization, implies that $\gamma=\mathrm{id}$.
Thus $\alpha=f$, and $f$ is monic in $\mathrm{Span}(\mathcal{A},\mathcal{A}_0)$.
The result now follows immediately.

The converse is obvious from the uniqueness of the $(\mathcal{A}_0^\dagger,\mathcal{A})$-factorization.
\end{proof}

\subsection{Confluentness}
The category $\mathrm{Span}(\mathcal{A},\mathcal{A}_0)$ has a very good property on pullback squares in $\mathcal{A}_0$.

\begin{prop}\label{prop:c0dagpo}
Let $\mathcal{A}$ be a category and $\mathcal{A}_0$ a subcategory with the condition \ref{cond:star}.
Suppose that we have a pullback square
\begin{equation}
\label{eq:c0pb}
\begin{gathered}
\xymatrix{
  a \ar[r]^{\eta_1} \ar[d]_{\eta_2} \ar@{}[dr]|(.4){\pbcorner} & a_1 \ar[d]^{\delta_1} \\
  a_2 \ar[r]^{\delta_2} & a_0 }
\end{gathered}
\end{equation}
in $\mathcal{A}_0$.
Then the following square in $\mathrm{Span}(\mathcal{A},\mathcal{A}_0)$ is an absolute pushout:
\begin{equation}
\label{eq:c0dagpo}
\begin{gathered}
\xymatrix{
  a_0 \ar[r]^{\delta_1^\dagger} \ar[d]_{\delta_2^\dagger} \ar@{}[dr]|(.6){\pocorner} & a_1 \ar[d]^{\eta_1^\dagger} \\
  a_2 \ar[r]^{\eta_2^\dagger} & c }
\end{gathered}
\end{equation}
\end{prop}
\begin{proof}
Note first that since \eqref{eq:c0pb} is a pullback, we have $\delta_1^\dagger\delta_2=\eta_1\eta_2^\dagger$ and $\delta_2^\dagger\delta_1=\eta_2\eta_1^\dagger$.
Hence the square \eqref{eq:c0dagpo} is commutative.

We now check the universal property of pushouts in $\mathrm{Span}(\mathcal{A},\mathcal{A}_0)^\wedge$.
Suppose we have a commutative square
\[
\xymatrix{
  a_0 \ar[r]^{\delta_1^\dagger} \ar[d]_{\delta_2^\dagger} & a_1 \ar[d]^{f_1} \\
  a_2 \ar[r]^{f_2} & X }
\]
with $X\in\mathrm{Span}(\mathcal{A},\mathcal{A}_0)^\wedge$.
Then by proposition \ref{prop:monocrit}, we have
\[
f_1\eta_1 = f_1\delta_1^\dagger\delta_1\eta_1 = f_2\delta_2^\dagger\delta_2\eta_2 =  f_2\eta_2.
\]
Set $f:=f_1\eta_1=f_2\eta_2:a\to X$.
We obtain
\[
f_1
= f_1\delta_1^\dagger\delta_1
= f_2\delta_2^\dagger\delta_1
= f_2\eta_2\eta_1^\dagger
= f\eta_1^\dagger,
\]
and
\[
f_2
= f_2\delta_2^\dagger\delta_2
= f_1\delta_1^\dagger\delta_2
= f_1\eta_1\eta_2^\dagger
= f\eta_2^\dagger
\]
Since $\eta_1^\dagger$ and $\eta_2^\dagger$ are epimorphisms by proposition \ref{prop:monocrit}, such $f$ is unique.
Therefore \eqref{eq:c0dagpo} is a pushout square in $\mathrm{Span}(\mathcal{A},\mathcal{A}_0)^\wedge$.
This implies \eqref{eq:c0dagpo} is an absolute pushout in $\mathrm{Span}(\mathcal{A},\mathcal{A}_0)$.
\end{proof}

\begin{defin}
Let $\mathcal{A}_0$ be a wide subcategory of a small category $\mathcal{A}$ satisfying the condition \ref{cond:star}.
Then a class $\mathcal{M}$ of morphisms in $\mathcal{A}$ is said to be $\mathcal{A}_0$-stable if whenever we have a pullback square
\[
\xymatrix{
  a'_1 \ar[d]_{f_1} \ar[r]^{\gamma'} \ar@{}[dr]|(.4){\pbcorner} & a' \ar[d]^{f} \\
  a_1 \ar[r]^{\gamma} & a }
\]
with $\gamma,\gamma'\in\mathcal{A}_0$ and $f\in\mathcal{M}$, we have $f_1\in\mathcal{M}$.
\end{defin}

Let $\mathcal{A}'$ be a subcategory of $\mathcal{A}$ which is $\mathcal{A}_0$-stable as a class of morphisms.
We denote by $\mathrm{Span}(\mathcal{A}',\mathcal{A}_0)$ the class of morphisms of the form $f\gamma^\dagger$ with $\gamma\in\mathcal{A}_0$ and $f\in\mathcal{A}'$.
Then $\mathrm{Span}(\mathcal{A}',\mathcal{A}_0)$ clearly forms a subcategory of $\mathrm{Span}(\mathcal{A},\mathcal{A}_0)$.

\begin{corol}
Let $\mathcal{A}$ be a small category with a confluent degeneracy system $\mathcal{A}^-$.
Let $\mathcal{A}^-$ be a confluent degeneracy system on $\mathcal{A}$ and $\mathcal{A}_0$ be a wide subcategory of $\mathcal{A}$ satisfying the condition \ref{cond:star}.
Suppose the following hold:
\begin{enumerate}[label={\rm(\roman*)}]
  \item $\mathcal{A}^-$ is $\mathcal{A}_0$-stable, and $\mathcal{A}_0$ consists of non-degenerate monomorphisms.
  \item If we have morphisms $a_1\xrightarrow{\gamma}a_0\xrightarrow{\sigma}a_2$ with $\sigma\in\mathcal{A}^-$ and $\gamma\in\mathcal{A}_0$, then there is an absolute pushout diagram
\[
\xymatrix{
  a_0 \ar[r]^{\gamma^\dagger} \ar[d]_{\sigma} & a_1 \ar[d]^{f} \\
  a_2 \ar[r]~{g} & a }
\]
with $f,g\in\mathrm{Span}(\mathcal{A}^-,\mathcal{A}_0)$.
\end{enumerate}
Then the wide subcategory $\mathrm{Span}(\mathcal{A}^-,\mathcal{A}_0)$ is a confluent degeneracy system on $\mathrm{Span}(\mathcal{A},\mathcal{A}_0)$.
\end{corol}
\begin{proof}
Let $\deg:\mathcal{A}^-\to\mathbb{N}^\opposite$ be the degree function of the confluent degeneracy system $\mathcal{A}^-$.
Since every morphisms in $\mathcal{A}_0$ is a non-degenerate monomorphism by the assumption, it does not lower the degree.
Hence the degree function naturally extends to $\deg:\mathcal{A}^-\mathcal{A}_0^\dagger\to\mathbb{N}^\opposite$.
Then the condition \ref{cond:CDSisom} is obvious.

Now notice that we have $\mathcal{A}_0^\dagger\pitchfork\mathcal{A}$ in $\mathrm{Span}(\mathcal{A},\mathcal{A}_0)$, since there is the unique $(\mathcal{A}_0,\mathcal{A})$-factorization.
By proposition \ref{prop:monocrit}, this implies that the class of non-degenerate monomorphisms in $\mathrm{Span}(\mathcal{A},\mathcal{A}_0)$ with respect to $\mathcal{A}^-\mathcal{A}_0^\dagger$ coincides with that in $\mathcal{A}$ with respect to $\mathcal{A}^-$.
Thus the conditions \ref{cond:CDSorth} and \ref{cond:CDSfact} follow.

Finally, the condition \ref{cond:CDSconf} directly follows from the assumption and proposition \ref{prop:c0dagpo}.
\end{proof}

\subsection{Spans relative to thin-powered structures}
A typical example of $\mathcal{A}_0\subset\mathcal{A}$ which satisfies the condition \ref{cond:star} is a semicomplete thin-powered structure.
Let $\mathcal{R}$ be a thin-powered category.
Then lemma \ref{lem:thinpow-pb} implies that $\mathcal{R}^+_0$ satisfies the condition \ref{cond:star}.
So we can consider the category $\mathrm{Span}(\mathcal{R},\mathcal{R}^+_0)$, which we denote by $\mathcal{V}(\mathcal{R})$ by abusing the notion.

Recall that a Boolean thin-powered structure $\mathcal{R}^+_0$ on $\mathcal{R}$ gives rise to a join-semilattice representation $\mathcal{R}^+_0\!/(\,\cdot\,):\mathcal{R}\to\mathbf{SemiLat}^\vee$.
This representation extends to the category $\mathcal{V}(\mathcal{R})$ under certain assumptions.
In fact, we have the following stronger result:

\begin{prop}\label{prop:span-thinpow}
Let $\mathcal{R}^+_0$ be a stable Boolean thin-powered structure on $\mathcal{R}$.
Then $\mathcal{R}^+_0$ is also a stable Boolean thin-powered structure on the category $\mathcal{V}(\mathcal{R})$.
\end{prop}
\begin{proof}
First we show that $\mathcal{R}^+_0\subset\mathcal{V}(\mathcal{R})$ is a thin-powered structure.
The conditions \ref{cond:DImono} and \ref{cond:DIcomp} are obvious, so we have to show \ref{cond:DIfact}.
Notice that the stability of the thin-powered category $\mathcal{R}$ implies that the class $\mathcal{V}(\mathcal{R})^-$ of morphisms which are compositions of $(\mathcal{R}^+_0)^\dagger$ and $\mathcal{R}^-$ is closed under compositions.
Moreover, it is obvious that every morphism in $\mathcal{V}(\mathcal{R})$ is uniquely written in the form $\delta\sigma\gamma^\dagger$ with $\delta,\gamma\in\mathcal{R}^+_0$ and $\sigma\in\mathcal{R}^-$.
The uniqueness of the factorization implies $\mathcal{V}(\mathcal{R})^-\pitchfork\mathcal{R}^+_0$, so that we obtain \ref{cond:DIfact}.

We next verify the stability.
In particular, it suffices to construct pullbacks for diagrams of the form $s\xrightarrow{\delta}r\xleftarrow{\gamma^\dagger}r'$ with $\delta,\gamma\in\mathcal{R}^+_0$.
Let $\delta':=(\gamma\delta)\vee(\neg\gamma)\in\mathcal{R}^+_0\!/r'$, say $\delta':s'\to r'$.
Since $\gamma\delta\le\delta'$ in the poset $\mathcal{R}^+_0\!/r'$, there is a morphism $\varepsilon:s\to s'\in\mathcal{R}^+_0$ such that $\gamma\delta=\delta'\varepsilon$.
Then the following diagram is commutative:
\begin{equation}
\label{eq:span-DIpull}
\xymatrix{
  s' \ar[r]^{\delta'} \ar[d]_{\varepsilon^\dagger} & r' \ar[d]^{\gamma^\dagger} \\
  s \ar[r]^{\delta} & r }
\end{equation}
Indeed, we have $\gamma\wedge((\gamma\delta)\vee(\neg\gamma))=\gamma\delta$.
We show the square is in fact a pullback.

Suppose we have a diagram
\[
\xymatrix{
  t \ar[r]^{f\xi^\dagger} \ar[d]_{g\theta^\dagger} & r' \ar[d]^{\gamma^\dagger} \\
  s \ar[r]^{\delta} & r }
\]
with $f,g\in\mathcal{R}$ and $\xi,\theta\in\mathcal{R}^+_0$.
By the definition of the composition $\gamma^\delta f$ and the uniqueness of the factorization, we have $\image(f)\wedge\gamma\le\gamma\delta$, which implies $\image(f)\le(\gamma\delta)\vee(\neg\gamma)=\delta'$.
Hence there is a morphism $\mu:t\to s'$ such that $f=\delta'\mu$.
On the other hand, we have
\[
g\theta^\dagger
= \delta^\dagger\delta g\theta^\dagger
= \delta^\dagger\gamma^\dagger f\xi^\dagger
= \varepsilon^\dagger(\delta')^\dagger\delta'\mu\xi^\dagger
= \varepsilon^\dagger\mu\xi^\dagger\,.
\]
Now, we obtain the following commutative diagram:
\[
\xymatrix@R-2ex{
  {} && r' \ar[dr]^{\gamma^\dagger} & {} \\
  t \ar@/^10pt/[urr]^{f\xi^\dagger} \ar@/_10pt/[drr]_{g\theta^\dagger} \ar[r]^{\mu\xi^\dagger} & s' \ar[ur]^{\delta'} \ar[dr]_{\varepsilon^\dagger} && r \\
  {} && s \ar[ur]_{\delta} & {} }
\]
Moreover, since $\delta'$ is a monomorphism by proposition \ref{prop:monocrit}, such morphism $\mu\xi^\dagger$ is unique.
Therefore, the square \eqref{eq:span-DIpull} is a pullback, and the stability follows.
Finally, the Booleanness directly follows.
\end{proof}

\begin{corol}\label{cor:V-repn}
Let $\mathcal{R}$ and $\mathcal{R}^+_0$ be as in proposition \ref{prop:span-thinpow}.
Then the representation $\mathcal{R}^+_0\!/(\blankdot):\mathcal{R}\to\mathbf{SemiLat}^\vee$ extends to a representation $\mathcal{V}(\mathcal{R})\to\mathbf{SemiLat}^\vee$ by setting $(\gamma^\dagger)_\ast:=\gamma^\ast:\mathcal{R}^+_0\!/r\to\mathcal{R}^+_0\!/r$ for a $\mathcal{R}^+_0$-morphism $\gamma:s\to r$.
\end{corol}
\begin{proof}
By proposition \ref{prop:span-thinpow}, the extension actually exists as $\mathcal{R}^+_0\!/(\blankdot)=(\mathcal{V}(\mathcal{R}))^+_0\!/(\blankdot)$.
So we have to compute $(\gamma^\dagger)_\ast:\mathcal{R}^+_0\!/r' \to \mathcal{R}^+_0\!/r$ for $\gamma:r\to r'\in\mathcal{R}^+_0$.
But, for $\delta\in\mathcal{R}^+_0\!/r'$, we have $\gamma^\dagger\delta=(\gamma^\ast\delta)(\delta^\ast\gamma)^\dagger$.
Hence we obtain $(\gamma^\dagger)_\ast=\gamma^\ast$ as required.
\end{proof}

By proposition \ref{prop:span-thinpow}, $\mathcal{V}(\mathcal{R})$ is naturally a thin-powered category.
We use the same notation $\mathcal{R}^+_0$ for it.
So we write $\mathcal{R}^+_0\!/(\blankdot):\mathcal{V}(\mathcal{R})\to\mathbf{SemiLat}^\vee$ for the induced representation.

Next, we discuss the confluent property of thin-powered categories.
We fix a thin-powered category $\mathcal{R}$.
By abusing notion, we denote by $\mathcal{V}[r]$ the representable presheaf $\mathcal{V}(\mathcal{R})(\blankdot,r)\in\mathcal{V}(\mathcal{R})^\wedge$ for $r\in\mathcal{V}(\mathcal{R})$.

\begin{lemma}\label{lem:Vdag-fact}
Let $\mathcal{R}$ be a locally finite confluent stable thin-powered category.
Then for every cell $x:\mathcal{V}[r]\to X$, there is a unique factorization $x=x'\gamma^\dagger$ such that $\gamma\in\mathcal{R}^+_0$ and $x'$ does not factor through any non-trivial $(\mathcal{R}^+_0)^\dagger$-morphisms in $\mathcal{V}(\mathcal{R})^\wedge$.
\end{lemma}
\begin{proof}
Let $x:\mathcal{V}(r)\to X$ be a cell of $X\in\mathcal{V}(\mathcal{R})^\wedge$.
Note that by proposition \ref{prop:span-thinpow}, $\mathcal{R}^+_0$ is also a locally finite stable thin-powered structure on $\mathcal{V}(\mathcal{R})$.
Hence we have the degree function $\deg:\mathcal{V}(\mathcal{R})\to\mathbb{N}$.
So we can choose a factorization $\mathcal{V}[r]\xrightarrow{\gamma^\dagger}\mathcal{V}[r']\xrightarrow{x'} X$ of $x$ with $\deg r'$ smallest among such factorizations.
We show it is actually the required factorization.

First, $x'$ does not factor through any non-trivial $\mathcal{R}^+_0$.
Indeed, say $x'=x''\delta^\dagger$ for some $\delta:r''\to r'\in\mathcal{R}^+_0$.
Then we have $\deg r''\le\deg r'$, and the minimality of $r'$ implies $\deg r''\to \deg r'$.
Since every non-trivial $\mathcal{R}^+_0$-morphism raises the degree, $\delta$ needs to be the identity, and we obtain $x'=x''$.

Now suppose $x=x'_1\gamma_1^\dagger$ is another factorization such that $\gamma_1:r'_1\to r\in\mathcal{R}^+_0$ and $x'_1$ does not factor through any non-trivial $(\mathcal{R}^+_0)^\dagger$-morhism.
Then we have the following commutative diagram:
\[
\xymatrix{
  \mathcal{V}[r] \ar[r]^{\gamma^\dagger} \ar[d]_{\gamma_1^\dagger} & \mathcal{V}(r') \ar[d]^{x'} \\
  \mathcal{V}[r'_1] \ar[r]^{x'_1} & X }
\]
By proposition \ref{prop:c0dagpo}, there are $r'\xleftarrow{\delta}r_0\xrightarrow{\delta_1} r'_1\in\mathcal{R}^+_0$ and $x_0:\mathcal{V}[r_0]\to X$ such that $\delta^\dagger\gamma^\dagger=\delta_1^\dagger\gamma_1^\dagger$, $x'=x_0\delta^\dagger$ and $x'_1=x_0\delta_1^\dagger$.
Since $x'$ and $x'_1$ do not factor through any non-trivial $(\mathcal{R}^+_0)^\dagger$-morphism, $\delta$ and $\delta_1$ need to be the identities.
Hence we obtain $\delta=\delta_1$ and $x'=x_0=x'_1$.
It follows that the factorization is unique.
\end{proof}

\begin{lemma}\label{lem:posplitdagger}
Let $\mathcal{R}$ be a locally finite confluent coherent thin-powered category.
Suppose $\gamma:r_1\to r_0\in\mathcal{R}^+_0$ and $\sigma:r_0\to r_2$ is $\mathcal{R}^+_0$-surjective.
By lemma \ref{lem:maxsatur}, there is the greatest saturated element $\gamma_0\in\mathcal{R}^+_0\!/r_0$ with respect to $\sigma$ with $\gamma_0\le\gamma$, say $\gamma_0=\gamma\varepsilon_0$ and $\sigma\gamma_0=\gamma'\sigma'$ with $\gamma'\in\mathcal{R}^+_0$ and $\sigma'\in\mathcal{R}^-$.
Then the square
\[
\xymatrix{
  r_0 \ar[r]^{\gamma^\dagger} \ar[d]_{\sigma} \ar@{}[dr]|(.4){\pbcorner} & r_1 \ar[d]^{\sigma'\varepsilon_0^\dagger} \\
  r_2 \ar[r]^{\gamma'^\dagger} & r}
\]
is an absolute pushout in $\mathcal{V}(\mathcal{R})$.
\end{lemma}
\begin{proof}
First note that the square is commutative since a square
\[
\xymatrix{
  s_0 \ar[r]^{\gamma_0=\gamma\varepsilon_0} \ar[d]_{\sigma'} \ar@{}[dr]|(.4){\pbcorner} & r_0 \ar[d]^{\sigma} \\
  r \ar[r]^{\gamma'} & r_2 }
\]
is a pullback in $\mathcal{R}$.
Suppose $X\in\mathcal{V}(\mathcal{R})^\wedge$ and we have a commutative square
\[
\xymatrix{
  \mathcal{V}[r_0] \ar[r]^{\gamma^\dagger} \ar[d]_{\sigma} & \mathcal{V}[r_1] \ar[d]^{x_1} \\
  \mathcal{V}[r_2] \ar[r]^{x_2} & X }
\]
in $\mathcal{V}(\mathcal{R})$.
By lemma \ref{lem:Vdag-fact}, there are unique factorization $x_1=x'_1\delta_1^\dagger$ and $x_2=x'_2\delta_2^\dagger$ such that  $\delta_1:s_1\to r_1,\delta_2:s_2\to r_2\in\mathcal{R}^+_0$ and $x'_1,x'_2$ do not factor through any non-trivial $\mathcal{R}^+_0$-morphism.
Taking a pullback, we obtain the pullback square:
\[
\xymatrix{
  s_0 \ar[r]^{\xi} \ar[d]_{\mu} \ar@{}[dr]|(.4){\pbcorner} & r_0 \ar[d]^{\sigma} \\
  s_2 \ar[r]^{\delta_2} & r_2 }
\]
Then we have $x'_2\mu\xi^\dagger = x'_2\delta_2^\dagger\sigma = x'_1\delta_1^\dagger\gamma^\dagger$.

Now since $\sigma$ is a split epimorphism, $\mu$ is a split epimorphism, so that it is an $\mathcal{R}^+_0$-surjective.
It follows that $\sigma_{\ast}(\eta)=\delta_2$ and $\eta$ is saturated with respect to $\sigma$.
Thus by the maximality of $\gamma_0$, we have $\eta\le\gamma_0=\gamma\varepsilon_0$ and there is a diagram
\[
\xymatrix{
  t_1 \ar[r]^{\varepsilon_1} \ar[d]_{\mu} \ar@{}[dr]|(.4){\pbcorner} & s_0 \ar[r]^{\gamma_0=\gamma\varepsilon_0} \ar[d]^{\sigma'} \ar@{}[dr]|(.4){\pbcorner} & r \ar[d]^{\sigma} \\
  t_2 \ar[r]^{\varepsilon_2} & s' \ar[r]^{\gamma'} & r' }
\]
with $\gamma_0\varepsilon_1=\eta$ and $\gamma'\varepsilon_2=\delta_2$.
Since $\gamma_0=\gamma\varepsilon_0$, $\eta=\gamma\delta_1$ and $\gamma$ is a monoomrphism, we have $\delta_1=\varepsilon_0\varepsilon_1$.
Let $f=f'_2\varepsilon_2^\dagger$
Finally we obtain
\[
f_1
= f'_1\delta_1^\dagger
= f'_2\mu\delta_1^\dagger
= f'_2\mu\varepsilon_1^\dagger\varepsilon_0^\dagger
= f'_2\varepsilon_2^\dagger\sigma'\varepsilon_0^\dagger
= f\sigma'\varepsilon_0^\dagger,
\]
and
\[
f_2
= f'_2\delta_2^\dagger
= f'_2\varepsilon_2^\dagger\gamma'^\dagger
= f\gamma'^\dagger.
\]
These implies that we have a commutative diagram:
\[
\xymatrix{
  r \ar[r]^{\gamma^\dagger} \ar[d]_{\sigma} \ar@{}[dr]|(.4){\pbcorner} & s \ar[d]^{\sigma'\varepsilon_0^\dagger} \ar@/^1.5pc/[ddr]^{f_1} & {} \\
  r' \ar[r]^{\gamma'^\dagger} \ar@/_/[drr]_{f_2} & r' \ar[dr]^<<{f} & {} \\
  {} & {} & X }
\]
Since $\sigma'\varepsilon_0^\dagger$ and $\gamma'^\dagger$ are epimorphisms, the uniqueness of such $f$ is obvious.
Therefore the square is an absolute pushout.
\end{proof}

\section{Cubicalizations}
\label{sec:crsrep}

\subsection{Crossed modules for categories}
The notion of crossed modules was originally introduced by Whitehead (\cite{Whi}) in the study of relative homotopy $2$-types in the homotopy theory.
Brown and Spencer introduced a little bit generalized version of crossed modules (\cite{BS} and \cite{BHS}).
They showed that the category of crossed modules is a model for the $2$-homotopy theory.

We use this tool to produce a new category $\square(\mathcal{R})$ from a thin-powered category $\mathcal{R}$ with some properties.
To do this, we need to generalize the definition of crossed modules for groupoids (see \cite{BHS}) to arbitrary categories.

\begin{defin}
Let $\mathcal{J}$ be a small category.
A crossed $\mathcal{J}$-module is a functor $M:\mathcal{J}\to\mathbf{Mon}$ from $\mathcal{J}$ to the category $\mathbf{Mon}$ of monoids together with a family of monoid homomorphisms
\[
\left\{\mu_j:M(j)\to\mathrm{End}_{\mathcal{J}}(j)\,|\,j\in\mathcal{J}\right\}
\]
which satisfies the following conditions:
For $f:j\to k$ of $\mathcal{J}$ and $a,b\in M(j)$, we have
\begin{gather}
\mu_j(f_{\ast}a) f \mu_j(a) = \mu_j(f_{\ast}a)^2 f \tag{CM1}\label{cond:cmtrans}\\
a b = (\mu_j(a)_{\ast}b)a \tag{CM2}\label{cond:cmbraid}
\end{gather}
We often abuse the notion and say $M$ is a crossed $\mathcal{J}$-module, and write $\mu=\mu_j$ briefly.
\end{defin}

We warn that the notion of crossed modules differs from that of crossed groups defined in section \ref{sec:crsgrp}.
They have very similar names and definitions, so that they are really subject to be confused.

\begin{exam}\label{exam:grpd}
If $P$ is a groupoid and $M:P\to\mathbf{Mon}$ is a crossed $P$-module such that each $M(p)$ is a group, then $M$ is a crossed module over $P$ in the sense of \cite{BS}.
This example is important in the study of homotopy 2-types.
\end{exam}

\begin{exam}\label{exam:subset}
Let $\kappa$ be a cardinal and denote by $\widetilde{\mathbf{Set}^\kappa}$ the category whose objects are $\kappa$-small sets and whose morphisms $(f,A):X\to Y$ are pairs of subset $A\subset X$ and map $f:A\to Y$.
The composition is given by the formula $(g,B)\circ (f,A)=\left(gf,f^{-1}(B)\right)$.
For a set $X$, we denote by $P(X)$ the power set of $X$, hence $P(X)$ is an idempotent monoid with operation $\cup$.
Then $P:\widetilde{\mathbf{Set}^\kappa}\to\mathbf{Mon}$ is a crossed $\widetilde{\mathbf{Set}^\kappa}$-module.
Indeed, for $A\in P(X)$, we set $\mu(A)=(\mathrm{id},X\setminus A)$.
For $(f,A'):X\to Y$ of $\widetilde{\mathbf{Set}^\kappa}$, we have
\begin{equation*}
\begin{split}
\mu(f(A\cap A'))\circ(f,A')
&= \left(f,A'\cap f^{-1}(f(A')\setminus f(A\cap A'))\right) \\
&= \left(f,A'\cap f^{-1}(f(A'\setminus A)\setminus f(A\cap A'))\right)\\
&= \mu(f(A\cap A'))\circ(f,A')\circ\mu(A)
\end{split}
\end{equation*}
This implies \eqref{cond:cmtrans}.
Note that in this example, $\mu(A)_{\ast}(A')=A'\setminus A$, so that \eqref{cond:cmbraid} is equivalent to $(A'\setminus A)\cup A = A'\cup A$, which is obvious.

This example is a special case of more general examples we will see later.
\end{exam}

Now let $M$ be a crossed $\mathcal{J}$-module.
For $a\in M(j)$, we set
\[
M(j)^a = \left\{x\in M(j)\,\mid\,\mu(a)_{\ast}x = x\right\}.
\]
Clearly $M(j)^a$ is a submonoid of $M(j)$.

\begin{lemma}\label{lem:invcomm}
If $M$ is a crossed $\mathcal{J}$-module for a small category $\mathcal{J}$, and $j\in\mathcal{J}$, $a\in M(j)$, then $a$ commutes with all elements of $M(j)^a$.
\end{lemma}
\begin{proof}
For each $x\in M(j)^a$, by \eqref{cond:cmbraid}, we have
\[
ax = \mu(a)_{\ast}x\cdot a = xa.
\]
Hence the result follows.
\end{proof}

\begin{corol}\label{lem:funccomm}
Let $M:\mathcal{J}\to\mathbf{Mon}$ is a crossed $\mathcal{J}$-module for a small category $\mathcal{J}$.
Suppose that $f:j\to k\in\mathcal{J}$ and $b\in M(k)$ satisfy $\mu(b)f=f$.
Then $b$ commutes with $f_{\ast}a$ for every $a\in M(j)$; i.e. $b (f_{\ast}a) = (f_{\ast}a)b$.
\end{corol}
\begin{proof}
Since $\mu(b)_{\ast}f_{\ast}a = (\mu(b)f)_{\ast}a = f_{\ast}a$, we have $f_{\ast}a\in M(j)^b$.
Hence the result follows from lemma \ref{lem:invcomm}.
\end{proof}

Next, we define a new category from a crossed module.
Let $\mathcal{J}$ be a small category and $M:\mathcal{J}\to\mathbf{Mon}$ be a crossed module.
We define a set $M\!\mathcal{J}(j,k)$ for a crossed $\mathcal{J}$-module $M$ and $j,k\in\mathcal{J}$ by
\begin{equation}
\label{eq:crsmod-inv-pairs}
M\!\mathcal{J}(j,k) := \left\{(a,f)\in M(k)\times\mathcal{J}(j,k)\mid \mu(a)f=f\right\}
\end{equation}
We further define a ``composition map'' $\circ:M\!\mathcal{J}(k,l)\times M\!\mathcal{J}(j,k) \to M\!\mathcal{J}(j,l)$ by the formula
\begin{equation}
\label{eq:crsmod-inv-composition}
(b,g)\circ (a,f) = \left(b\!\cdot\!g_{\ast}a, \mu(g_{\ast}a)gf\right).
\end{equation}
Note that $(b,g)\circ (a,f)$ in fact belongs to $M\!\mathcal{J}(j,l)$.
Indeed, using lemma \ref{lem:funccomm} and \eqref{cond:cmtrans}, we have
\begin{equation}
\label{eq:stable}
\begin{split}
\mu(b\!\cdot\!g_{\ast}a)\mu(g_{\ast}a)gf
&= \mu(g_{\ast}a)^2\mu(b)gf \\
&= \mu(g_{\ast}a)g\mu(a)f\\
&= \mu(g_{\ast}a)gf.
\end{split}
\end{equation}

\begin{prop}
The composition $\circ$ defined by \eqref{eq:crsmod-inv-composition} is assosiative.
\end{prop}
\begin{proof}
Suppose $(a,f)\in M\!\mathcal{J}(j,k)$, $(b,g)\in M\!\mathcal{J}(k,l)$ and $(c,h)\in M\!\mathcal{J}(l,m)$.
Then we have
\begin{equation*}
\begin{split}
(c,h)\circ\left((b,g)\circ(a,f)\right)
&= (c,h)\circ(g_{\ast}a\cdot b, \mu(g_{\ast}a)gf) \\
&= \left(h_{\ast}g_{\ast}a\cdot h_{\ast}b \cdot c, \mu(h_{\ast}g_{\ast}a \cdot h_{\ast}b)h\mu(g_{\ast}a)gf\right).
\end{split}
\end{equation*}
Since $\mu(c)h=h$ and $\mu(b)g=g$, using the calculus \eqref{eq:stable} and lemma \ref{lem:funccomm}, we obtain
\begin{equation*}
\begin{split}
\mu(h_{\ast}g_{\ast}a \cdot h_{\ast}b)h\mu(g_{\ast}a)gf
&= \left(\mu(h_{\ast}g_{\ast}a\cdot h_{\ast}b\right)^2 h\mu(g_{\ast}a)\mu(b)gf \\
&= \mu(h_{\ast}g_{\ast}a\cdot h_{\ast}b)hgf.
\end{split}
\end{equation*}
This implies that
\begin{equation}
\label{eq:rassoc}
(c,h)\circ\left((b,g)\circ(a,f)\right) = \left( h_{\ast}g_{\ast} a\cdot h_{\ast}b\cdot c, \mu(h_{\ast}g_{\ast}a\cdot h_{\ast} b)hgf\right)
\end{equation}
On the other hand, we have
\begin{equation*}
\begin{split}
\left((c,h)\circ(b,g)\right)\circ(a,f)
&= (h_{\ast}b\cdot c, \mu(h_{\ast}b)hg)\circ (a,f) \\
&= \left(\left(\mu(h_{\ast}b)hg\right)_{\ast}a\cdot h_{\ast}b\cdot c, \mu\left(\left(\mu(h_{\ast}b)hg\right)_{\ast}a\right)\mu(h_{\ast}b)hgf\right).
\end{split}
\end{equation*}
By \eqref{cond:cmbraid} and lemma \ref{lem:funccomm}
\[
\left(\mu(h_{\ast}b)hg\right)_{\ast}a\cdot h_{\ast}b = h_{\ast}b\cdot h_{\ast}g_{\ast} a = h_{\ast}g_{\ast}a\cdot h_{\ast} b
\]
Therefore, comparing with \eqref{eq:rassoc}, we obtain
\begin{equation*}
\begin{split}
\left((c,h)\circ(b,g)\right)\circ(a,f)
&= \left( h_{\ast}g_{\ast} a\cdot h_{\ast}b\cdot c, \mu(h_{\ast}g_{\ast}a\cdot h_{\ast} b)hgf\right) \\
&= (c,h)\circ\left((b,g)\circ(a,f)\right).
\end{split}
\end{equation*}
\end{proof}

\begin{corol}\label{cor:crsmod-inv-cat}
Let $M:\mathcal{J}\to\mathbf{Mon}$ be a crossed $\mathcal{J}$-module for a small category $\mathcal{J}$.
Then there is a category whose objects are those of $\mathcal{J}$, whose hom-sets are $M\!\mathcal{J}(j,k)$ defined by \eqref{eq:crsmod-inv-pairs} and whose composition is the map $\circ$ defined by \eqref{eq:crsmod-inv-composition}.
We denote the category by $M\!\mathcal{J}$.
\end{corol}

The identity on $j\in M\!\mathcal{J}$ of the category $M\!\mathcal{J}$ is $(1_j,\mathrm{id}_j)$ where $1_j\in M(j)$ is the unit.
In general, if $f\in\mathcal{J}(j,k)$ and $g\in\mathcal{J}(k,l)$, we have $(1_l,g)\circ(1_k,f)=(1_l,gf)$.
Hence $M\!\mathcal{J}$ contains $\mathcal{J}$ as its wide subcategory.

Note that if $f:j\to k$ in $\mathcal{J}$ is an epimorphism, the equation $\mu(m)f=f$ implies $\mu(m)=\mathrm{id}$.
Thus if every morphism in $\mathcal{J}$ is an epimorphism, set $M_0(j)=\mu^{-1}(\mathrm{id})\subset M(j)$, and we have
\[
M\!\mathcal{J}(j,k) = M_0(k)\times\mathcal{J}(j,k).
\]
In the case, the composition is given by $(n,g)\circ(m,f)=(g_{\ast}m\cdot n,gf)$.
This is to say that the category $M\!\mathcal{J}$ equals to the opposite of the total category $(\mathcal{J}^{\mathrm{op}}\!M_0)^{\mathrm{op}}$ of a crossed $\mathcal{J}^{\mathrm{op}}$-monoid $M_0$, see \cite{BM}.

\subsection{Cubicalization of thin-powered categories}
We now introduce the notion of cubicalization of small categories.
For this, we need some assumptions for the process.
We here assume the following:

\begin{defin}
A small category $\mathcal{R}$ is said to be cubicalizable if it has an initial object $0$ and is equipped with a locally finite, coherent and Boolean thin-powered structure such that for each object $r\in\mathcal{R}$, the unique morphism $0\to r$ is a distinguished injection.
\end{defin}

Recall that we write $\mathcal{V}(\mathcal{R})=\mathrm{Span}(\mathcal{R},\mathcal{R}^+_0)$, and we have a semilattice representation $\mathcal{R}^+_0\!/(\blankdot):\mathcal{V}(\mathcal{R})\to\mathbf{SemiLat}^\vee$.
Since every semilattice can be seen as a monoid with multiplication $\vee$, we have a representation $\mathcal{V}(\mathcal{R})\to\mathbf{Mon}$.
We give a crossed module structure to this representation.
For $\xi\in\mathcal{R}^+_0\!/r$, we write $\Lambda(\xi):=\xi\xi^\dagger$.
Note that we have
\[
\Lambda(\xi)_\ast=(\xi\xi^\dagger)_\ast=\xi_\ast\xi^\ast=(\blankdot)\wedge\xi.
\]

\begin{lemma}\label{lem:lambdacalc}
Let $\mathcal{R}$ be a cubicalizable small category.
Then the following hold:
\begin{enumerate}[label={\rm(\arabic*)}]
  \item\label{sublem:flambda} If $f:r'\to r$ is a morphism of $\mathcal{R}$, then for every $\xi\in\mathcal{R}^+_0\!/r$, we have the formula
\[
\Lambda(\xi)f = f\Lambda(f^\ast\xi)
\]
in $\mathcal{V}(\mathcal{R})$.
  \item\label{sublem:daglambda} If $\gamma:r'\to r\in\mathcal{R}^+_0$, then for every $\xi\in\mathcal{R}^+_0\!/r$, we have the formula:
\[
\Lambda(\gamma^\ast\xi)\gamma^\dagger = \gamma^\dagger\Lambda(\xi)\,.
\]
  \item\label{sublem:lambdawedge} For each $\xi_1,\xi_2\in\mathcal{R}^+_0\!/r$, we have
\[
\Lambda(\xi_1)\Lambda(\xi_2)=\Lambda(\xi_2)\Lambda(\xi_1)=\Lambda(\xi_1\wedge\xi_2).
\]
\end{enumerate}
\end{lemma}
\begin{proof}
By taking the pullback, we obtain the following diagram in $\mathcal{R}$:
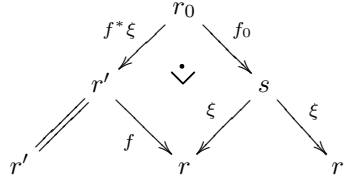
\[
\xymatrix@=4ex{
  {} && r_0 \ar[dl]_{f^\ast\xi} \ar[dr]^{f_0} \ar@{}[dd]|(.4){\rotatebox{-45}{\pbcorner}} && {} \\
  {} & r' \ar@{=}[dl] \ar[dr]_{f} && s \ar[dl]_{\xi} \ar[dr]^{\xi} & {} \\
  r' && r && r }
\]
Hence as morphisms of $\mathcal{V}(\mathcal{R})$, we have
\[
\Lambda(\xi)f = \xi\xi^\dagger f = \xi f_0 (f^\ast(\xi))^\dagger = f f^\ast(\xi) (f^\ast(\xi))^\dagger = f \Lambda(f^\ast(\xi)).
\]
This is \ref{sublem:flambda}.
Similarly, for $\gamma:r'\to r\in\mathcal{R}^+_0$, we have the following diagram in $\mathcal{R}$:
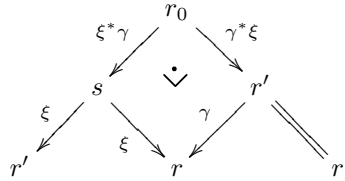
\[
\xymatrix@=4ex{
  {} && r_0 \ar[dl]_{\xi^\ast\gamma} \ar[dr]^{\gamma^\ast\xi} \ar@{}[dd]|(.4){\rotatebox{-45}{\pbcorner}} && {} \\
  {} & s \ar[dl]_{\xi} \ar[dr]_{\xi} && r' \ar[dl]_{\gamma} \ar@{=}[dr] & {} \\
  r' && r && r }
\]
Then \ref{sublem:daglambda} can be verified as follows:
\[
\gamma^\dagger\Lambda(\xi)
= \gamma^\dagger\xi\xi^\dagger
= (\gamma^\ast\xi)(\xi^\ast\gamma)^\dagger\xi^\dagger
= (\gamma^\ast\xi)(\gamma^\ast\xi)^\dagger\gamma^\dagger
= \Lambda(\gamma^\ast\xi)\gamma^\dagger\,.
\]

Now if $\xi_1,\xi_2\in\mathcal{R}^+_0\!/r$, then we have the diagram
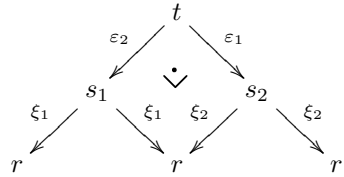
\[
\xymatrix@=4ex{
  {} && t \ar[dl]_{\varepsilon_2} \ar[dr]^{\varepsilon_1} \ar@{}[dd]|(.4){\rotatebox{-45}\pbcorner} && {} \\
  {} & s_1 \ar[dl]_{\xi_1} \ar[dr]^{\xi_1} && s_2 \ar[dl]_{\xi_2} \ar[dr]^{\xi_2} & {} \\
  r && r && r }
\]
where $\xi_1\varepsilon_2=\xi_2\varepsilon_1=\xi_1\wedge\xi_2$.
This implies that $\Lambda(\xi_1\wedge\xi_2)=\Lambda(\xi_2)\Lambda(\xi_1)$.
\end{proof}

\begin{corol}\label{cor:disjointlambda}
Let $\mathcal{R}$ and $\mathcal{R}^+_0$ be as above.
Then for $\xi\in\mathcal{R}^+_0\!/r$ and $f:r'\to r\in\mathcal{R}$, $\mathrm{im}(f)\le\xi$ if and only if $\Lambda(\xi)f=f$.
\end{corol}
\begin{proof}
Note that by the definition of $\Lambda$, we have $\Lambda(\mathrm{im}(f))f=f$.
Hence if $\mathrm{im}(f)\le\xi$, then
\[
\Lambda(\xi)f = \Lambda(\xi\wedge\mathrm{im}(f))f = \Lambda(\mathrm{im}(f))f = f.
\]

Conversely, by lemma \ref{lem:lambdacalc}-\ref{sublem:flambda} we have
\[
\Lambda(\xi)f = f\Lambda(f^\ast\xi) = f (f^\ast(\xi))(f^\ast(\xi))^\dagger.
\]
Because of the uniqueness of the factorization in $\mathcal{V}(\mathcal{R},\mathcal{R}^+_0)$, in order to have $\Lambda(\xi)f=f$, it is necessary that $f^\ast(\xi)=1$.
Since $\mathcal{R}^+_0$ is stable, by corollary \ref{cor:ffmeet}, we have $f_\ast f^\ast(\xi)=\xi\wedge\mathrm{im}(f)$ and $f_\ast(1)=\mathrm{im}(f)$.
Therefore $\Lambda(\xi)f=f$ implies $\xi\wedge\mathrm{im}(f)=\mathrm{im}(f)$, which is equivalent to $\mathrm{im}(f)\le\xi$.
\end{proof}

Consider a cubicalizable small category $\mathcal{R}$.
By lemma \ref{lem:lambdacalc}-\ref{sublem:lambdawedge}, $\Lambda$ preserves the meet operation.
We, however, see $\mathcal{R}^+_0\!/r$ as a monoid with the join operation, hence $\Lambda$ is not necessarily a crossed module structure.
So we consider its complement.
Indeed, since $\mathcal{R}^+_0\!/r$ is Boolean, there is a complement operation $\neg:(\mathcal{R}^+_0\!/r)^\opposite\to\mathcal{R}^+_0\!/r$.
Then lemma \ref{lem:lambdacalc}-\ref{sublem:lambdawedge} implies that we have $\Lambda(\neg(\xi_1\vee\xi_2))=\Lambda(\neg\xi_1)\Lambda(\neg\xi_2)$.
By setting $\Lambda^\complement=\Lambda(\neg(\blankdot))$, we obtain a monoid homomorphism:
\[
\Lambda^\complement:\mathcal{R}^+_0\!/r\to\mathrm{End}_{\mathcal{V}(\mathcal{R})}(r)
\]

\begin{prop}\label{prop:cube-crsmod}
If $\mathcal{R}$ is a cubicalizable category, then the functor $\mathcal{R}^+_0\!/(\blankdot):\mathcal{V}(\mathcal{R})\to\mathbf{Mon}$ and a monoid homomorphism $\Lambda^\complement$ defined above form a structure of crossed $\mathcal{V}(\mathcal{R})$-module.
\end{prop}
\begin{proof}
We have to verify the conditions \eqref{cond:cmtrans} and \eqref{cond:cmbraid}.
Let $f:r\to r'$ and $\xi\in\mathcal{R}^+_0\!/r$.
Since $\neg f_\ast(\xi)\le f_\ast(\neg\xi)$, we have $\Lambda^\complement(f_\ast(\xi))=\Lambda(f_\ast(\neg\xi))$ by lemma \ref{lem:lambdacalc}.
Hence we obtain
\[
\begin{split}
\Lambda^\complement(f_\ast(\xi)) f \Lambda^\complement(\xi)
&= \Lambda^\complement(f_\ast(\xi))\Lambda(f_\ast(\neg\xi)) f \Lambda^\complement(\xi)\\
&= \Lambda^\complement(f_\ast(\xi))f \Lambda(f^\ast f_\ast(\neg\xi)\wedge\neg\xi)\\
&= \Lambda^\complement(f_\ast(\xi))f \Lambda^\complement(\xi)
\end{split}
\]
and \eqref{cond:cmtrans} follows.

On the other hand, for any $\xi_1,\xi_2\in\mathcal{R}^+_0\!/r$, by the distributivity, we have
\[
\begin{split}
\Lambda^\complement(\xi_1)_\ast(\xi_2)\vee\xi_1
&= (\neg\xi_1\wedge\xi_2)\vee\xi_1\\
&= (\neg\xi_1\vee\xi_1)\wedge(\xi_2\vee\xi_1)\\
&= 1\wedge\xi_2\wedge\xi_1\\
&= \xi_1\vee\xi_2.
\end{split}
\]
Hence \eqref{cond:cmbraid} follows.
\end{proof}

As a result, we obtain the new categories.

\begin{prop}
Suppose $\mathcal{R}$ is a cubicalizable category, that is, coherent, locally finite and Boolean thin-powered category.
Then there is a category $\square(\mathcal{R})$ whose objects are those of $\mathcal{R}$ and whose morphisms $r'\to r$ are pairs $(\xi,f\gamma^\dagger)$ of $(\mathcal{R},\mathcal{R}^+_0)$-spans $f\gamma^\dagger$ and $\xi\in\mathcal{R}^+_0\!/r$ with $\image(f)\wedge\xi=0$.
The composition is give by the formula
\[
(\zeta,g\gamma^\dagger)\circ (\xi,f\beta^\dagger) = \left(\zeta\vee g_\ast\xi,(\neg g_\ast\xi)(\neg g_\ast\xi)^\dagger g\gamma^\dagger f\beta^\dagger\right)\,.
\]
We call $\square(\mathcal{R})$ the cubicalization of $\mathcal{R}$.
\end{prop}
\begin{proof}
By proposition \ref{prop:cube-crsmod}, $M:=(\mathcal{R}^+_0\!/(\blankdot),\Lambda^\complement)$ is a crossed module on $\mathcal{V}(\mathcal{R})$.
Then the result follows from corollary \ref{cor:crsmod-inv-cat}.
\end{proof}

Now let $\mathcal{R}$ be a cubicalizable category, and let $\square(\mathcal{R})$ be its cubicalization.
Notice that by corollary \ref{cor:disjointlambda}, for a morphism $f\gamma^\dagger:r'\to r\in\mathcal{V}(\mathcal{R})$ and an element $\xi\in\mathcal{R}^+_0\!/r$, we have $\Lambda^\complement(\xi)f\gamma^\delta=f$ if and only if $\mathrm{im}(f)\le\neg\xi$, which is equivalent to $\mathrm{im}(f)\wedge\gamma=0$.
Hence a morphism $r'\to r$ in $\square(\mathcal{R})$ is represented as a pair of a diagram
\[
\xymatrix{
  {} & s \ar[dl]_{\gamma} \ar[dr]^{f} & {} \\
  r' && r }
\]
and an element $\xi\in\mathcal{R}^+_0\!/r$ such that $\mathrm{im}(f)\wedge\xi=0$, here $\gamma\in\mathcal{R}^+_0$ and $f\in\mathcal{R}$.
In particular, we denote by $f^\xi$ the morphism represented as $(f,\xi)$ for $f\in\mathcal{R}$ and $\xi\in\mathcal{R}^+_0\!/r$ with $\mathrm{im}(f)\wedge\xi=0$.
Hence every morphism in $\square(\mathcal{R})$ is uniquely represented as the composition $f^\xi\gamma^\dagger$.
Moreover, by the unique factorization of morphisms in $\mathcal{R}$, we obtain the following result:

\begin{lemma}\label{lem:cube-fact}
Let $\mathcal{R}$ be a cubicalizable category.
Then every morphism $r'\to r$ of $\square(\mathcal{R})$ is uniquely represented as the composition
\[
\delta^\xi\sigma\gamma^\dagger
\]
such that $\gamma:s'\to r'\in\mathcal{R}^+_0$, $\sigma:s'\to s\in\mathcal{R}^-$, $\delta:s\to r\in\mathcal{R}^+_0$ and $\xi\in\mathcal{R}^+_0\!/r$ with $\delta\wedge\xi=0$.
\end{lemma}

\begin{lemma}\label{lem:cube-r0}
Let $\mathcal{R}$ be a cubicalizable category.
Then the following hold:
\begin{enumerate}[label={\rm(\arabic*)}]
  \item\label{sublem:cube-r0calc} If $\delta:s\to r\in\mathcal{R}^+_0$ and $\beta:t\to s\in\mathcal{R}^+_0$, then we have
\[
\delta^\xi\beta^\zeta=(\delta\beta)^{\xi\vee\delta_\ast\zeta}
\]
for each $\xi\in\mathcal{R}^+_0\!/r$ and $\zeta\in\mathcal{R}^+_0\!/s$ with $\xi\wedge\delta=0$ and $\zeta\wedge\beta=0$.
  \item\label{sublem:cube-r0span} If $f:r'\to r\in\mathcal{R}$ is a morphism, $\xi\in\mathcal{R}^+_0\!/r$ with $\xi\wedge\image(f)=0$ and $\gamma:t\to r\in\mathcal{R}^+_0$, then we have
\[
\gamma^\dagger f^\xi
= (\gamma^\dagger f)^{\gamma^\ast\xi}\,.
\]
  \item\label{sublem:cube-r0mono} Each $\delta:s\to r\in\mathcal{R}^+_0$ and $\xi\in\mathcal{R}^+_0\!/r$ with $\delta\wedge\xi=0$, we have $\delta^\dagger\delta^\xi=\mathrm{id}$.
Consequently, $\delta^\xi$ is a split monomorphism.
  \item\label{sublem:cube-initial} $\square(\mathcal{R})$ has a terminal object, namely $0$ the initial object of $\mathcal{R}$.
\end{enumerate}
\end{lemma}
\begin{proof}
Since $\delta^\ast\delta_\ast=\mathrm{id}$, by lemma \ref{lem:lambdacalc}, we have $\Lambda^\complement(\delta_\ast(\zeta))\delta=\delta\Lambda^\complement(\zeta)$.
Hence we obtain
\[
\delta^\xi\beta^\zeta
= (\xi\vee\delta_\ast\zeta,\Lambda^\complement(\delta_\ast\zeta)\delta\beta)
= (\xi\vee\delta_\ast\zeta,\delta\Lambda^\complement(\zeta)\beta)
= (\xi\vee\delta_\ast\zeta,\delta\beta)\,,
\]
and \ref{sublem:cube-r0calc} follows.

Similarly, we have
\[
\gamma^\dagger f^\xi
= (\gamma^\ast\xi, \Lambda^\complement(\gamma^\ast\xi)\gamma^\dagger f)\,,
\]
where by \ref{sublem:daglambda} of lemma \ref{lem:lambdacalc},
\[
\Lambda^\complement(\gamma^\ast\xi)\gamma^\dagger f
= \Lambda(\gamma^\ast(\neg\xi))\gamma^\dagger f
= \gamma^\dagger\Lambda(\neg\xi) f
= \gamma^\dagger f\,.
\]
Hence we obtain
\[
\gamma^\dagger f^\xi = (\gamma^\dagger f)^{\gamma^\ast\xi}\,,
\]
which implies \ref{sublem:cube-r0span}.

On the other hand, if $\delta\wedge\xi=0$, we have $\delta^\ast\xi=0$.
Hence we obtain
\[
\delta^\dagger\delta^\xi
= (\delta^\ast\xi,\delta^\dagger\delta)
= (0,\mathrm{id})\,
\]
which implies \ref{sublem:cube-r0mono}

Finally, for each object $r\in\square(\mathcal{R})$, the unique morphism $r\to 0\in\square(\mathcal{R})$ is $(0,0_r^\dagger)$, where $0_r:0\to r$ is the unique morphism in $\mathcal{R}$.
Hence we obtain \ref{sublem:cube-initial}.
\end{proof}

We denote by $\square(\mathcal{R})^+$ the set of morphisms in $\square(\mathcal{R})$ consisting of morphisms of the form $\delta^\xi$ with $\delta:s\to r\in\mathcal{R}^+_0$ and $\xi\in\mathcal{R}^+_0\!/r$ such that $\delta\wedge\xi=0$.
Then \ref{sublem:cube-r0calc} of lemma \ref{lem:cube-r0} implies that $\square(\mathcal{R})^+$ is in fact a wide subcategory of $\square(\mathcal{R})$.

Note also that there is an embedding $\mathcal{V}(\mathcal{R})\to\square(\mathcal{R})$ defined as $f\gamma^\dagger\mapsto f^0\gamma^\dagger$.
This preserves the composition and the identities since $0$ is the unit of the join operation.

\begin{exam}
Recall that the category $\widetilde\Delta^+$ of finite ordinals and order-preserving injections is a thin-powered category with thin-powered structure $\widetilde\Delta^+$ itself.
Its cubicalization $\square:=\square(\widetilde\Delta^+)$ is called the cubical site (see \cite{BHS}).
For a finite ordinal $\underline{n}$, the poset $\widetilde\Delta^+\!/\underline{n}$ is naturally identified with the power set $[1]^n$ of $\underline{n}=\{0,\dots,n-1\}$.
\end{exam}

\begin{exam}
Isaacson introduces the category $\square_\Sigma$ in \cite{Isa}, which he call the extended cubical site.
We can construct it in the above way:
Let $\Sigma=\{\Sigma_n\}_n$ denotes the group operad of symmetric groups.
Then we can naturally see it a crossed $\widetilde\Delta$-group, so we obtain a thin-powered category $\widetilde\Delta\Sigma$.
Of course, $\widetilde\Delta\Sigma$ is cubicalizable.
It is easily verifies that we have $\square_\Sigma=\square(\widetilde\Delta\Sigma)$.
\end{exam}


The reason of the terminology ``cube'' is as follows.
If $\mathcal{R}$ is a cubicalizable category, then each poset $\mathcal{R}^+_0\!/r$ is a finite Boolean lattice.
Recall that we set $\deg r=|\mathcal{R}^+_0\!/r|$.
Put $d(r):=\log_2(\deg r)$, then $\mathcal{R}^+_0\!/r$ is identified with the Power set of the set $\{0,\dots,d(r)-1\}$.
Hence let $N:\mathbf{Poset}\to\mathbf{SSet}$ be the nerve functor, and we have an isomorphism $N(\mathcal{R}^+_0\!/r)\simeq I^{d(r)}$ for $I=\Delta[1]$.
In other words, we have a functor $\mathcal{R}\ni r\mapsto I^{d(r)}\in\mathbf{SSet}$, and proposition \ref{prop:span-thinpow} implies it extends to $\mathcal{V}(\mathcal{R})\to\mathbf{SSet}$.
We, however, do not call $\mathcal{R}$ or $\mathcal{V}(\mathcal{R})$ cubical, since they do not have enogh face maps, while $\square(\mathcal{R})$ does:

\begin{prop}\label{prop:cube-repn}
Let $\mathcal{R}$ be a cubicalizable category.
Then the poset representation $\mathcal{R}^+_0\!/(\blankdot):\mathcal{R}\to\mathbf{Poset}$ extends to the functor $\square(\mathcal{R})\to\mathbf{Poset}$ by setting as follows:
For $\gamma:s\to r'\in\mathcal{R}^+_0$, $f:s\to r\in\mathcal{R}$ and $\xi\in\mathcal{R}^+_0\!/r$ with $\mathrm{im}(f)\wedge\xi=0$, we set
\[
(f^\xi\gamma^\dagger)_\ast(\eta):=f_\ast\gamma^\ast(\eta)\vee\xi
\]
for each element $\eta\in\mathcal{R}^+_0\!/r'$.
\end{prop}
\begin{proof}
We have to verify that the extension is compatible with the compositions.
Notice that for $\delta:s'\to r\in\mathcal{R}^+_0$, we have
\[
\delta^\dagger f^\xi
= (\delta^\ast\xi,\Lambda^\complement(\delta^\ast\xi)\delta^\dagger f)\,,
\]
and for $g:r\to r''$ and $\zeta\in\mathcal{R}^+_0\!/r''$ with $\image(g)\wedge\zeta=0$, we have
\[
g^\zeta f^\xi
= (\zeta\vee g_\ast\xi, \Lambda^\complement(g_\ast\xi)gf)\,.
\]
These imply that for each $\eta\in\mathcal{R}^+_0\!/r'$, we have
\[
\begin{split}
(\delta^\dagger f^\xi)_\ast(\eta)
&= \Lambda^\complement(\delta^\ast\xi)_\ast\delta^\ast f_\ast(\eta)\vee \delta^\ast\xi \\
&= ((\neg\delta^\ast\xi)\wedge \delta^\ast f_\ast(\eta))\vee \delta^\ast\xi \\
&= \delta^\ast f_\ast(\eta) \vee \delta^\ast\xi \\
&= \delta^\ast(f_\ast(\eta)\vee\xi) \\
&= (\delta^\dagger)_\ast (f^\xi)_\ast(\eta)\,,
\end{split}
\]
and
\[
\begin{split}
(g^\zeta f^\xi)_\ast(\eta)
&= \Lambda^\complement(g_\ast\xi)_\ast g_\ast f_\ast(\eta)\vee(\zeta\vee g_\ast\xi) \\
&= ((\neg g_\ast\xi)\wedge g_\ast f_\ast(\eta))\vee(\zeta\vee g_\ast\xi) \\
&= g_\ast f_\ast(\eta) \vee (\zeta\vee g_\ast\xi) \\
&= g_\ast( f_\ast(\eta)\vee\xi)\vee\zeta \\
&= (g^\zeta)_\ast(f^\xi)_\ast(\eta)\,.
\end{split}
\]
It follows that the extension is compatible with the compositions.
\end{proof}

\subsection{Simplicial realizations of cubicalizations}
We introduce the notion of simplicial realization.
We denote by $\square[r]$ the presheaf on $\square(\mathcal{R})$ represented by $r\in\mathcal{R}$.
By proposition \ref{prop:cube-repn} and the Yoneda extension, we obtain the following result:

\begin{lemma}
Let $\mathcal{R}$ be a cubicalizable category.
Then there is a functor $|\blankdot|:\square(\mathcal{R})^\wedge\to\mathbf{SSet}$ such that it preserves all (small) colimits, and for each $r\in\mathcal{R}$,
\[
|\square[r]|\simeq I^{d(r)}
\]
where $I=\Delta[1]$ and $d(r)=\log_2(\deg r)$.
We call it the simplicial realization functor.
\end{lemma}

Next, we define the following special $\square(\mathcal{R})$-set $\partial\square[r]$ as follows:
For $\delta\in\mathcal{R}^+_0\!/r$, we write
\[
\operatorname{codeg}\delta:=|(\mathcal{R}^+_0\!/r)_{\le\neg\delta}|\,.
\]
Notice that $\operatorname{codeg}\delta$ is always a power of $2$.
If $\operatorname{codeg}\delta'=4$ for $\delta'\in\mathcal{R}^+_0\!/r$, there are precisely two elements $\delta_1,\delta_2\in\mathcal{R}^+_0\!/r$ such that $\delta'\le\delta_1,\delta_2$ and $\operatorname{codeg}\delta_1=\operatorname{codeg}\delta_2=2$.
Say $\delta':s'\to r$ and $\delta_i:s_i\to r$, and $\delta'=\delta_i\varepsilon_i$ for $\varepsilon_i:s'\to s_i$.
Then for each $\xi\in\mathcal{R}^+_0\!/r$ with $\xi\le\neg\delta'$, we have two maps
\[
\square[s']
\xrightarrow{\varepsilon^{\delta_i^\ast\xi}}
\square[s_i]\times\{\xi\wedge\neg\delta_i\}
\hookrightarrow
\square[s_i]\times(\mathcal{R}^+_0\!/r)_{\le\neg\delta_i}
\]
for $i=1,2$.
These maps induce paralleled arrows in $\square(\mathcal{R})^\wedge$:
\[
\raisedunderop\coprod{\substack{(\delta':s'\to r)\in\mathcal{R}^+_0\!/r\\\operatorname{codeg}\delta'=4}} \square[s']\times (\mathcal{R}^+_0\!/r)_{\le\neg\delta'}
\rightrightarrows
\raisedunderop\coprod{\substack{(\delta:s\to r)\in\mathcal{R}^+_0\!/r\\\operatorname{codeg}\delta=2}} \square[s]\times {\mathcal{R}^+_0\!/r}_{\le\neg\delta}
\]
We define $\partial\square[r]$ to be the coequalizer of the diagram above.

For $\delta:s\to r\in\mathcal{R}^+_0$ and $\xi\in\mathcal{R}^+_0\!/r$ with $\delta\wedge\xi=0$, we have a morphism $\delta^\xi:s\to r$.
Hence we obtain a morphism below:
\[
\square[s]\times(\mathcal{R}^+_0\!/r)_{\le\neg\delta}\to\square[r]
\]
This clearly induces the morphism $\partial\square[r]\to\square[r]$.
This is a monomorphism.
To see this, we need some lemmas:

\begin{lemma}\label{lem:cube-face-ord}
Let $\mathcal{R}$ be a cubicalizable category.
Suppose we have morphisms $\delta^\xi:s\to r$ and $\gamma^\zeta:t\to r$ in $\square(\mathcal{R})$ with $\delta,\gamma\in\mathcal{R}^+_0$.
Then $\gamma^\zeta$ factors through $\delta^\xi$ if and only if we have $\gamma\le\delta$ and $\xi=\zeta\wedge\neg\delta$.
More precisely, in this case, we have $\gamma^\zeta=\delta^\xi(\delta^\ast\gamma)^{\delta^\ast\zeta}$.
\end{lemma}
\begin{proof}
Suppose $\gamma^\zeta=\delta^\xi f$ for some $f:t\to s\in\square(\mathcal{R})$.
By the uniquness of the factorization in \ref{lem:cube-fact}, $f=\varepsilon^\omega$ for some $\omega\in\mathcal{R}^+_0$.
Then by \ref{sublem:cube-r0calc} of lemma \ref{lem:cube-r0}, we obtain
\[
\gamma^\zeta=\delta^\xi\varepsilon^\omega=(\delta\varepsilon)^{\xi\vee\delta_\ast\omega}\,,
\]
hence $\gamma=\delta\varepsilon$ and $\zeta=\xi\vee\delta_\ast\omega$.
These imply that $\gamma\le\delta$ and $\zeta\wedge\neg\delta=\xi$ respectively.

Conversely, suppose $\gamma\le\delta$ and $\zeta\wedge\neg\delta=\xi$.
Then we have
\[
\delta^\xi(\delta^\ast\gamma)^{\delta^\ast\zeta}
= (\delta_\ast\delta^\ast\gamma)^{\xi\vee\delta_\ast\delta^\ast\zeta}
= (\gamma\wedge\delta)^{\xi\vee(\zeta\wedge\delta)}
= \gamma^{(\zeta\wedge\neg\delta)\vee(\zeta\wedge\delta)}
= \gamma^\zeta\,.
\]
\end{proof}

\begin{lemma}\label{lem:cube-pb}
Let $\mathcal{R}$ be a cubicalizable category, ans suppose $\delta_i^{\xi_i}:s_i\to r$ is a morphism in $\square(\mathcal{R})$ with $\delta_i\in\mathcal{R}^+_0$ for $i=1,2$.
Then the following hold:
\begin{enumerate}[label={\rm(\arabic*)}]
  \item\label{sublem:cube-disj} If $\varepsilon_1\wedge\neg\delta_2\neq\varepsilon_2\wedge\neg\delta_1$, the following square is a pullback in $\square(\mathcal{R})^\wedge$:
\begin{equation}
\label{eq:face-disj-sq}
\xymatrix{
  \varnothing \ar[r] \ar[d] \ar@{}[dr]|(.4){\pbcorner} & \square[s_1] \ar[d]^{\delta_1^{\xi_1}} \\
  \square[s_2] \ar[r]^{\delta_2^{\xi_2}} & \square[r] }
\end{equation}
  \item\label{sublem:cube-meet} If $\varepsilon_1\wedge\neg\delta_2=\varepsilon_2\wedge\neg\delta_1$, the following square is a pullback in $\square(\mathcal{R})$:
\begin{equation}
\label{eq:face-meet-sq}
\renewcommand{\labelstyle}{\textstyle}
\xymatrix@+1pc{
  s_0 \ar[r]^{(\delta_1^\ast\delta_2)^{\delta_1^\ast\varepsilon_2}} \ar[d]_{(\delta_2^\ast\delta_1)^{\delta_2^\ast\varepsilon_1}} \ar@{}[dr]|(.4){\pbcorner} & s_1 \ar[d]^{\delta_1^{\xi_1}} \\
  s_1 \ar[r]^{\delta_2^{\xi_2}} & r }
\end{equation}
Moreover, the composition equals to $(\delta_1\wedge\delta_2)^{\varepsilon_1\vee\varepsilon_2}$.
\end{enumerate}
\end{lemma}
\begin{proof}
To see the square \eqref{eq:face-disj-sq} is a pullback, it suffices to show that there is no commutative square in $\square(\mathcal{R})$ of the form:
\[
\xymatrix{
  t \ar[r]^{f_1} \ar[d]_{f_2} & s_1 \ar[d]^{\delta_1^{\xi_1}} \\
  s_2 \ar[r]^{\delta_2^{\xi_2}} & r }
\]
We show its contraposition, so suppose we have such a square.
By the unique factorization in lemma \ref{lem:cube-fact} and \ref{sublem:cube-r0calc} of lemma \ref{lem:cube-r0}, we may assume $f_i=\gamma_i^{\zeta_i}$ with $\gamma_i\in\mathcal{R}^+_0$ for $i=1,2$.
The commutativity implies that we have
\[
(\delta_1\gamma_1)^{\xi_1\vee\delta_1\zeta_1} = (\delta_2\gamma_2)^{\xi_2\vee\delta_2\zeta_2}\,.
\]
Then we obtain
\[
\xi_1\wedge\neg\delta_2
= (\neg\delta_1\wedge\neg\delta_2)\wedge(\xi_1\vee\delta_1\zeta_1)
= (\neg\delta_1\wedge\neg\delta_2)\wedge(\xi_2\vee\delta_2\zeta_2)
= \xi_2\wedge\neg\delta_1
\]
as required.

Conversely, suppose we have $\varepsilon_1\wedge\neg\delta_2 = \varepsilon_2\wedge\neg\delta_1$.
Notice that the equation implies $\xi_1\vee\delta_1=\xi_2\vee\delta_2$, and hence
\[
\xi_1 \le \xi_1\vee\delta_1 = \xi_2\vee\delta_2 \ge \xi_2\,,
\]
which implies $\xi_1\vee\xi_2\le\xi_1\vee\delta_1$.
Then the square is in fact commutative; indeed we have
\[
\delta_1^{\xi_1}(\delta_1^\ast\delta_2)^{\delta_1^\ast\xi_2}
= (\delta_1\wedge\delta_2)^{\xi_1\vee(\xi_2\wedge\delta_1)}
= (\delta_1\wedge\delta_2)^{(\xi_1\vee\xi_2)\wedge(\xi_1\vee\delta_1)}
= (\delta_1\wedge\delta_2)^{\xi_1\vee\xi_2}
\]
and the other is similar.
Now let
\[
\xymatrix{
  t \ar[r]^{f_1} \ar[d]_{f_2} & s_1 \ar[d]^{\delta_1^{\xi_1}} \\
  s_2 \ar[r]^{\delta_2^{\xi_2}} & r }
\]
be a commutative square in $\square(\mathcal{R})$.
By the same reason above, we may assume $f_i=\gamma_i^{\zeta_i}$ with $\gamma_i\in\mathcal{R}^+_0$ for $i=1,2$.
Then we have
\[
(\delta_1\gamma_1)^{\xi_1\vee\delta_1\zeta_1} = (\delta_2\gamma_2)^{\xi_2\vee\delta_2\zeta_2}\,.
\]
In particular, $\delta_1\gamma_1=\delta_2\gamma_2$ and $\xi_1\vee\delta_1\zeta_1=\xi_2\vee\delta_2\zeta_2$.
It follows that $\gamma_1\le\delta_1^\ast\delta_2$ and $\delta_1^\ast\xi_2=\zeta\wedge\neg\delta_1^\ast\delta_2$ since we have
\[
\xi_2
= (\xi_1\vee\delta_1\zeta_1)\wedge\neg\delta_2
= (\xi_1\wedge\neg\delta_2)\vee(\delta_1\zeta_1\wedge\neg\delta_2)
= (\xi_2\wedge\neg\delta_1)\vee(\delta_1\zeta_1\wedge\neg\delta_2)\,.
\]
Hence by lemma \ref{lem:cube-face-ord}, $\gamma_1^{\zeta_1}$ factors through $(\delta_1^\ast\delta_2)^{\delta_1^\ast\xi_2}$, say $\gamma_1^{\zeta_1}=(\delta_1^\ast\delta_2)^{\delta_1^\ast\xi_2}f'$.
We have
\[
\delta_2^{\xi_2}\gamma_2^{\zeta_2}
= \delta_1^{\xi_1}\gamma_1^{\zeta_1}
= \delta_1^{\xi_1}(\delta_1^\ast\delta_2)^{\delta_1^\ast\xi_2}f'
= \delta_2^{\xi_2}(\delta_2^\ast\delta_1)^{\delta_2^\ast\xi_1}f'\,,
\]
and since $\delta_2^{\xi_2}$ is a monomorphism, we obtain $\gamma_2^{\zeta_2}=(\delta_2^\ast\delta_1)^{\delta_2^\ast\xi_1}f'$.
Since $(\delta_1^\ast\delta_2)^(\delta_1^\ast\xi_2)$ is a monomorphism, such $f'$ is unique.
Therefore, the square \eqref{eq:face-meet-sq} is a pullback in $\square(\mathcal{R})$.
\end{proof}

Now we are ready to prove the required result:

\begin{prop}
Let $\mathcal{R}$ be a cubicalizable category.
Then for each $r\in\mathcal{R}$, the morphism $\partial\square[r]\to\square[r]$ is a monomorphism.
\end{prop}
\begin{proof}
It suffices to show that it is a pointwise monomorphism; i.e. the map $\partial\square[r](t)\to\square[r](t)$ is injective for each $t\in\mathcal{R}$.
Notice that the map is induced by the map
\[
\square(\mathcal{R})(t,s)\times(\mathcal{R}^+_0)_{\le\neg\delta}
\ni (f,\xi)\mapsto \delta^\xi f\in\square(\mathcal{R})(t,r)
\]
for each $\delta\in\mathcal{R}^+_0\!/r$ with $\operatorname{codeg}\delta=2$.
Choose $(\delta_i:s_i\to r)\in\mathcal{R}^+_0$ with $\operatorname{codeg}\delta_i=2$, $\xi_i\in\mathcal{R}^+_0\!/r$ with $\delta_i\wedge\xi_i=0$ and $f_i:t\to s_i\in\square(\mathcal{R})$ for $i=1,2$, and suppose we have $\delta_1^{\xi_1}f_1=\delta_2^{\xi_2}f_2$.
We have to show that $f_1=f_2$ or they are images of a common element of $\square(\mathcal{R})(t,s')\times(\mathcal{R}^+_0\!/r)_{\le\neg\delta'}$ for $\delta'\in\mathcal{R}^+_0\!/r$ with $\operatorname{codeg}\delta'=4$.

If $(\delta_1,\xi_1)=(\delta_2,\xi_2)$, by \ref{sublem:cube-r0mono} of lemma \ref{lem:cube-r0}, we obtain $f_1=f_2$.
So we assume $(\delta_1,\xi_1)\neq(\delta_2,\xi_2)$.
In particular, if $\delta_1=\delta_2$ and $\xi_1\neq\xi_2$, we have
\[
\xi_2\wedge\neg\delta_1 = \xi_2
\neq \xi_1 = \xi_1\wedge\neg\delta_2
\]
and by \ref{sublem:cube-disj} of lemma \ref{lem:cube-pb}, the assumption cannot hold.
Hence we may assume $\delta_1\neq\delta_2$, which imply that
\[
\xi_1\wedge\neg\delta_2 = 0 = \xi_2\wedge\neg\delta_1\,.
\]
By lemma \ref{lem:cube-pb}, if $(\delta_1\wedge\delta_2)^{\xi_1\vee\xi_2}:s_0\to r$, there is a morphism $f':t\to s_0$ such that $f_1 = (\delta_1^\ast\delta_2)^{\delta_1^\ast\xi_2}f'$ and $f_2=(\delta_2^\ast\delta_1)^{\delta_2^\ast\xi_1}f'$.
It follows that $(f_1,\xi_1)$ and $(f_2,\xi_2)$ are the images of a common element $(f',\xi_1\vee\xi_2)\in\square(\mathcal{R})(t,s_0)\times(\mathcal{R}^+_0\!/r)_{\le\neg(\delta_1\wedge\delta_2)}$ by $\delta_1^\ast\delta_2^{\delta_1^\ast\xi_2}$ and $\delta_2^\ast\delta_1^{\delta_2^\ast\xi_1}$ respectively.
Since $\delta_1\neq\delta_2$ and $\operatorname{codeg}\delta_1=\operatorname{codeg}\delta_2=2$, we have $\operatorname{codeg}(\delta_1\wedge\delta_2)=4$, so that the result follows.
\end{proof}

\begin{corol}
Let $\mathcal{R}$ be a cubicalizable category, and let $r\in\mathcal{R}$.
Then $\partial\square[r]$ is a subpresheaf of $\square[r]$ consisting of cells $f:s\to r$ which factors as $f=\delta^\xi\sigma\gamma^\dagger$ with $\delta\in\mathcal{R}^+_0$ non-trivial.
\end{corol}

Finally, we prove the following proposition:

\begin{prop}\label{prop:cube-bound-mono}
If $\mathcal{R}$ is a cubicalizable category and $r\in\mathcal{R}$, then the simplicial map
\[
|\partial\square[r]|\to|\square[r]|
\]
is a monomorphism.
\end{prop}
\begin{proof}
Since $|\blankdot|:\square(\mathcal{R})^\wedge\to\mathbf{SSet}$ preserves colimits, we have
\begin{equation}
\label{eq:real-bound}
|\partial\square[r]|
\simeq \operatorname{coeq}\left(\coprod |\square[s']|\times(\mathcal{R}^+_0\!/r)_{\le\neg\delta'}\rightrightarrows \coprod|\square[s]|\times(\mathcal{R}^+_0\!/r)_{\le\neg\delta}\right)
\end{equation}
where coproducts run over certain sets as in the definition.
Let $n=\log_2(\deg r)$ and for each pair $1\le i<j\le n$, define simplicial maps
\[
\bar{d}_i,\bar{d}'_j:\Delta[1]^{n-2}\times\{0,1\}^2\to\Delta[1]^{n-1}\times\{0,1\}
\]
by the following:
\[
\begin{split}
\bar{d}_i(x_1,\dots,x_{n-2},\varepsilon_1,\varepsilon_2)
&= (x_1,\dots,x_{i-1},\varepsilon_1,x_i,\dots,x_{n-2},\varepsilon_2) \\
\bar{d}'_j(x_1,\dots,x_{n-2},\varepsilon_1,\varepsilon_2)
&= (x_1,\dots,x_{i-1},\varepsilon_2,x_i,\dots,x_{n-2},\varepsilon_1)
\end{split}
\]
These maps define the simplicial map
\[
\coprod_{1\le i<j\le n} \Delta[1]^{n-2}\times\{0,1\}^2
\rightrightarrows
\coprod_{1\le i\le n} \Delta[1]^{n-1}\times\{0,1\}\,.
\]
These arrows are isomorphic to those in the coequalizer \eqref{eq:real-bound}.
Hence we obtain an isomorphism
\[
|\partial\square[r]|
\simeq \bigcup_{i=1}^n \Delta[1]^{i-1}\times\partial\Delta[1]\times\Delta[1]^{n-i}
\]
and $|\partial\square[r]|\to |\square[r]|\simeq\Delta[1]^n$ is nothing but the inclusion.
\end{proof}

\subsection{The Eilenberg-Zilber condition}
In this section, we consider a very nice situations:

\begin{defin}
A cubicalizable category $\mathcal{R}$ is said to be EZ cubicalizable if it is confluent and every $\mathcal{R}^+_0$-surjection is a split epimorphism.
\end{defin}

It is clear that every EZ cubicalizable category is itself an EZ category (see example \ref{exam:EZ-cat} in section \ref{sec:conflu} for the definition).

Suppose $\mathcal{R}$ is an EZ cubicalizable category and $\sigma:r'\to r$ is a degree-preserving $\mathcal{R}^+_0$-surjection.
Since $\mathcal{R}$ is EZ cubicalizable, $\sigma$ has a section $f:r\to r'$.
Then $f$ also preserves the degree, so $\image(f)$ must be the identity.
This implies that $f$ is an $\mathcal{R}^+_0$-surjection, which is again a split epimorphism.
Hence $f$ is an isomorphism, and $\sigma$ is its inverse.
These argument shows that the poset representation $\mathcal{R}^+_0\!/(\blankdot):\mathcal{R}\to\mathbf{Poset}$ is conservative.

Moreover, as expected from the terminology, we have the following:

\begin{prop}\label{prop:cubeEZ}
If $\mathcal{R}$ is an EZ cubicalizable category, then the cubicalization $\square(\mathcal{R})$ is an EZ category.
\end{prop}

An essential part of the proof of proposition \ref{prop:cubeEZ} is to verify the condition \ref{cond:EZconf}.
Since $\mathcal{R}$ is confluent, every pair of morphisms in $\mathcal{R}^-$ with the common codomain has an absolute pushout.
Every pair of morphisms in $(\mathcal{R}^+_0)^\dagger$ with the common codomain also has an absolute pushout diagram by proposition \ref{prop:c0dagpo}.
Moreover, we have the following result:

\begin{lemma}\label{lem:r-rdag-po}
Let $\mathcal{R}$ is an EZ cubicalizable category.
Suppose we have a diagram $r_1\xleftarrow{\sigma}r_0\xrightarrow{\gamma^\dagger}r_2$ with $\sigma\in\mathcal{R}^-$ and $\gamma\in\mathcal{R}^+_0$.
Say $\gamma_0\in\mathcal{R}^+_0\!/r_0$ is the maximum saturated element with respect to $\sigma$ such that $\gamma_0\le\gamma$.
Then the following diagram is an absolute pushout square in $\square(\mathcal{R})$:
\begin{equation}
\label{eq:R-Rdag-po}
\xymatrix@+1pc{
  r_0 \ar[d]_{\sigma} \ar[r]^{\gamma^\dagger} \ar@{}[dr]|(.6){\pocorner} & r_2 \ar[d]^{\coim(\sigma\gamma_0)(\gamma^\ast\gamma_0)^\dagger} \\
  r_1 \ar[r]^{\image(\sigma\gamma_0)^\dagger} & r }
\end{equation}
\end{lemma}
\begin{proof}
First note that since $\gamma(\gamma^\ast\gamma_0)=\gamma\wedge\gamma_0=\gamma_0$ and the diagram
\[
\xymatrix@R+.3pc@C+1pc{
  r'_2 \ar[r]^{\gamma_0} \ar[d]_{\coim(\sigma\gamma_0)} \ar@{}[dr]|(.4){\pbcorner} & r_0 \ar[d]^{\sigma} \\
  r \ar[r]^{\image(\sigma\gamma_0)} & r_1 }
\]
is a pullback, the square \eqref{eq:R-Rdag-po} is in fact commutative.

Now suppose we have a square
\[
\xymatrix{
  r_0 \ar[d]_{\sigma} \ar[r]^{\gamma^\dagger} & r_2 \ar[d]^{x_2} \\
  r_1 \ar[r]^{x_1} \ar[r] & X }
\]
in $\square(\mathcal{R})^\wedge$.
By \ref{sublem:cube-r0mono} of lemma \ref{lem:cube-r0}, we have
\[
x_1\sigma\gamma^{\neg\gamma} = x_2\gamma^\dagger\gamma^{\neg\gamma} = x_2\,.
\]
Here
\[
\sigma\gamma^{\neg\gamma}
= (\sigma_\ast(\neg\gamma), \Lambda^\complement(\sigma_\ast(\neg\gamma))\sigma\gamma)
= (\sigma_\ast(\neg\gamma), \sigma\Lambda^\complement(\sigma^\ast\sigma_\ast(\neg\gamma))\gamma)\,.
\]
By lemma \ref{lem:maxsatur}, we have $\gamma_0=\neg\sigma^\ast\sigma_\ast(\neg\gamma)$, so that we obtain
\[
\begin{split}
\sigma\gamma^{\neg\gamma}
&= (\sigma_\ast(\neg\gamma), \sigma\Lambda(\gamma_0)\gamma) \\
&= (\sigma_\ast(\neg\gamma), \sigma\gamma_0(\gamma^\ast\gamma_0)^\dagger) \\
&= \image(\sigma\gamma_0)^{\sigma_\ast(\neg\gamma)}\coim(\sigma\gamma_0)(\gamma^\ast\gamma_0)^\dagger\,.
\end{split}
\]
Put $x_0:=x_1\image(\sigma\gamma_0)^{\sigma_\ast(\neg\gamma)}$, then we have $x_2 = x_0\coim(\sigma\gamma_0)(\gamma^\ast\gamma_0)^\dagger$.
On the other hand, we have
\[
x_0\image(\sigma\gamma_0)^\dagger\sigma
= x_0\coim(\sigma\gamma_0)(\gamma^\ast\gamma_0)^\dagger\gamma^\dagger
= x_2\gamma^\dagger
= x_1\sigma\,.
\]
Since $\sigma$ is a split epimorphism, we obtain $x_1=x_0\image(\sigma\gamma_0)^\dagger$.
Such $x_0$ is unique since $\image(\sigma\gamma_0)^\dagger$ is a split epimorphism.
Therefore, the square \eqref{eq:R-Rdag-po} is an absolute pushout.
\end{proof}

\begin{proof}[Proof of proposition \ref{prop:cubeEZ}.]
Every morphism of $\square(\mathcal{R})$ is represented as $\delta^\xi\sigma\gamma^\dagger$.
Then $\gamma^\dagger$ is a split epimorphism, and $\sigma$ is also a split epimorphism since $\mathcal{R}$ is EZ cubicalizable.
On the other hand, by lemma \ref{lem:cube-fact} $\delta^\xi$ is a monomorphism.
It follows that every morphism factors as a split epimorphism followed by a monomorphism.
Moreover, every monomorphism factors as a morphism of the form $\delta^\xi$ followed by an isomorphism, hence it raises the degree whenever it is not an isomorphism.
Finally the condition \ref{cond:EZconf} follows from lemma \ref{lem:r-rdag-po}.
\end{proof}

\begin{corol}\label{cor:real-mono}
If $\mathcal{R}$ is an EZ cubicalizable category, then the simplicial realization functor
\[
|\blankdot|:\square(\mathcal{R})^\wedge\to\mathbf{SSet}
\]
preserves monomorphisms.
\end{corol}
\begin{proof}
Notice that in the case that $\mathcal{R}$ is EZ cubicalizable, the notion $\partial\square[r]$ agrees with that defined in section \ref{sec:conflu}.
Then since by proposition \ref{prop:cubeEZ}, $\square(\mathcal{R})$ is an EZ category, corollary \ref{cor:ez-cellular} implies that every monomorphism in $\square(\mathcal{R})^\wedge$ is a transfinite composition of pushouts of morphisms in the set
\[
S_0=\{H\backslash\partial\square[r]\hookrightarrow H\backslash\square[r]\mid r\in\mathcal{R},H<\mathrm{Aut}_{\square(\mathcal{R})}(r)\}\,.
\]
Since the simplicial realization preserves colimits and the class of monomorphisms in $\mathbf{SSet}$ is closed under transfinite compositions and pushouts, it suffices to show that each morphism
\[
|H\backslash\partial\square[r]|\to|H\backslash\square[r]|
\]
is a monomorphism.
Now, since $|\blankdot|$ preserves colimits, the morphism is isomorphic to
\[
H\backslash|\partial\square[r]|\to H\backslash|\square[r]|\,,
\]
and by proposition \ref{prop:cube-bound-mono}, $|\partial\square[r]|$ is a simplicial subset of $|\square[r]|$.
Thus the morphism is a monomorphism.
\end{proof}

\section{Test category theory}
\label{sec:test-cat}
In the section, we prove that the cubicalizations are actually test categories.
The notion of test categories was introduced by Grothendieck in \cite{Gro}, where it was pointed out that a test category is a small category on which we can do homotopy theory.
In fact, Cisinski developed a model structure on the category of presheaves over a test category in \cite{Cis} based on the idea.
So we can regard cubicalizations as a model of homotopy theory.
Notice that, as a result of the integrated construction, most of the discussion in the section goes parallel to that in the simplest case $\square$, see \cite{Cis} and \cite{Jar}.

\subsection{Homotopical categories}
\label{sec:H-cats}
Before the discussion about test categories, it is convenient to introduce the notion of homotopical categories.
In \cite{DHKS}, it was pointed out that a homotopy theory given by a model structure on a category is actually determined by the class of weak equivalences, and the cofibrations and fibrations are only additional structures.
So we can isolate the model-categorical arguments involving only weak equivalences from the whole theory.

\begin{defin}
A homotopical category is a bicomplete category $\mathcal{C}$ together with a wide subcategory $\mathcal{W}$ of $\mathcal{C}$ which contains all isomorphisms and satisfies the two-out-of-six property; i.e. if we have a sequence
\[
X\xrightarrow{f} Y\xrightarrow{g} Z\xrightarrow{h} W
\]
of morphisms in $\mathcal{C}$, then $gf,hg\in\mathcal{W}$ imply $f,g,h,hgf\in\mathcal{W}$.
An element of $\mathcal{W}$ is called a weak equivalence.
\end{defin}

We often identify a wide subcategory with a class of morphisms which contains all the identities and is closed under compositions.
In the original article \cite{DHKS}, the assumption of bicompleteness does not appear.
However, in this paper, everything is discussed for bicomplete categories.
This is why we added the assumption to the definition of homotopical categories.

Note that if $\mathcal{C}$ is a category, then the class $\mathrm{Isom}(\mathcal{C})$ of isomorphisms satisfies the two-out-of-six property; indeed, $gf,hg\in\mathrm{Isom}(\mathcal{C})$ imply $g\in\mathrm{Isom}(\mathcal{C})$.
Moreover, suppose $F:\mathcal{C}\to\mathcal{D}$ is a functor and $\mathcal{W}_{\mathcal{D}}\subset\mathcal{D}$ is a class of morphisms which satisfies the two-out-of-six property.
Then the class
\[
F^{-1}\mathcal{W}_{\mathcal{D}}:=\{f:X\to Y\in\mathcal{C}\mid F(f)\in\mathcal{W}_{\mathcal{D}}\}
\]
also satisfies the two-out-of-six property.
In particular, if $\mathcal{C}$ is a model category and $\mathcal{W}$ is the class of weak equivalences, then $\mathcal{W}$ satisfies the two-out-of-six property since it is the inverse image of $\mathrm{isom}(\operatorname{Ho}\mathcal{C})$ by the localization $\mathcal{C}\to\mathrm{isom}(\operatorname{Ho})\mathcal{C}$.

Conversely, it is easily verified that if a wide subcategory $\mathcal{W}\subset\mathcal{C}$ satisfies the two-out-of-six, it also satisfies the two-out-of-three.

\begin{defin}
Let $\mathcal{C}$ be a homotopical catetory.
\begin{enumerate}[label={\rm(\roman*)}]
  \item A cylinder object on $X\in\mathcal{C}$ is a diagram
\[
X\xrightrightarrows[i_1]{i_0} \widehat X\xrightarrow{p} X
\]
such that $pi_0=pi_1=\mathrm{id}_X$ and $p:\widehat X\to X$ is a weak equivalence.
  \item Dually, a path object on $X\in\mathcal{C}$ is a diagram
\[
X\xrightarrow{j} \widetilde X\xrightrightarrows[q_1]{q_0} X
\]
such that $q_0j=q_1j=\mathrm{id}_X$ and $j:X\to\widetilde X$ is a weak equivalence.
  \item Two morphisms $f_0,f_1:X\to Y\in\mathcal{C}$ are said to be left homotopic if there are a cylinder object $\widehat X$ on $X$ and a morphism $h:\widehat X\to Y$ such that $f_0=h i_0$ and $f_1=h i_1$.
  \item Two morphisms $f_0,f_1:X\to Y\in\mathcal{C}$ are said to be right homotopic if there are a path object $\widetilde Y$ on $Y$ and a morphism $k:X\to\widetilde Y$ such that $f_0=p_0k$ and $f_1=p_1k$.
\end{enumerate}
\end{defin}

Notice that the homotopic relations are symmetric and reflexive.
In general, however, it is not transitive, so the relations do not become equivalence relations on morphism sets.

\begin{lemma}\label{lem:homotopic-weq}
Let $\mathcal{C}$ be a homotopical category, and suppose morphisms $f_0,f_1:X\to Y$ are left (resp. right) homotopic.
Then $f_0$ is a weak equivalence if and only if $f_1$ is so.
\end{lemma}
\begin{proof}
By duality, we only show the statement in the case that $f_0$ and $f_1$ are left homotopic.
So we have a cylinder object
\[
X\amalg X\xrightarrow{(i_0,i_1)}\widehat X\xrightarrow{p} X
\]
and a morphism $h:\widehat X\to Y$ such that $f_0=hi_0$ and $f_1=hi_1$.
Since $\widehat X$ is a cylinder object, the morphism $p$ is a weak equivalence.
Moreover, we have $pi_0=pi_1=\mathrm{id}_X$, which implies $i_0$ and $i_1$ are weak equivalences since the class $\mathcal{W}$ of weak equivalences satisfies the two-out-of-three property.
Thus $f_0$ is a weak equivalence if and only if $h$ is so, and the other is the same.
\end{proof}

\begin{defin}
Let $\mathcal{C}$ be a homotopical category.
Let $f:X\to Y$ and $g:Y\to X$ be morphisms in $\mathcal{C}$.
Then $g$ is called a homotopy inverse of $f$ if $gf$ is either left or right homotopic to the identity, and $fg$ is either left or right homotopic to the identity.
\end{defin}

It is clear that $g$ is a homotopy inverse of $f$ if and only if $f$ is a homotopy inverse of $g$.

\begin{corol}\label{cor:H-inverse}
Let $\mathcal{C}$ be a homotopical category.
If $f:X\to Y\in\mathcal{C}$ admits a homotopy inverse, then $f$ is a weak equivalence.
\end{corol}
\begin{proof}
Let $g:Y\to X$ be a homotopy inverse of $f$.
Then lemma \ref{lem:homotopic-weq} implies that $gf$ and $fg$ are weak equivalences.
Hence by the two-out-of-six property, $f$ is a weak equivalence.
\end{proof}

\subsection{Test categories}
The notion of test categories was introduced by Grothendieck in \cite{Gro}.
Further developments are established by Cisinski in \cite{Cis} and Jardine in \cite{Jar}.
Before giving the definition, we begin with an essential proposition:

\begin{prop}
Let $\mathcal{A}$ be a small category.
We define functors $\int^{\mathcal{A}}:\mathcal{A}^\wedge\to\mathbf{Cat}$ and $N_{\mathcal{A}}:\mathbf{Cat}\to\mathcal{A}^{\wedge}$ by
\[
\txtint^{\mathcal{A}}X = \mathcal{A}/X
\]
and
\[
N_{\mathcal{A}}C = \mathbf{Cat}(\mathcal{A}/(\blankdot),C)\,.
\]
Then $\txtint^{\mathcal{A}}\dashv N_{\mathcal{A}}$ is an adjoint pair.
\end{prop}
\begin{proof}
For $X\in\mathcal{A}^\wedge$, we defin a morphism $\eta:X\to N_{\mathcal{A}}\txtint^{\mathcal{A}} X$ by the composition
\[
X
\simeq \mathcal{A}^\wedge(\mathcal{A}[\blankdot],X)
\xrightarrow{\txtint^{\mathcal{A}}} \mathbf{Cat}(\mathcal{A}/(\blankdot),\mathcal{A}/X)
= N_{\mathcal{A}}\txtint^{\mathcal{A}}X\,,
\]
where the first isomorphism arises from the Yoneda lemma.
On the other hand, if we write $j:\mathcal{A}\ni a\mapsto \mathcal{A}/a\in\mathbf{Cat}$, then we have a natural isomorphism $\txtint^{\mathcal{A}}N_{\mathcal{A}}C\simeq\commacat{j}{C}$ for $C\in\mathbf{Cat}$.
Indeed, the objects of the category $\txtint^{\mathcal{A}}N_{\mathcal{A}}C$ are diagrams of the form
\[
\mathcal{A}[a]\to N_{\mathcal{A}}C
\]
in $\mathcal{A}^\wedge$, which one-to-one correspond to diagrams
\[
\mathcal{A}/a \to C
\]
in $\mathbf{Cat}$.
This correspondence naturally extends to morphisms.
We now define $\varepsilon:\txtint^{\mathcal{A}}N_{\mathcal{A}}C\to C$ by
\[
\txtint^{\mathcal{A}}N_{\mathcal{A}}C
\simeq \commacat{j}{C}
\xrightarrow{\mathrm{eval}_1} C\,,
\]
where $\mathrm{eval}_1:\commacat{j}{C}\to C$ is the evaluation at the terminal object, that is, the functor
\[
\mathrm{eval}_1:\commacat{j}{C}\ni\left(\mathcal{A}/a\xrightarrow{\gamma}C\right)\mapsto \gamma(1_a)\in C\,.
\]

We show $(\txtint^{\mathcal{A}},N_{\mathcal{A}},\eta,\varepsilon)$ is actually an adjunction.
To do so, we have to show the triangle identities.
For $X\in\mathcal{A}^\wedge$ and each $(\mathcal{A}[a]\xrightarrow{x} X)\in\txtint^{\mathcal{A}}X$, the image of $x$ by the composition
\begin{equation}\label{eq:trig-id-l}
\txtint^{\mathcal{A}} X
\xrightarrow{\txtint^{\mathcal{A}}\eta} \txtint^{\mathcal{A}}N_{\mathcal{A}}\txtint^{\mathcal{A}}X
\xrightarrow{\varepsilon} \txtint^{\mathcal{A}}X
\end{equation}
is computed as follows:
\[
\begin{split}
\varepsilon\circ\txtint^{\mathcal{A}}\eta(x)
&= \varepsilon\left(\mathcal{A}[a]\xrightarrow{x} X \xrightarrow{\eta} N_{\mathcal{A}}\txtint^{\mathcal{A}}X\right) \\
&= \varepsilon\left(\mathcal{A}[a]\xrightarrow{x} X \simeq\mathcal{A}^\wedge(\mathcal{A}[\blankdot],X)\xrightarrow{\txtint^{\mathcal{A}}} \mathbf{Cat}(\mathcal{A}/(\blankdot),\mathcal{A}/X)\right) \\
&= \mathrm{eval}_1\left(\mathcal{A}/a\xrightarrow{x_\ast}\mathcal{A}/X\right) \\
&= x\,.
\end{split}
\]
Thus the composition \eqref{eq:trig-id-l} is the identity.

We finally show the other identity.
Notice that for a small category $C$ and a cell $(\mathcal{A}/a\xrightarrow{\gamma}C)\in N_{\mathcal{A}}C$, we can define a functor $\varphi(\gamma):\mathcal{A}/a\to\commacat{j}{C}$ which sends each object $(a'\xrightarrow{f} a)\in\mathcal{A}/a$ to the composition
\[
j(a')=\mathcal{A}/a'\xrightarrow{f_\ast}\mathcal{A}/a\xrightarrow{\gamma} C\,.
\]
This assignment gives rise to a morphism $\varphi:N_{\mathcal{A}}C\to N_{\mathcal{A}}(\commacat{j}{C})$ in $\mathcal{A}^\wedge$.
Then we obtain the following commutative diagram:
\[
\xymatrix{
  N_{\mathcal{A}}C \ar@{=}[r]^-{\sim} \ar[d]_{\varphi} & \mathcal{A}^\wedge(\mathcal{A}/(\blankdot),N_{\mathcal{A}}C) \ar[d]^{\txtint^{\mathcal{A}}} \\
  N_{\mathcal{A}}(\commacat{j}{C}) \ar@{=}[r]^-{\sim} & N_{\mathcal{A}}\txtint^{\mathcal{A}}N_{\mathcal{A}} C }
\]
Thus, by composing $\mathrm{eval}_1:\commacat{j}{C}\to C$, one can verify that the composition
\[
N_{\mathcal{A}}C
\xrightarrow{\eta} N_{\mathcal{A}}\txtint^{\mathcal{A}} N_{\mathcal{A}} C
\xrightarrow{N_{\mathcal{A}}\varepsilon} N_{\mathcal{A}}C
\]
is the identity.
\end{proof}

To compute the counit $\varepsilon:\txtint^{\mathcal{A}}N_{\mathcal{A}}\to\mathrm{Id}_{\mathbf{Cat}}$ of the adjunction, the following lemma is an essential tool:

\begin{lemma}\label{lem:counit-comma}
Let $\mathcal{A}$ be a small category.
Let $\txtint^{\mathcal{A}}:\mathcal{A}^\wedge\xrightleftarrows{}\mathbf{Cat}:N_{\mathcal{A}}$ be the adjoint pair and $\varepsilon:\txtint^{\mathcal{A}}N_{\mathcal{A}}\to\mathrm{Id}_{\mathbf{Cat}}$ be the counit given above.
Then for every small category $C$ and for every object $c\in C$,there is an isomorphism
\[
\commacat{\varepsilon}{c}\simeq \txtint^{\mathcal{A}}N_{\mathcal{A}}(C/c)
\]
of categories.
\end{lemma}
\begin{proof}
Recall that if $D$ is a small category, the category $\int^{\mathcal{A}}N_{\mathcal{A}}D$ can be naturally identified with the category $\commacat{j}{D}$ for $j:\mathcal{A}\ni a\mapsto\mathcal{A}/a\in\mathbf{Cat}$.
Under the identification, $\varepsilon:\txtint^{\mathcal{A}}N_{\mathcal{A}}D\to D$ is identified with the evaluation $\mathrm{eval}_1:\commacat{j}{D}\to D$ at the terminal object.
Hence, the category $\commacat{\varepsilon}{c}$ is the category whose objects are pairs of diagrams
\[
\left(\mathcal{A}/a\xrightarrow{\gamma}C, \gamma(1_a)\xrightarrow{f} c\right)\,.
\]
The morphisms are morphisms $\alpha:a\to a'\in\mathcal{A}$ which commute with structure morphisms.
On the other hand, objects of $\txtint^{\mathcal{A}}N_{\mathcal{A}}(C/c)\simeq\commacat{j}{(C/c)}$ are functors
\[
\mathcal{A}/a \xrightarrow{\widetilde\gamma} C/c\,,
\]
which consist of data $(\mathcal{A}/a\xrightarrow{\gamma} C,\gamma(1_a)\xrightarrow{f} c)$.
Therefore, objects of $\commacat\varepsilon{c}$ correspond in one-to-one to those of $\txtint^{\mathcal{A}}N_{\mathcal{A}}(C/c)$.
This correspondence clearly extends functorially to morphisms, so we obtain the result.
\end{proof}

\begin{defin}
We denote by $N:\mathbf{Cat}\to\mathbf{SSet}$ the usual nerve functor.
Then a functor $\varphi:C\to D$ is called a Thomason equivalence if it induces a weak equivalence $N\varphi:NC\to ND$ in $\mathbf{SSet}$ in the Quillen model structure.
A small category $C$ is said to be aspherical if the unique functor $C\to\ast$ is a Thomason equivalence.
\end{defin}

The terminology comes from the model structure on $\mathbf{Cat}$ introduced by Thomason in \cite{Tho}, see also \cite{Cis2}.
Many people studied about conditions in which a functor becomes a Thomason equivalence.
The most well-known result is the following:

\begin{theo}[Quillen's theorem A]
Let $\varphi:C\to D$ be a functor between small categories
If $\varphi$ is aspherical, i.e. $\commacat{\varphi}{d}$ is aspherical for each $d\in D$, then $\varphi$ is a Thomason equivalence.
\end{theo}

The proof was originally done by Quillen in \cite{Qui2}.
Jardine gave another proof in \cite{Jar2}, where he used the discussion about homotopy colimits in model categories.

Now we define test categories.

\begin{defin}
Let $\mathcal{A}$ be a small category.
\begin{enumerate}[label={\rm(\arabic*)}]
  \item $\mathcal{A}$ is called a weak test category if the counit functor
\[
\varepsilon:\txtint^{\mathcal{A}}N_{\mathcal{A}}C\to C
\]
is a Thomason equivalence.
  \item $\mathcal{A}$ is called a local test category if $\mathcal{A}/a$ is a weak test category for each $a\in\mathcal{A}$.
  \item $\mathcal{A}$ is called a test category if it is both a weak test category and a local test category.
\end{enumerate}
\end{defin}

\begin{prop}\label{prop:wk-loc-test}
Let $\mathcal{A}$ be a small category.
Then the following hold:
\begin{enumerate}[label={\rm(\arabic*)}]
  \item\label{subprop:wktest} $\mathcal{A}$ is a weak test category if and only if for each small category $D$ with terminal object, the category $\txtint^{\mathcal{A}}N_{\mathcal{A}}D$ is aspherical.
  \item\label{subprop:loctest} $\mathcal{A}$ is a local test category if and only if for each small category $D$ with terminal object, the forgetful functor $\txtint^{\mathcal{A}}N_{\mathcal{A}}D\to\mathcal{A}$ is aspherical.
\end{enumerate}
\end{prop}
\begin{proof}
If $\mathcal{A}$ is a weak test category and $D$ is a small category with terminal object, then we have a sequence $\txtint^{\mathcal{A}}N_{\mathcal{A}}D\xrightarrow{\varepsilon}D\to\ast$ of Thomason equivalences.
Conversely, suppose $\txtint^{\mathcal{A}}N_{\mathcal{A}}D$ is aspherical for each small category $D$ with terminal object.
Let $C$ be an arbitrary small category.
For each $c\in C$, by lemma \ref{lem:counit-comma}, we have an isomorphism $\commacat\varepsilon{c}\simeq\txtint^{\mathcal{A}}N_{\mathcal{A}}(C/c)$.
The assumption implies it is aspherical, so the counit $\varepsilon:\txtint^{\mathcal{A}}N_{\mathcal{A}}C\to C$ is aspherical.
Hence by Quillen's theorem A, $\varepsilon$ is a Thomason equivalence, and we obtain the part \ref{subprop:wktest}.

We finally show the part \ref{subprop:loctest}.
For each $a\in\mathcal{A}$ and a small category $C$, objects of the category $\txtint^{\mathcal{A}/a}N_{\mathcal{A}/a}C$ are pairs
\[
\left( a'\xrightarrow{\mu} a, (\mathcal{A}/a)/\mu\xrightarrow{\gamma} C\right)\,.
\]
However, since for $\mu:a'\to a\in\mathcal{A}$, $(\mathcal{A}/a)/\mu$ is isomorphic to $\mathcal{A}/a'$, the category $\txtint^{\mathcal{A}/a}N_{\mathcal{A}/a}C$ is isomorphic to the category $\commacat{(\txtint^{\mathcal{A}}N_{\mathcal{A}}C)}{a}$ whose objects are pairs
\[
\left( a'\xrightarrow{\mu} a, \mathcal{A}/a'\xrightarrow{\gamma} C\right)\,.
\]
Therefore, by the part \ref{subprop:wktest}, $\mathcal{A}/a$ is a weak test category if and only if the category $\commacat{(\txtint^{\mathcal{A}}N_{\mathcal{A}}D)}{a}$ is aspherical for every small category $D$ with terminal object, which says that the forgetful functor $\txtint^{\mathcal{A}}N_{\mathcal{A}}D\to\mathcal{A}$ is aspherical.
\end{proof}

\begin{corol}\label{cor:test-aspherical}
A small category $\mathcal{A}$ is a test category if and only if it is aspherical and a local test category.
\end{corol}
\begin{proof}
First notice that every weak test categories are aspherical; indeed, we have a sequence
\[
\mathcal{A} \simeq \mathcal{A}/\ast \simeq \txtint^{\mathcal{A}}N_{\mathcal{A}}\ast\xrightarrow{\varepsilon}\ast
\]
of Thomason equivalences.

We show the converse.
Suppose $\mathcal{A}$ is aspherical and a local test category.
It suffices to show that $\mathcal{A}$ is a weak test category.
Notice that as shown in the proof of proposition \ref{prop:wk-loc-test}, for every small category $D$ with terminal object and each $a\in\mathcal{A}$, there is an isomorphism
\[
\commacat{(\txtint^{\mathcal{A}}N_{\mathcal{A}}D)}{a} \simeq \txtint^{\mathcal{A}/a}N_{\mathcal{A}/a}D\,.
\]
Since $\mathcal{A}$ is a local test category, by \ref{subprop:wktest} in proposition \ref{prop:wk-loc-test}, they are aspherical.
Then by Quillen's theorem A, the forgetful functor $\txtint^{\mathcal{A}}N_{\mathcal{A}}D\to\mathcal{A}$ is a Thomason equivalence.
In particular, since $\mathcal{A}$ is aspherical, $\txtint^{\mathcal{A}}N_{\mathcal{A}}D$ is aspherical.
Thus by \ref{subprop:wktest} in proposition \ref{prop:wk-loc-test}, $\mathcal{A}$ is a weak test category.
\end{proof}

Notice that if $\mathcal{A}$ has a terminal object, then it is aspherical.
In fact, most of examples of test categories have terminal objects.

Finally, we state the most important theorem about test categories:

\begin{theo}[Corollaire 4.2.18 in \cite{Cis}]
\label{theo:test-model}
If $\mathcal{A}$ is a test category, then there is a cofibrantly generated model structure on $\mathcal{A}^\wedge$ such that
\begin{itemize}
  \item $f:X\to Y$ is a weak equivalence precisely if it is an $\infty$-equivalence.
  \item $f:X\to Y$ is a cofibration precisely if it is a monomorphism.
\end{itemize}
\end{theo}

The proof is omitted here (see \cite{Cis} or \cite{Jar}).

\subsection{Homotopy theory on presheaves on aspherical categories}
Next, we give a sufficient condition for aspherical categories to be a test categories.
The proof is really related to the interval-based homotopy theory.

Let $\mathcal{A}$ be an aspherical category.
We call a morphism $f:X\to Y\in\mathcal{A}^\wedge$ an $\infty$-equivalence if it induces a Thomason equivalence $f_\ast:\mathcal{A}/X\to\mathcal{A}/Y$.
Hence by \ref{subprop:loctest} in proposition \ref{prop:wk-loc-test} and corollary \ref{cor:test-aspherical}, $\mathcal{A}$ is a test category if and only if the morphism $N_{\mathcal{A}}D\to\ast\in\mathcal{A}^\wedge$ is an $\infty$-equivalence for every small category $D$ with terminal object.
We regard $\mathcal{A}^\wedge$ as a homotopical category with the class of $\infty$-equivalences.

We also say a morphism $f:X\to Y\in\mathcal{A}^\wedge$ is aspherical if it induces an aspherical morphism $f_\ast:\mathcal{A}/X\to\mathcal{A}/Y$.
An object $X\in\mathcal{A}^\wedge$ is said to be aspherical if the unique morphism $X\to\ast$ is aspherical.
By Quillen's theorem A, that $f$ is aspherical implies $f$ is an $\infty$-equivalence.

\begin{lemma}\label{lem:aspheric-prod-slice}
Let $\mathcal{A}$ be an aspherical category.
Then a morphism $f:X\to Y\in\mathcal{A}$ is aspherical if and only if for each cell $y:\mathcal{A}[a]\to Y$, the category $\mathcal{A}/(X\times_Y\mathcal{A}[a])$ is aspherical.
In particular, $X\in\mathcal{A}$ is aspherical if and only if $\mathcal{A}/(X\times\mathcal{A}[a])$ is aspherical for each $a\in\mathcal{A}$.
\end{lemma}
\begin{proof}
The last part follows from the first one.
Let $f:X\to Y\in\mathcal{A}^\wedge$ be a morphism.
Then the induced functor $f_\ast:\mathcal{A}/X\to\mathcal{A}/Y$ is aspherical if and only if for each cell $y:\mathcal{A}[a]\to Y$, the category $\commacat{f_\ast}{y}$ is aspherical.
Objects of $\commacat{f_\ast}{y}$ are written as commutative squares of the form:
\[
\xymatrix{
  \mathcal{A}[a'] \ar[r] \ar[d] & \mathcal{A}[a] \ar[d]^{y} \\
  X \ar[r]^{f} & Y }
\]
Clearly they correspond one-to-one to morphism $\mathcal{A}[a']\to X\times_Y\mathcal{A}[a]\in\mathcal{A}^\wedge$.
Thus, we obtain an isomorphism $\commacat{f_\ast}{y}\simeq\mathcal{A}/(X\times_Y\mathcal{A}[a])$, and the first part follows.

The last part follows from the first one.
\end{proof}

\begin{corol}\label{cor:aspheric-vs-wkequiv}
Let $\mathcal{A}$ be an aspherical category.
Then a morphism $f:X\to Y\in\mathcal{A}^\wedge$ is aspherical if and only if for each cell $y:\mathcal{A}[a]\to Y$, the projection $X\times_Y\mathcal{A}[a]\to\mathcal{A}[a]$ is an $\infty$-equivalence.
In particular, an object $X\in\mathcal{A}^\wedge$ is aspherical if and only if the projection $X\times\mathcal{A}[a]\to\mathcal{A}[a]$ is an $\infty$-equivalence for each $a\in\mathcal{A}$.
\end{corol}
\begin{proof}
The condition that $X\times_Y\mathcal{A}[a]\to\mathcal{A}[a]$ is an $\infty$-equivalence says that the induced functor
\[
\mathcal{A}/(X\times_Y\mathcal{A}[a])\to\mathcal{A}/a
\]
is a Thomason equivalence.
Since $\mathcal{A}[a]/a$ has a terminal object, hence aspherical, this holds if and only if $\mathcal{A}/(X\times_Y\mathcal{A}[a])$ is aspherical.
Thus the result follows from lemma \ref{lem:aspheric-prod-slice}.
\end{proof}

\begin{corol}\label{cor:product-aspherical}
Let $\mathcal{A}$ be an aspherical category.
If $f:X\to Y\in\mathcal{A}^\wedge$ is aspherical, then for every $W\in\mathcal{A}^\wedge$, the morphism $f\times W:X\times W\to Y\times W$ is aspherical.
In particular, if an object $X\in\mathcal{A}^\wedge$ is aspherical, the projection $X\times W\twoheadrightarrow W$ is aspherical for every $W\in\mathcal{A}^\wedge$.
\end{corol}
\begin{proof}
By lemma \ref{lem:aspheric-prod-slice}, it suffices to show that the category
\[
\mathcal{A}/\left((X\times W)\times_{(Y\times W)}\mathcal{A}[a]\right)
\]
is aspherical for each cell $(y,w):\mathcal{A}[a]\to Y\times W$.
Notice that we have the following pullback square:
\[
\xymatrix{
  X\times_Y\mathcal{A}[a] \ar[r] \ar[d] \ar@{}[dr]|(.4){\pbcorner} & \mathcal{A}[a] \ar[d]^{(y,w)} \\
  X\times W \ar[r]^{f\times W} & Y\times W }
\]
Thus we obtain $(X\times W)\times_{(Y\times W)}\mathcal{A}[a]\simeq X\times_Y\mathcal{A}[a]$, which induces an isomorphism
\[
\mathcal{A}/\left((X\times W)\times_{(Y\times W)}\mathcal{A}[a]\right)\simeq\mathcal{A}/(X\times_Y\mathcal{A}[a])\,.
\]
Since $f:X\to Y$ is aspherical, the right hand side is aspherical, so we obtain the result.
\end{proof}

Now fix an aspherical category $\mathcal{A}$.
Suppose we have a functor $\theta:\mathcal{A}\to\mathbf{Cat}$.
We denote by $\theta^\ast:\mathbf{Cat}\to\mathcal{A}^\wedge$ the functor defined by
\[
\theta^\ast(C):=\mathbf{Cat}(\theta(\blankdot),C)\,.
\]
The functor $\theta^\ast$ is right adjoint to the Yoneda extension $\theta_!:\mathcal{A}^\wedge\to\mathbf{Cat}$ of $\theta$.
Hence $\theta^\ast$ preserves limits, in particular it preserves cartesian products and the terminal object.
It follows that under certain assumptions, the functor $\theta^\ast:\mathbf{Cat}\to\mathcal{A}^\wedge$ derives an interval theory on $\mathcal{A}^\wedge$ from $\mathbf{Cat}$.
This is the key observation to obtain the following result.

\begin{prop}\label{prop:presentation-test}
Let $\mathcal{A}$ be an aspherical category.
Suppose we have a functor $\theta:\mathcal{A}\to\mathbf{Cat}$ such that
\begin{enumerate}[label={\rm(\alph*)}]
  \item for each $a\in\mathcal{A}$, the category $\theta(a)$ has a terminal object;
  \item if a small category $D$ has a terminal object, then the object $\theta^\ast D\in\mathcal{A}^\wedge$ is aspherical.
\end{enumerate}
Then $\mathcal{A}$ is a test category.
\end{prop}
\begin{proof}
By corollary \ref{cor:test-aspherical} and \ref{subprop:loctest} in proposition \ref{prop:wk-loc-test}, it suffices to show that for each small category $D$ with terminal object, $N_{\mathcal{A}}D$ is aspherical.
By corollary \ref{cor:aspheric-vs-wkequiv}, it also suffices to show that for each $a\in\mathcal{A}$, the projection $N_{\mathcal{A}}D\times\mathcal{A}[a]\twoheadrightarrow\mathcal{A}[a]$ is an $\infty$-equivalence.

We denote by $i_0,i_1:\ast\to[1]\in\mathbf{Cat}$ the two inclusions whose images are $0$ and $1$ respectively.
The assumptions and corollary \ref{cor:product-aspherical} imply that for each $X\in\mathcal{A}^\wedge$, the sequence
\[
X\amalg X
\simeq (X\times\theta^\ast\ast)\amalg(X\times\theta^\ast\ast)
\xrightarrow{(1\times\theta^\ast i_0,1\times\theta^\ast i_1)} X\times\theta^\ast[1]
\xrightarrow{\text{proj.}} X
\]
gives rise to a cylinder object $X\times\theta^\ast[1]$ on $X$ in the sense of the section \ref{sec:H-cats}.
Thus to see the projection $N_{\mathcal{A}}D\times\mathcal{A}[a]\twoheadrightarrow\mathcal{A}[a]$ is an $\infty$-equivalence, by corollary \ref{cor:H-inverse}, it suffices to show that for each $a\in\mathcal{A}$, it admits a homotopy inverse with respect to the cylinder object $N_{\mathcal{A}}D\times\theta^\ast[1]$.

We define $t:\ast\to N_{\mathcal{A}}D$ as the adjoint morphism to the functor $\mathcal{A}\to D$ which sends all objects to a terminal object.
We show the morphism $t$ gives rise to required homotopy inverses.
Since $D$ has a terminal object, there is the obvious homotopy $H:D\times[1]\to D$ from the identity functor to the contraction to a terminal object.
Then we obtain the sequence
\[
\begin{split}
\mathcal{A}[a]\times N_{\mathcal{A}}D\times\theta^\ast[1]
&\xrightarrow{1\times 1\times\beta} \mathcal{A}[a]\times N_{\mathcal{A}}D\times N_{\mathcal{A}}[1]
\simeq \mathcal{A}[a]\times N_{\mathcal{A}}(D\times[1]) \\
&\xrightarrow{1\times N_{\mathcal{A}}(H)} \mathcal{A}[a]\times N_{\mathcal{A}}D\,,
\end{split}
\]
where $\beta:\theta^\ast\to N_{\mathcal{A}}$ is the natural transformation induced by the functor $\mathcal{A}/a\to\theta(a)$ which sends each $(a'\xrightarrow{f} a)\in\mathcal{A}/a$ to the image of the terminal object of $\theta(a')$ by $\theta(f):\theta(a')\to\theta(a)$.
Since $\beta:\theta^\ast\ast\to N_{\mathcal{A}}\ast$ is naturally isormophic to the identity on the terminal object of $\mathcal{A}^\wedge$, the composition above gives rise to a homotopy on $\mathcal{A}[a]\times N_{\mathcal{A}}D\times\theta^\ast[1]$ from the identity to the contraction
\[
\mathcal{A}[a]\times N_{\mathcal{A}}D
\twoheadrightarrow\mathcal{A}[a]
\xrightarrow{1\times t}\mathcal{A}[a]\times N_{\mathcal{A}}D\,.
\]
Hence, by corollary \ref{cor:H-inverse}, the projection $N_{\mathcal{A}}D\times\mathcal{A}[a]\twoheadrightarrow\mathcal{A}[a]$ is an $\infty$-equivalence.
The result now follows from corollary \ref{cor:aspheric-vs-wkequiv}.
\end{proof}

\begin{exam}
Let $\Delta$ be the simplex category.
Then we have an embedding $\theta:\Delta\hookrightarrow\mathbf{Cat}$.
The induced functor $\theta^\ast:\mathbf{Cat}\to\Delta\mathbf{Set}=\mathbf{SSet}$ is nothing but the usual nerve functor.
Hence the conditions in proposition \ref{prop:presentation-test} are obviously satisfied, so $\Delta$ is a test category.
\end{exam}

\subsection{Cubicalizations and test categories}
Finally, we show that some cubicalizations are test categories.
Before that, we have to consider a structure on a cubicalizable category which induces a cylinder-like structure on its cubicalization.

\begin{defin}
Let $\mathcal{R}$ be a cubicalizable category.
Then an enlargement on $\mathcal{R}$ is an endofunctor $c:\mathcal{R}\to\mathcal{R}$ together with a natural transformation $\iota:\mathrm{Id}\to c$ such that
\begin{enumerate}[label={\rm(\arabic*)}]
  \item for each $r\in\mathcal{R}$, $\iota:r\to c(r)\in\mathcal{R}^+_0$ and $\operatorname{codeg}\iota_r=2$;
  \item $c$ preserves $(\mathcal{R}^+_0,\mathcal{R}^-)$-factorization; i.e. if $\sigma\in\mathcal{R}^-$ and $\delta\in\mathcal{R}^+_0$, then $c(\sigma)\in\mathcal{R}^-$ and $c(\delta)\in\mathcal{R}^+_0$;
  \item for each $f:r\to r'\in\mathcal{R}$, the square
\[
\xymatrix{
  r \ar[r]^{f} \ar[d]_{\iota} \ar@{}[dr]|(.4){\pbcorner} & r' \ar[d]^{\iota} \\
  c(r) \ar[r]^{c(f)} & c(r') }
\]
is a pullback.
\end{enumerate}
\end{defin}

Notice that the last condition implies that for each $r\in\mathcal{R}$, $c:\mathcal{R}^+_0\!/r\to\mathcal{R}^+_0\!/c(r)$ is a section of $\iota^\ast:\mathcal{R}^+_0\!/c(r)\to\mathcal{R}^+_0\!/r$.
Actually, we can write this map $c:\mathcal{R}^+_0\!/r\to\mathcal{R}^+_0\!c(r)$ explicitly as follows:

\begin{lemma}\label{lem:enlarge-explicit}
Let $\mathcal{R}$ be a cubicalizable category, and let $c:\mathcal{R}\to\mathcal{R}$ be an enlargement.
Then for each $\delta\in\mathcal{R}^+_0\!/r$, we have
\[
c(\delta) = \iota_\ast\delta\vee\neg\iota\,.
\]
\end{lemma}
\begin{proof}
Since $\operatorname{codeg}\iota=2$, we have either $\neg\iota\le c(\delta)$ or $c(\delta)\le\iota$.
If $c(\delta)\le\iota$, then we have $c(\delta)\wedge\iota=c(\delta)$; i.e. the following square is a pullback:
\[
\xymatrix{
  c(s) \ar[r] \ar@{=}[d] \ar@{}[dr]|(.4){\pbcorner} & r \ar[d]^{\iota} \\
  c(s) \ar[r]^{c(\delta)} & c(r) }
\]
This contradicts to the definition of enlargements.
So we have $\neg\iota\le c(\delta)$.
Then we have
\[
c(\delta)
= (\iota\wedge c(\delta)) \vee (\neg\iota\wedge c(\delta))
= (\iota_\ast\iota^\ast c(\delta)) \vee \neg\iota
= \iota_\ast\delta \vee \neg\iota\,.
\]
\end{proof}

\begin{corol}\label{cor:enlarge-pb}
In the above situation, for every morphism $f:r'\to r\in\mathcal{R}$ and $\delta\in\mathcal{R}^+_0\!/r$, we have
\[
c(f)^\ast c(\delta)
= \iota_\ast f^\ast(\delta)\vee\neg\iota
= c(f^\ast\delta)\,.
\]
Consequently, the functor $c:\mathcal{R}\to\mathcal{R}$ preserves pullback squares of morphisms in $\mathcal{R}^+_0$.
\end{corol}
\begin{proof}
First notice that, by definition, we have
\[
\sigma(f)^\ast(\neg\iota) = \neg\sigma(f)^\ast(\iota) = \neg\iota\,.
\]
Thus for $\delta'\in\mathcal{R}^+_0\!/c(r)$, we have
\[
\begin{multlined}
c(f)^\ast(\delta')
= (c(f)^\ast(\delta')\wedge\iota)\vee(c(f)^\ast(\delta')\wedge\neg\iota) \\
= (\iota_\ast\iota^\ast c(f)^\ast(\delta'))\vee c(f)^\ast(\delta'\wedge\neg\iota)
= (\iota_\ast f^\ast\iota^\ast(\delta'))\vee c(f)^\ast(\delta'\wedge\neg\iota)\,.
\end{multlined}
\]
In particular, the case $\delta'=c(\delta)$, by lemma \ref{lem:enlarge-explicit}, $\neg\iota\le c(\delta)$, and we obtain
\[
c(f)^\ast(c(\delta))
= (\iota_\ast f^\ast\iota^\ast(\iota_\ast\delta\vee\neg\iota))\vee c(f)^\ast(\neg\iota)
= (\iota_\ast f^\ast(\delta))\vee\neg\iota\,.
\]
This is the first part.

Now setting $f'$ to be the pullback of $f$ by $\delta$, we have the following commutative square:
\[
\xymatrix{
  c(s') \ar[r]^{c(f')} \ar[d]_{c(f^\ast(\delta))} & c(s) \ar[d]^{c(\delta)} \\
  c(r') \ar[r]^{f} & c(r) }
\]
Since $c(\delta)$ is a monomorphism, the first part and the uniqueness of the pullback imply that the square is actually a pullback.
Thus the last part follows.
\end{proof}

\begin{exam}
Recall that $\widetilde\Delta$ is the category of finite ordinals and order-preservin maps.
For $\underline{n}\in\widetilde\Delta$, set $c(\underline{n}):=\underline{n+1}$ and $\iota:\underline{n}\to c(\underline{n})$ to be the inclusion.
Then $c$ and $\iota$ form an enlargement on $\widetilde\Delta$.
\end{exam}

We see that enlargements naturally induce a structure on the cubicalization.
The following two lemmas are key results:

\begin{lemma}\label{lem:cube-cyl-c}
Let $\mathcal{R}$ be a cubicalizable category and $c:\mathcal{R}\to\mathcal{R}$ be an enlargement.
Then $c$ extends to $\square(\mathcal{R})\to\square(\mathcal{R})$ by setting
\[
c(\gamma^\dagger) = c(\gamma)^\dagger
\quad,\qquad
c(\delta^\xi) = c(\delta)^{\iota_\ast\xi}\,.
\]
for $\gamma,\delta,\xi\in\mathcal{R}^+_0$ such that the above makes sense.
\end{lemma}
\begin{proof}
First, notice that since $\iota$ is preserved by the pullback by $c(\gamma)$, we have
\[
c(\gamma)^\ast\iota_\ast = \iota_\ast\gamma^\ast\,,
\]
and by corollary \ref{cor:enlarge-pb}, we have
\[
c(\gamma)^\dagger c(\delta)
= c(\gamma^\dagger\delta)\,.
\]
Thus by \ref{sublem:cube-r0span} of lemma \ref{lem:cube-r0}, for every $f:r'\to r\in\mathcal{R}$ and $\xi\in\mathcal{R}^+_0\!/r$ with $\xi\wedge\image(f)=0$, we obtain
\[
\begin{multlined}
c(\gamma^\dagger)c(\xi,f)
= c(\gamma)^\dagger (\iota_\ast\xi,c(f))
= (c(\gamma)^\ast\iota_\ast\xi, c(\gamma)^\dagger c(f)) \\
= (\iota_\ast\gamma^\ast\xi, c(\gamma^\dagger f))
= c(\gamma^\ast\xi,\gamma^\dagger f)
= c(\gamma^\dagger (\xi,f))\,.
\end{multlined}
\]

On the other hand, we have
\[
\begin{multlined}
c(\delta^\xi)c(\varepsilon^\zeta)
= c(\delta)^{\iota_\ast\xi} c(\varepsilon)^{\iota_\ast\zeta}
= (\iota_\ast\xi\vee c(\delta)_\ast\iota_\ast\zeta,c(\delta)c(\varepsilon)) \\
= (\iota_\ast(\xi\vee\delta_\ast\zeta),c(\delta)c(\varepsilon))
= c(\delta^\xi\varepsilon^\zeta)\,.
\end{multlined}
\]
Hence the result follows.
\end{proof}

\begin{lemma}\label{lem:cube-cyl}
In the situation above, the following hold:
\begin{enumerate}[label={\rm(\arabic*)}]
  \item\label{sublem:enlarge-span-pb} For every $\gamma:r'\to r\in\mathcal{R}^+_0$, the diagram
\[
\xymatrix{
  r \ar[r]^{\gamma^\dagger} \ar[d]_{\iota} & r' \ar[d]^{\iota} \\
  c(r) \ar[r]^{c(\gamma^\dagger)} & c(r') }
\]
is a pullback in $\mathcal{V}(\mathcal{R})$.
  \item\label{sublem:enlarge-natcube} The diagram
\[
r \xrightrightarrows[\iota^{\neg\iota}]{\iota} c(r) \xrightarrow{\iota^\dagger} r
\]
in $\square(\mathcal{R})$ is natural with respect to $r\in\square(\mathcal{R})$.
\end{enumerate}
\end{lemma}
\begin{proof}
Suppose $\gamma:r'\to r\in\mathcal{R}^+_0$.
By the proof of proposition \ref{prop:span-thinpow}, if we set $\delta:=c(\gamma)\iota\vee\neg c(\gamma)$, then the pullback square is given as
\begin{equation}
\label{eq:enlarge-span-pb}
\xymatrix{
  s' \ar[r]^{(\delta^\ast(c(\gamma)\iota))^\dagger} \ar[d]_{\delta} \ar@{}[dr]|(.4){\pbcorner} & r \ar[d]^{\iota} \\
  c(r') \ar[r]^{c(\gamma^\dagger)} & c(r) }
\end{equation}
By lemma \ref{lem:enlarge-explicit}, we have
\[
\delta
= c(\gamma)\iota\vee\neg c(\gamma)
= (\iota\gamma) \vee (\neg(\iota_\ast\gamma)\wedge\iota)
= \iota_\ast(\gamma\vee\iota^\ast(\neg\iota_\ast\gamma))
= \iota_\ast(\gamma\vee\neg\gamma)
= \iota\,.
\]
Hence
\[
(\delta^\ast(c(\gamma)\iota))^\dagger
= (\iota^\ast\iota_\ast\gamma)^\dagger
= \gamma^\dagger\,,
\]
so that the square \eqref{eq:enlarge-span-pb} is in fact the required one in \ref{sublem:enlarge-span-pb}.

Now we show \ref{sublem:enlarge-natcube}.
Let $\delta^\xi\sigma\gamma^\dagger:r'\to r\in\square(\mathcal{R})$, and let $\eta'=0$ or $\neg\iota\in\mathcal{R}^+_0\!/c(r')$.
Put $\eta=c(\delta^\xi\sigma\gamma^\dagger)_\ast\eta'$.
Then by part \ref{sublem:enlarge-span-pb}, we have
\[
\iota\wedge\eta
= \iota\wedge c(\delta\sigma\gamma^\dagger)_\ast\eta'
= \iota\wedge\eta'
= 0\,.
\]
Hence $\eta\le\neg\iota$, so that $\eta$ values $0$ and $\neg\iota$ precisely when $\eta$ is $0$ and $\neg\iota$ respectively.
We obtain
\[
c(\delta^\xi\sigma\gamma^\dagger)\iota^{\eta'}
= ((\iota_\ast\xi)\vee\eta, \Lambda^\complement(\eta) c(\delta\sigma\gamma^\dagger)\iota)
= ((\iota_\ast\xi)\vee\eta, \Lambda^\complement(\eta) \iota\delta\sigma\gamma^\dagger)
= \iota^\eta\delta^\xi\sigma\gamma^\dagger\,.
\]
Thus $\iota^\eta$ is natural.

On the other hand, by part \ref{sublem:enlarge-span-pb}, we have
\[
\iota^\dagger c(\delta\sigma\gamma^\dagger)
= \delta\sigma\gamma^\dagger\iota^\dagger\,,
\]
in $\mathcal{V}(\mathcal{R})$, hence
\[
\iota^\dagger c(\delta^\xi\sigma\gamma^\dagger)
= \iota^\dagger c(\delta)^{\iota_\ast\xi}c(\sigma\gamma^\dagger)
= (\iota^\ast\iota_\ast\xi,\iota^\dagger c(\delta)) c(\sigma\gamma^\dagger)
= \delta^\xi\iota^\dagger c(\sigma\gamma^\dagger)
= \delta^\xi\sigma\gamma^\dagger\iota^\dagger\,.
\]
This implies that $\iota^\dagger$ is also natural.
\end{proof}

We are now ready to state and prove one of the main results:

\begin{theo}\label{theo:cube-is-test}
Suppose $\mathcal{R}$ is a cubicalizable category with an enlargement $c:\mathcal{R}\to\mathcal{R}$.
Then the cubicalization $\square(\mathcal{R})$ is a test category.
\end{theo}
\begin{proof}
Recall that every cubicalizable category $\mathcal{R}$ is canonically equipped with a canonical poset representation which extends to
\[
\mathcal{R}^+_0\!/(\blankdot):\square(\mathcal{R})\to\mathbf{Poset}\hookrightarrow\mathbf{Cat}\,.
\]
We put $\theta:=\mathcal{R}^+_0\!/(\blankdot)$.
Then since $\square(\mathcal{R})$ has a terminal object by \ref{sublem:cube-initial} in lemma \ref{lem:cube-r0}, to see $\square(\mathcal{R})$ is a test category, it suffices to show that $\theta$ satisfies the conditions in proposition \ref{prop:presentation-test}.

Since for each $r\in\square(\mathcal{R})$, the category $\theta(r)$ is a finite lattice so that it has a terminal object.
So we verify the other condition.
Suppose $D$ is a small category with terminal object.
To see $\theta^\ast D$ is aspherical, by corollary \ref{cor:aspheric-vs-wkequiv}, it suffices to show that the projection
\[
\theta^\ast_{\mathcal{R}}D\times\square[r] \twoheadrightarrow \square[r]
\]
is an $\infty$-equivalence.
More precisely, we shall directly show that the induced functor
\[
\square(\mathcal{R})/(\theta^\ast_{\mathcal{R}}D\times\square[r]) \to \square(\mathcal{R})/r
\]
is a Thomason equivalence.
Notice that we have an isomorphism
\[
\square(\mathcal{R})/(\theta^\ast_{\mathcal{R}}D\times\square[r])
\simeq (\commacat{\theta}{D})\times_{\square(\mathcal{R})}(\square(\mathcal{R})/r)\,.
\]
We define a functor $h_D:\commacat{\theta}{D}\to\commacat{\theta}{D}$ as follows:
Fix a terminal object $t\in D$.
Then for each $(\theta(r)\xrightarrow{\rho} D)\in\commacat\theta{D}$, define $h_D(\rho)$ to be the obvious homotopy
\[
\theta(c(r))\simeq\mathcal{R}^+_0\!/c(r)\simeq \mathcal{R}^+_0\!/r\times[1]\to D
\]
from $\rho$ to the constant functor at $t\in D$.
We then obtain the following commutative diagram of functors
\[
\xymatrix{
  \commacat{\theta}{D} \ar[r] \ar[d]_{h_D} & \square(\mathcal{R}) \ar[d]_{c} & \square(\mathcal{R})/r \ar[l] \ar[d]^{(\iota^\dagger)^\ast} \\
  \commacat{\theta}{D} \ar[r] & \square(\mathcal{R}) & \square(\mathcal{R})/r \ar[l] }
\]
which induces a functor $h:\square(\mathcal{R})/(\theta^\ast_{\mathcal{R}}D\times\square[r])\to\square(\mathcal{R})/(\theta^\ast_{\mathcal{R}}D\times\square[r])$.

Since an enlargenemt gives rise to an interval theory on the cubicalization, we have a diagram
\[
\mathrm{Id}_{\square(\mathcal{R})}
\xrightrightarrows[\iota^1]{\iota^0} c
\xrightarrow{\iota^\dagger} \mathrm{Id}_{\square(\mathcal{R})}
\]
of natural transformations.
They also induce natural transformations
\[
\begin{gathered}
\mathrm{Id}_{\commacat\theta{D}}\xrightarrow{\iota^0_\ast} h_D \xleftarrow{\iota^1_\ast} t \\
\mathrm{Id}_{\square(\mathcal{R})/r} \xrightarrow{\iota^0} (\iota^\dagger)^\ast \xleftarrow{\iota^1_\ast} \mathrm{Id}_{\square(\mathcal{R})/r}
\end{gathered}
\]
where we identify $t\in D$ with the constant functor $\theta\to D$.
Hence we obtain natural transformations
\begin{equation}\label{eq:proj-H-inv}
\mathrm{Id} \xrightarrow{} h \xleftarrow{} (t\times\square[r])_\ast
\end{equation}
between functors $\square(\mathcal{R})/(\theta^\ast_{\mathcal{R}}D\times\square[r])\to\square(\mathcal{R})/(\theta^\ast_{\mathcal{R}}D\times\square[r])$, here we denote by $t\times\square[r]$ the composition
\[
\square(\mathcal{R})/(\theta^\ast_{\mathcal{R}}D\times\square[r])
\xrightarrow{\text{proj}_\ast} \square(\mathcal{R})/r
\xrightarrow{(t\times\square[r])_\ast} \square(\mathcal{R})/(\theta^\ast_{\mathcal{R}}D\times\square[r])\,.
\]
Thus the natural transformations in \eqref{eq:proj-H-inv} imply that $\square(\mathcal{R})/(\theta^\ast_{\mathcal{R}}D\times\square[r])\to\square(\mathcal{R})/r$ is a Thomason equivalence by lemma \ref{lem:homotopic-weq} and the two-out-of-six property.

Now, we have a functor $\theta=\theta:\square(\mathcal{R})\to\mathbf{Cat}$ satisfying the conditions in proposition \ref{prop:presentation-test}, so the result follows.
\end{proof}

As a result, we obtain a model structure:

\begin{corol}
\label{cor:test-model-cube}
Suppose $\mathcal{R}$ is a cubicalizable category with an enlargement $c:\mathcal{R}\to\mathcal{R}$.
Then there is a cofibrantly generated model structure on $\square(\mathcal{R})^\wedge$ such that
\begin{itemize}
  \item $f:X\to Y$ is a weak equivalence precisely if it is an $\infty$-equivalence.
  \item $f:X\to Y$ is a cofibration precisely if it is a monomorphism.
\end{itemize}
We call it the standard model structure.
\end{corol}

\section{Homotpy theory on cubicalized categories}
\label{sec:htp-cube}
In the previous sections, we saw that if $\mathcal{R}$ is a cubicalizable category with an enlargement, the cubicalization $\square(\mathcal{R})$ is a test category.
It follows that the category $\square(\mathcal{R})^\wedge$ canonically admits a model structure.
In this section, we give another model structure on it in the case $\mathcal{R}$ is EZ cubicalizable.
Recall that if $\mathcal{R}$ is a cubicalizable category, then we have a simplicial realization functor $|\blankdot|:\square(\mathcal{R})^\wedge\to\mathbf{SSet}$.
We see that the functor induces a model structure.
It is non-trivial that the model structure coincides with the standard model structure, but we can actually prove it in some cases containing $\square(\widetilde\Delta)=\square^{\mathsf{c}}$ the cubical category with connections.

\subsection{Homotopy theory based on cylinders}
\label{subsec:homotopy-cyl}
First, we discuss cylinders on $\square(\mathcal{R})^\wedge$.
Fix a cubicalizable category $\mathcal{R}$ and an enlargement $c:\mathcal{R}\to\mathcal{R}$ on it.
Recall that by lemma \ref{lem:cube-cyl}, we obtain a functor $c:\square(\mathcal{R})\to\square(\mathcal{R})$ and a diagram
\[
\mathrm{Id} \xrightrightarrows[\iota^1]{\iota^0} c \xrightarrow{\iota^\dagger} \mathrm{Id}
\]
of functors $\square(\mathcal{R})\to\square(\mathcal{R})$, where we write $\iota^1=\iota^{\neg\iota}$ and $\iota^0=\iota$.
Define a functor $\mathrm{Cyl}:\square(\mathcal{R})^\wedge\to\square(\mathcal{R})^\wedge$ to be the left Kan extension of $c$.
We call $\mathrm{Cyl}$ the cylinder functor induced from the enlargement $c$.
The above diagram induces
\begin{equation}
\label{eq:cylcyl}
\mathrm{Id} \xrightrightarrows[\iota^1]{\iota^0} \mathrm{Cyl} \xrightarrow{\iota^\dagger} \mathrm{Id}\,.
\end{equation}
Notice that the proof of lemma \ref{lem:cube-cyl} also implies that
\[
(\mathcal{R}^+_0\!/r)\times[1]\ni(\xi,k)\mapsto(\iota^k)_\ast\xi\in\mathcal{R}^+_0\!/c(r)
\]
is a natural isomorphism with respect to $r\in\square(\mathcal{R})$.
It gives rise to a natural isomorphism $|\square[c(r)]|\simeq|\square[r]|\times\Delta[1]$, and moreover for $X\in\square(\mathcal{R})^\wedge$, we obtain
\[
\begin{split}
|\operatorname{Cyl}X|
&\simeq \int^{r\in\square(\mathcal{R})} X(r)\times |\square[c(r)]| \\
&\simeq \int^{r\in\square(\mathcal{R})} X(r)\times |\square[r]|\times\Delta[1] \\
&\simeq |X|\times\Delta[1]\,.
\end{split}
\]
Under this isomorphism, the simplicial realization of diagram \eqref{eq:cylcyl} is isomorphic to the diagram
\[
|X| \xrightrightarrows[(\mathrm{id},1)]{(\mathrm{id},0)} |X|\times\Delta[1] \overset{\text{proj}}{\twoheadrightarrow} |X|\,.
\]

The discussion above implies the following result:

\begin{prop}\label{prop:cyl-homotopy}
Let $\mathcal{R}$ be a cubicalizable category together with a fixed enlargement $c:\mathcal{R}\to\mathcal{R}$.
Suppose $f_0,f_1:X\to Y\in\square(\mathcal{R})^\wedge$ are morphisms such that there is a morphism $h:\operatorname{Cyl}X\to Y$ with $h\iota^k = f_k$ for $k=0,1$.
Then $|f_0|,|f_1|:|X|\to |Y|$ are homotopic in $\mathbf{SSet}$.
\end{prop}

\begin{defin}
Let $\mathcal{R}$ be a cubicalizable category, and let $c:\mathcal{R}\to\mathcal{R}$ be an enlargement.
Then two morphisms $f_0,f_1:X\to Y\in\square(\mathcal{R})^\wedge$ are said to be homotopic with respect to $c$ if there is a morphism $h:\operatorname{Cyl}X\to Y$ with $h\iota^k=f_k$ for $k=0,1$.
A morphism $f:X\to Y\in\square(\mathcal{R})^\wedge$ is called a homotopy equivalence if there is a morphism $g:Y\to X$ such that $gf$ and $fg$ are homotopic to the identities $\mathrm{id}_X$ and $\mathrm{id}_Y$ respectively.
\end{defin}

Thanks to the two-out-of-six property, proposition \ref{prop:cyl-homotopy} also implies the following:

\begin{corol}
Let $\mathcal{R}$ be a cubicalizable category together with a fixed enlargement $c:\mathcal{R}\to\mathcal{R}$.
Then every homotopy equivalence $f:X\to Y\in\square(\mathcal{R})^\wedge$ induces a weak equivalence $|f|:|X|\to |Y|$ in $\mathbf{SSet}$.
\end{corol}

Recall that if $\mathcal{R}$ is a cubicalizable category with a thin-powered structure $\mathcal{R}^+_0$, then $\mathcal{R}^+_0$ is itself cubicalizable, and $\square(\mathcal{R}^+_0)$ is a wide subcategory of $\square(\mathcal{R})$.
We denote by $i:\square(\mathcal{R}^+_0)\hookrightarrow\square(\mathcal{R})$ the inclusion.
Moreover, if $c:\mathcal{R}\to\mathcal{R}$ is an enlargement, then it restricts to an enlargement on $\mathcal{R}^+_0$, which we also denote by $c$.
Then we have $i c=c i$ and $i(\iota)=\iota$.
We denote by $i^\ast:\square(\mathcal{R})\to\square(\mathcal{R}^+_0)$ the restriction functor.
Then we can compare the homotopy theories on $\square(\mathcal{R})^\wedge$ and $\square(\mathcal{R}^+_0)^\wedge$ in the following sense.

\begin{prop}
Let $\mathcal{R}$ be a cubicalizable category, and let $c:\mathcal{R}\to\mathcal{R}$ be an enlargement.
Suppose morphisms $f_0,f_1:X\to Y\in\square(\mathcal{R})^\wedge$ are homotopic with respect to $c$.
Then $i^\ast f_0,i^\ast f_1$ are homotopic with respect to the restricted enlargement.
\end{prop}
\begin{proof}
Let $i_!:\square(\mathcal{R}^+_0)^\wedge\to\square(\mathcal{R})^\wedge$ be the left Kan extension of $i$.
Notice first that since $ic=ci$, we have a natural isomorphism $i_!\operatorname{Cyl}\simeq\operatorname{Cyl}i_!$ and the following diagram of natural transformations:
\[
\xymatrix{
  i_! \ar@<-.5ex>[r]_{i_!(\iota^1)} \ar@<.5ex>[r]^{i_!(\iota^0)} \ar@{=}[d] & i_!\operatorname{Cyl} \ar[r]^{i_!(\iota^\dagger)} \ar[d]^{\simeq} & i_! \ar@{=}[d] \\
  i_! \ar@<-.5ex>[r]_{\iota^1} \ar@<.5ex>[r]^{\iota^0} & \operatorname{Cyl}i_! \ar[r]^{\iota^\dagger} & i_! }
\]
Taking the adjoint diagram, we obtain the following diagram:
\[
\xymatrix{
  i^\ast \ar@<-.5ex>[r] \ar@<.5ex>[r] \ar[d]_{\eta} & \operatorname{Cyl}i^\ast \ar[r] \ar[d] & i^\ast \ar[d]^{\eta} \\
  i^\ast i_! i^\ast \ar@<-.5ex>[r] \ar@<.5ex>[r] \ar[d]_{i^\ast\varepsilon} & i^\ast \operatorname{Cyl} i_! i^\ast \ar[r] \ar[d]^{i^\ast\operatorname{Cyl}\eta} & i^\ast i_! i^\ast \ar[d]^{i^\ast\varepsilon} \\
  i^\ast \ar@<-.5ex>[r] \ar@<.5ex>[r] & i^\ast \operatorname{Cyl} \ar[r] & i^\ast }
\]
where $\eta:\mathrm{Id}\to i^\ast i_!$ and $\varepsilon:i_!i^\ast\to\mathrm{Id}$ are the unit and counit of the adjunction $i_!\dashv i^\ast$ respectively.
Then, by the triangle identities, the compositions of left and right edges are the identities.
Hence the result follows.
\end{proof}

\begin{corol}\label{cor:res-h-pres}
In the situation above, the functor $i^\ast:\square(\mathcal{R})^\wedge\to\square(\mathcal{R}^+_0)^\wedge$ preserves homotopy equivalence.
\end{corol}

\subsection{Model structure from the realization}
Next, we see that a cylinder functor actually induces a homotopy theory.
To do this, we need some preliminary results.
Most of them appear in the appendices of \cite{Lur}.

First of all, we need the notion of accessibilities:

\begin{defin}
Let $\mathcal{C}$ be a locally presentable category, and let $\kappa$ be an infinite regular cardinal.
Then a full subcategory $\mathcal{C}_0\subset\mathcal{C}$ is called a $\kappa$-accessible subcategory if it satisfies the following:
\begin{enumerate}[label={\rm(\arabic*)}]
  \item $\mathcal{C}_0$ is closed under $\kappa$-filtered colimits in $\mathcal{C}$.
  \item There is a (small) set $S\subset\objof\mathcal{C}_0$ of $\kappa$-small objects in $\mathcal{C}$ which generates $\mathcal{C}_0$ under $\kappa$-filtered colimits.
\end{enumerate}
We say $\mathcal{C}_0$ is an accessible subcategory if it is $\kappa$-accessible for some inifinite regular cardinal $\kappa$.
\end{defin}

Note that the accesibility implies retraction closed condition.
Indeed, suppose $\mathcal{C}_0\subset\mathcal{C}$ is a $\kappa$-accessible subcategory, and we have a retraction sequence $A\xrightarrow{i}X\xrightarrow{r}A$ in $\mathcal{C}$ with $X\in\mathcal{C}_0$.
Choose a $\kappa$-filtered cardinal $\lambda$, and define $A_\bullet:\lambda\to\mathcal{C}_0$ as follows:
For each $\alpha\le\lambda$, we set $A_\alpha=X$, and for each $\alpha\le\beta<\lambda$, we consider $ir:A_\alpha\to A_\beta$.
Then $A_\bullet:\lambda\to\mathcal{C}_0$ is a well-defined functor because of the equation $ri=\mathrm{id}$.
Now, it is easily verifies that
\[
A \simeq \colim_{\alpha<\lambda} A_\alpha\,.
\]
Since $\mathcal{C}_0$ is closed under $\kappa$-filtered colimits, we obtain $A\in\mathcal{C}_0$.

To verify the accessibility, the following result is most essential:

\begin{prop}[Corollary A.2.6.5 in\cite{Lur}]
Let $F:\mathcal{C}\to\mathcal{D}$ is a functor between locally presentable categories which preserves $\kappa$-filtered colimits for an inifinite regular cardinal $\kappa$.
Then for every $\kappa$-accessible subcategory $\mathcal{D}_0\subset\mathcal{D}$, the full subcategory $F^{-1}\mathcal{D}_0\subset\mathcal{C}$ spanned by objects $X\in\mathcal{C}$ with $F(X)\in\mathcal{D}_0$ is also a $\kappa$-accessible subcategory of $\mathcal{C}$.
\end{prop}

\begin{defin}
Let $\mathcal{C}$ be a locally presentable category.
Then a class $\mathcal{W}$ of morphisms of $\mathcal{C}$, which we will often identify with a full subcategory of $\mathcal{C}^{[1]}$, is said to be perfect if it satisfies the following:
\begin{enumerate}[label={\rm(\arabic*)}]
  \item $\mathcal{W}$ contains all isomorphisms.
  \item $\mathcal{W}$ satisfies the two-out-of-three axiom.
  \item $\mathcal{W}$ is an accessible subcategory of $\mathcal{C}^{[1]}$.
\end{enumerate}
\end{defin}

Notice that by the discussion above, a perfect class is closed under retracts.

\begin{corol}\label{cor:perfect-1}
If $F:\mathcal{C}\to\mathcal{D}$ is a functor between locally presentable category which preserves filtered colimit, then for every perfect class $\mathcal{W}_{\mathcal{D}}\subset\mathcal{D}^{[1]}$, the class $F^{-1}\mathcal{W}_{\mathcal{D}}\subset\mathcal{C}^{[1]}$ is also perfect.
\end{corol}

\begin{exam}
The class $\mathrm{isom}(\mathbf{Set})$ of all bijections in the category $\mathbf{Set}$ is perfect.
Indeed, $\mathrm{isom}(\mathbf{Set})$ is an $\omega$-accessible subcategory of $\mathbf{Set}^{[1]}$ since it is generated by $\{n=n\mid n\in\mathbb{N}\}\subset\mathrm{isom}(\mathbf{Set})$ under $\omega$-filtered colimits, where $\omega$ is the smallest infinite cardinal.
The other conditions are obvious.

Note that this argument works for every locally presentable category.
\end{exam}

\begin{exam}
Define a functor $\Pi_\bullet:\mathbf{SSet}\to\mathbf{Set}$ by
\[
\Pi_\bullet(X) := \coprod_{n=0}^\infty \mathbf{Top}(S^n,|X|)/\text{homotopy}\,,
\]
where $\mathbf{Top}$ is the category of compactly generated spaces and $|\blankdot|:\mathbf{SSet}\to\mathbf{Top}$ is the geometric realization.
Since $S^n$ and $S^n\times I$ are compact, and every simplicial map $X\to Y$ induces a relative cell complex $|X|\to |Y|$, the functor $\Pi_\bullet$ preserves filtered colimits.
Hence the class $\Pi_\bullet^{-1}(\mathrm{isom}(\mathbf{Set}))\subset\mathbf{SSet}^{[1]}$ is perfect.
Of course, $\Pi_\bullet^{-1}(\mathrm{isom}(\mathbf{Set}))$ is the class of weak equivalences in the Quillen model structure.
\end{exam}

\begin{theo}[Proposition A.2.6.13 in \cite{Lur}]
Let $\mathcal{C}$ be a locally presentable category.
Suppose we have a perfect class $\mathcal{W}$ and a set $C_0$ of morphisms of $\mathcal{C}$ such that
\begin{enumerate}[label={\rm(\alph*)}]
  \item $\mathcal{W}$ is closed under pushouts by pushouts of morphisms in $C_0$;
  \item Every morphism $f$ of $\mathcal{C}$ with $C_0\pitchfork f$ belongs to $\mathcal{W}$.
\end{enumerate}
Then there is a left proper cofibrantly generated model structure on $\mathcal{C}$ such that
\begin{itemize}
  \item A morphism $f$ is a weak equivalence if and only if $f\in\mathcal{W}$.
  \item A morphism $f$ is a cofibration if and only if $f\in\mathrm{Sat}(C_0)$, where $\mathrm{Sta}(C_0)$ is the saturated class generated by $C_0$.
  \item A morphism $f$ is a fibration if and only if $(\mathrm{Sat}(C_0)\cap\mathcal{W})\pitchfork f$.
\end{itemize}
\end{theo}

Now we obtain the model structure:

\begin{theo}\label{theo:cube-model}
Let $\mathcal{R}$ be an EZ cubicalizable category equipped with an enlargement $c:\mathcal{R}\to\mathcal{R}$.
Let $\square(\mathcal{R})$ denote the cubicalization, and let $|\blankdot|:\square(\mathcal{R})^\wedge\to\mathbf{SSet}$ be the geometric realization.
Then there is a left proper cofibrantly generated model structure on the category $\square(\mathcal{R})^\wedge$ such that
\begin{itemize}
  \item the class of weak equivalences is that of morphisms $f$ such that $|f|$ is a weak equivalence in $\mathbf{SSet}$;
  \item the class of cofibrations is that of monomorphisms;
  \item the class of fibrations is the right orthogonal class of trivial cofibrations.
\end{itemize}
Moreover, the functors $|\blankdot|:\square(\mathcal{R})^\wedge\to\mathbf{SSet}$ is a left Quillen functor.
\end{theo}
\begin{proof}
Let $\mathcal{W}$ be the class of morphisms $f:X\to Y\in\square(\mathcal{R})^\wedge$ such that $|f|$ is a weak equivalence in $\mathbf{SSet}$.
By corollary \ref{cor:perfect-1}, $\mathcal{W}$ is perfect.
We set
\[
C_0:=\{H\backslash\partial\square[r]\hookrightarrow H\backslash\square[r]\mid r\in\square(\mathcal{R}), H<\mathrm{Aut}_{\square(\mathcal{R})}(r)\}\,.
\]
Then the class of monomorphism is generated by $C_0$ as a saturated class.
Moreover, by corollary \ref{cor:real-mono}, $|\blankdot|:\square(\mathcal{R})^\wedge\to\mathbf{SSet}$ preserves monomorphism and the (Quillen) model structure on $\mathbf{SSet}$ is left proper, so that the class $\mathcal{W}$ is closed under pushouts by monomorphisms.
Hence by theorem \ref{theo:cube-model}, it suffices to show that if $f:X\to Y\in\square(\mathcal{R})^\wedge$ is right orghogonal to all monomorphisms, it belongs to $\mathcal{W}$.

Suppose $f:X\to Y$ is a morphism such that $C_0\pitchfork f$.
We define $Q_f\in\square(\mathcal{R})^\wedge$ by the following pushout:
\[
\xymatrix{
  X\amalg X \ar[r]^{\mathrm{id}\amalg f} \ar[d]_{(\iota^0,\iota^1)} \ar@{}[dr]|(.6){\pocorner} & X\amalg Y \ar[d] \\
  \operatorname{Cyl} X \ar[r] & Q_f }
\]
The morphisms
\[
\begin{gathered}
\operatorname{Cyl}X \xrightarrow{\iota^\dagger} X \xrightarrow{f} Y \\
X\amalg Y \xrightarrow{(f,\mathrm{id})} Y
\end{gathered}
\]
induce a morphism $\rho:Q_f\to Y$.
Notice that we have $|\operatorname{Cyl}X|\simeq |X|\times\Delta[1]$, so that the simplicial realization of the pushout square is the following:
\[
\xymatrix{
  |X|\amalg |X| \ar[r]^{\mathrm{id}\amalg |f|} \ar[d]_{(0,1)} \ar@{}[dr]|(.6){\pocorner} & |X|\amalg |Y| \ar[d] \\
  |X|\times\Delta[1] \ar[r] & |Q_f| }
\]
Thus $|Q_f|$ is nothing but the mapping cylinder of $|f|$ and $|\rho|:|Q_f|\to |Y|$ is the retraction, which is a weak equivalence.
It follows that $\rho\in\mathcal{W}$.

Now, write $i:X\xrightarrow{\text{inj}}X\amalg Y\to Q_f$.
Then we obtain a factorization $f=\rho i$.
Since $i$ is a monomorphism, there is a morphism $r:Q_f\to X$ such that the following diagram is commutative:
\[
\xymatrix{
  X \ar@{=}[r] \ar[d]_{i} & X \ar[d]^{f} \\
  Q_f \ar[r]^{\rho} \ar[ur]^{r} & Y }
\]
In other words, we have the following retraction diagram:
\[
\xymatrix{
  X \ar[r]^{i} \ar[d]_{f} & Q_f \ar[r]^{r} \ar[d]_{\rho} & X \ar[d]^{f} \\
  Y \ar@{=}[r] & Y \ar@{=}[r] & Y }
\]
Since $\mathcal{W}$ is closed under retract, we obtain $f\in\mathcal{W}$ as required.
\end{proof}

\subsection{Monoidal structures and thin-powered structure}
\label{subsec:monoidal-cube}
We next consider the monoidal structure on cubicalizations.
Recall that most of important examples of cubical categories admit monoidal structures.
We first see that some of the monoidal structures on cubicalizations $\square(\mathcal{R})$ come from monoidal structures on the original categories $\mathcal{R}$.

\begin{defin}
A cubicalizable category $\mathcal{R}$ is said to be monoidal if it is equipped with a strict monoidal structure $\otimes:\mathcal{R}\times\mathcal{R}\to\mathcal{R}$ such that
\begin{enumerate}[label={\rm(\arabic*)}]
  \item $\mathcal{R}$ has an initial object, which is the unit of the monoidal structure;
  \item if $\delta_1,\delta_2\in\mathcal{R}^+_0$ are morphisms, we have $\delta_1\otimes\delta_2\in\mathcal{R}^+_0$;
  \item the poset representation $\mathcal{R}\to\mathbf{Poset}$ is monoidal; indeed, there is a natural isomorphism $(\mathcal{R}^+_0\!/r_1)\times(\mathcal{R}^+_0\!/r_2)\simeq\mathcal{R}^+_0\!/(r_1\otimes r_2)$ which satisfies the coherent conditions.
\end{enumerate}
\end{defin}

\begin{exam}
Recall that $\widetilde\Delta$ admits a monoidal structure so called ``sum'' of ordinals.
Clearly the monoidal structure restricts to $\widetilde\Delta^+$ and satisfies the conditions above.
Hence $\widetilde\Delta$ is a monoidal cubicalizable category.
If $G$ is a group operad, $\widetilde\Delta G$ also admits a similar monoidal structure, and it is monoidal cubicalizable.

Notice that the thin-powered structure $\widetilde\Delta^+$ is not only a cubicalizable category, but also monoidal with the restricted monoidal structure.
More generally, if a cubicalizable category $\mathcal{R}$ is monoidal, then the monoidal structure restricts to $\mathcal{R}^+_0$ and $\mathcal{R}^+_0$ is also monoidal.
\end{exam}

As expected, if $\otimes$ is a monoidal structure as above, then it preserves $\mathcal{R}^-$ as well as $\mathcal{R}^+_0$.
Suppose $\sigma_1:r_1\to r'_1$ and $\sigma_2:r_2\to r'_2$ are $\mathcal{R}^+_0$-surjections.
Since the poset representation is monoidal, we have $(\sigma_1\otimes\sigma_2)_\ast\simeq (\sigma_1)_\ast\otimes(\sigma_2)_\ast$.
Notice that a morphism $f:r\to r'\in\mathcal{R}$ is $\mathcal{R}^+_0$-surjective if and only if the image of $f_\ast:\mathcal{R}^+_0\!/r\to\mathcal{R}^+_0\!/r'$ contains the greatest element $1:r'\to r'$.
Now we have $1\simeq 1\otimes 1 = (\sigma_1)_\ast(1)\otimes(\sigma_2)_\ast(1)$, so that $\sigma_1\otimes\sigma_2$ is $\mathcal{R}^+_0$-surjective.

Note also that if we write $m:\mathcal{R}^+_0\!/r_1\otimes\mathcal{R}^+_0\!/r_2\simeq\mathcal{R}^+_0\!/(r_1\otimes r_2)$ to be the monoidal isomorphism, for $\delta_1\in\mathcal{R}^+_0\!/r_1$ and $\delta_2\in\mathcal{R}^+_0\!/r_2$, we have
\[
m(\delta_1,\delta_2)
= m((\delta_1)_\ast1,(\delta_2)_\ast 1)
= (\delta_1\otimes\delta_2)_\ast m(1,1)
= (\delta_1\otimes\delta_2)_\ast 1
= \delta_1\otimes\delta_2\,,
\]
where $1$ denotes the greatest element.
Thus joins and meets are componentwisely computed; i.e. $(\delta_1\otimes\delta_2)\vee(\delta'_1\otimes\delta'_2)=(\delta_1\vee\delta'_1)\otimes(\delta_2\vee\delta'_2)$ and so on.
In particular, we have $\neg(\delta_1\otimes\delta_2)=(\neg\delta_1)\otimes(\neg\delta_2)$.

Moreover the isomorphism $(\mathcal{R}^+_0\!/r_1)\times(\mathcal{R}^+_0\!/r_2)\simeq\mathcal{R}^+_0\!/(r_1\otimes r_2)$ implies that every $\delta\in\mathcal{R}^+_0\!/(r_1\otimes r_2)$ is uniquely written as $\delta=\delta_1\otimes\delta_2$ for $\delta_1\in\mathcal{R}^+_0\!/r_1$ and $\delta_2\in\mathcal{R}^+_0\!/r_2$.
More precisely, denote by $0$ the initial object of $\mathcal{R}$.
We also denote by $0:0\to r$ the unique morphism in $\mathcal{R}$ for $r\in\mathcal{R}$.
Then for each $r_1,r_2\in\mathcal{R}$, we can defin the ``inclusions'' $j_i:r_i\to r_1\otimes r_2$ as
\[
\begin{gathered}
j_1:r_1 \simeq r_1\otimes 0 \xrightarrow{\mathrm{id}\otimes 0} r_1\otimes r_2\,,\\
j_2:r_2 \simeq 0\otimes r_2 \xrightarrow{0\otimes \mathrm{id}} r_1\otimes r_2\,.
\end{gathered}
\]
Since monoidal structure preserves $\mathcal{R}^+_0$, $j_1,j_2\in\mathcal{R}^+_0\!/(r_1\otimes r_2)$.
Moreover, by the natural isomorphism $(\mathcal{R}^+_0\!/r_1)\times(\mathcal{R}^+_0\!/r_2)\simeq\mathcal{R}^+_0\!/(r_1\otimes r_2)$, we have $\neg j_1=j_2$ and $\neg j_2=j_1$.
Therefore, for $\delta\in\mathcal{R}^+_0\!/(r_1\otimes r_2)$, we have
\[
\delta
= (\delta\wedge j_1) \vee (\delta\wedge j_2)
= (j_1)_\ast(j_1^\ast\delta)\vee (j_2)_\ast(j_2^\ast\delta)
= ((j_1^\ast\delta)\otimes 0) \vee (0\otimes(j_2^\ast\delta))
= (j_1^\ast\delta)\otimes(j_2^\ast\delta)\,.
\]

The monoidal structure is also functorial for pullbacks:

\begin{lemma}\label{lem:monoidal-pb}
Let $\mathcal{R}$ be a monoidal cubicalizable category.
Then for morphisms $f_i:r'_i\to r_i$ and $\delta_i\in\mathcal{R}^+_0\!/r_i$ for $i=1,2$, we have the formula:
\[
(f_1\otimes f_2)^\ast(\delta_1\otimes\delta_2) = (f_1^\ast\delta_1)\otimes (f_2^\ast\delta_2)\,.
\]
\end{lemma}
\begin{proof}
Since the functor $\mathcal{R}^+_0\!/(\blankdot)$ is monoidal, for $\delta'_i\in\mathcal{R}^+_0\!/r_i$ for $i=1,2$, the following equivalences hold:
\[
\begin{split}
&(f_1)_\ast(\delta'_1)\otimes (f_2)_\ast(\delta'_2) \le \delta_1\otimes\delta_2 \\
&\iff (f_1)_\ast(\delta'_1)\le \delta_1 \text{ and } (f_2)_\ast(\delta'_2)\le\delta_2 \\
&\iff \delta'_1\le f_1^\ast\delta_1 \text{ and } \delta'_2\le f_2^\ast\delta_2 \\
&\iff \delta'_1\otimes\delta'_2 \le (f_1^\ast\delta_1)\otimes(f_2^\ast\delta_2)\,.
\end{split}
\]
This implies that we have the following Galois connection:
\[
(f_1)_\ast\otimes(f_2)_\ast:\mathcal{R}^+_0\!/(r'_1\otimes r'_2) \xleftrightarrow{} \mathcal{R}^+_0\!/(r_1\otimes r_2):f_1^\ast\otimes f_2^\ast\,.
\]
Now the equation $f_1^\ast\otimes f_2^\ast=(f_1\otimes f_2)^\ast$ follows from the equation $(f_1)_\ast\otimes(f_2)_\ast=(f_1\otimes f_2)_\ast$ and the uniqueness of right adjoinction.
Hence the result follows.
\end{proof}

As a consequence, we have a monoidal structure on the cubicalization:

\begin{corol}\label{cor:cube-monoid-ext}
If $\mathcal{R}$ is a monoidal cubicalizable category, then the monoidal structure extends to the cubicalization $\square(\mathcal{R})$ by setting
\[
(\delta_1^{\xi_1}\sigma_1\gamma_1^\dagger)\otimes(\delta_2^{\xi_2}\sigma_2\gamma_2^\dagger)
= (\delta_1\otimes\delta_2)^{\xi_1\otimes\xi_2}(\sigma_1\otimes\sigma_2)(\gamma_1\otimes\gamma_2)^\dagger\,.
\]
The unit object is the terminal object.
\end{corol}

Notice that if $0$ is the initial object of $\mathcal{R}$, then it is the terminal object in the cubicalization $\square(\mathcal{R})$; for the unique $0:0\to r\in\mathcal{R}$, $0^\dagger:r\to 0\in\square(\mathcal{R})$ is the unique morphism onto the terminal object.

\begin{exam}
Consider the cubicalizable categories $\widetilde\Delta$, $\widetilde\Delta^+$ and $\widetilde\Delta G$ for a group operad $G$.
They are monoidal, and the induced monoidal structures on cubicalizations coincide with the canonical monoidal structures on them; defined by $\square[m]\otimes\square[n]:=\square[m+n]$.
\end{exam}

In general, it is known that a monoidal structure on a presite induces a monoidal structure on the category of presheaves on it.
Precisely, if $(\mathcal{A},\otimes)$ is a small monoidal category, then for $X,Y\in\mathcal{A}^\wedge$, we set
\[
X\otimes Y
:= \int^{p,q\in\mathcal{A}} \mathcal{A}[p\otimes q]\times X(p)\times Y(q)\,.
\]
We call it the convolution product.
By the co-Yoneda lemma, it is in fact a monoidal structure which is an extension of that on $\mathcal{A}$.
The unit object on $\mathcal{A}^\wedge$ is represented by that of $\mathcal{A}$.

Moreover, the convolution product is a part of adjuncions of two variables.
Indeed, we have
\[
\begin{split}
\mathcal{A}^\wedge(X\otimes Y,Z)
&\simeq \int_{a\in\mathcal{A}}\int_{p,q\in\mathcal{A}} \mathbf{Set}(\mathcal{A}(a,p\otimes q)\times X(p)\times Y(q),Z(a)) \\
&\simeq \int_{p,q\in\mathcal{A}} \mathbf{Set}(X(p)\times Y(q),Z(p\otimes q))\,.
\end{split}
\]
Hence if we set
\[
\begin{split}
\mathrm{Map}_{\mathrm{l}}(X,Z) &:= \int_{p\in\mathcal{A}} Z(p\otimes(\blankdot))^{X(p)}\,, \\
\mathrm{Map}_{\mathrm{r}}(Y,Z) &:= \int_{q\in\mathcal{A}} Z((\blankdot)\otimes q)^{Y(q)}\,,
\end{split}
\]
then we obtain the adjunctions
\[
\mathcal{A}^\wedge(Y,\mathrm{Map}_{\mathrm{l}}(X,Z))
\simeq \mathcal{A}^\wedge(X\otimes Y,Z)
\simeq \mathcal{A}^\wedge(X, \mathrm{Map}_{\mathrm{r}}(Y,Z))\,.
\]
Therefore, $\mathcal{A}^\wedge$ is a closed monoidal category.
To summarise, we obtain the following:

\begin{prop}
If $\mathcal{R}$ is a monoidal cubicalizable category, then the category $\square(\mathcal{R})^\wedge$ admits a closed monoidal structure
\[
X\otimes Y
:= \int^{p,q\in\square(\mathcal{R})} \square[p\otimes q]\times X(p)\times Y(q)
\]
with terminal unit.
Moreover, the simplicial realization functor $\square(\mathcal{R})^\wedge\to\mathbf{SSet}$ is (strongly) monoidal.
\end{prop}
\begin{proof}
The first statement follows from the discussion above.
So we show the last one.
Notice that the simplicial realization $|\blankdot|:\square(\mathcal{R})^\wedge\to\mathbf{SSet}$ is a left adjoint functor, so preserves colimits.
Notice also that the composition
\[
\square(\mathcal{R})
\xrightarrow{\mathcal{R}^+_0\!/(\blankdot)} \mathbf{Poset}
\xrightarrow{N} \mathbf{SSet}
\]
is monoidal.
Hence the required result is verified by the direct calculus as follows:
\[
\begin{split}
|X\otimes Y|
&\simeq \left| \int^{p,q\in\square(\mathcal{R})} \square[p\otimes q]\times X(p)\times Y(q) \right| \\
&\simeq \int^{p,q\in\square(\mathcal{R})} |\square[p\otimes q]| \times X(p)\times Y(q) \\
&\simeq \int^{p,q\in\square(\mathcal{R})} |\square[p]|\times |\square[q]| \times X(p) \times Y(q) \\
&\simeq \left(\int^{p\in\square(\mathcal{R})} |\square[p]|\times X(p)\right)\times \left(\int^{q\in\square(\mathcal{R})} |\square[q]|\times Y(q)\right) \\
&\simeq |X|\times |Y|\,.
\end{split}
\]
\end{proof}

Next, we observe that a monoidal cubicalizable category admits enlargements.
Let $\mathcal{R}$ be a monoidal cubicalizable category, and fix an object $t\in\mathcal{R}$ of degree $2$.
We define a functor $c_t:\mathcal{R}\to\mathcal{R}$ by $c_t(r):=r\otimes t$, and a natural transformation $\iota:\mathrm{Id}\to c_t$ by $\iota:r \simeq r\otimes0 \to r\otimes t$.

\begin{lemma}
In the situation above, $c_t$ and $\iota$ form an enlargement on $\mathcal{R}$.
\end{lemma}
\begin{proof}
We have $\operatorname{codeg}\iota = \deg t = 2$.
The other conditions follow from the discussion on the beginning of this subsection and lemma \ref{lem:monoidal-pb}.
\end{proof}

Hence by lemma \ref{lem:cube-cyl-c} and lemma \ref{lem:cube-cyl}, we obtain an extension $c_t:\square(\mathcal{R})\to\square(\mathcal{R})$ and the diagram
\[
\mathrm{Id}\xrightrightarrows[\iota^{\neg\iota}]{\iota} c_t \xrightarrow{\iota^\dagger} \mathrm{Id}
\]
on $\square(\mathcal{R})$.
However, by the formula of the extionsion, we have
\[
c_t(\delta^\xi\sigma\gamma^\dagger)
= c_t(\delta)^{\iota_\ast\xi} c_t(\sigma) c_t(\gamma)^\dagger
= (\delta\otimes t)^{\xi\otimes t} (\sigma\otimes t) (\gamma\otimes t)^\dagger
= (\delta^\xi\sigma\gamma^\dagger)\otimes t\,.
\]
In other words, the extension of $c_t$ is nothing but the extension of the functor $(\blankdot)\otimes t$ in corollary \ref{cor:cube-monoid-ext}.
The diagram is isomorphic to the following:
\[
\mathrm{Id}\xrightrightarrows[\mathrm{id}\otimes 0^t]{\mathrm{id}\otimes 0} (\blankdot)\otimes t \xrightarrow{\mathrm{id}\otimes 0^\dagger} \mathrm{Id}\,.
\]
Moreover, the functor $\operatorname{Cyl}:\square(\mathcal{R})^\wedge\to\square(\mathcal{R})^\wedge$ can be comupted as
\[
\operatorname{Cyl} X \simeq X\otimes\square[t]\,.
\]

Finally, we prove the following result:

\begin{prop}\label{prop:cube-model-monoidal}
Let $\mathcal{R}$ be a monoidal EZ cubicalizable category.
Then the model structure on $\square(\mathcal{R})^\wedge$ defined in thorem \ref{theo:cube-model} is compatible with the convolution product.
Consequently, $\square(\mathcal{R})^\wedge$ is a monoidal model category.
\end{prop}

Before the proof, we need some preliminary results.

\begin{lemma}\label{lem:bound-tensor}
Let $\mathcal{R}$ be a monoidal EZ cubicalizable category.
Then for each $r_1,r_2\in\mathcal{R}$, the inclusion $\partial\square[r_1\otimes r_2]\hookrightarrow\square[r_1\otimes r_2]$ is naturally isomorphic to the morphism
\[
\partial\square[r_1]\otimes\square[r_2] \mathop\amalg_{\partial\square[r_1]\otimes\partial\square[r_2]} \square[r_1]\otimes\partial\square[r_2]
\to \square[r_1]\otimes\square[r_2]\,.
\]
\end{lemma}
\begin{proof}
By the definition of $\partial\square[r_1\otimes r_2]$, it is given as the coequalizer of the sequence
\[
\resizebox{\hsize}{!}{$\displaystyle%
\raisedunderop\coprod{\substack{(\delta':s'\to r_1\otimes r_2)\in\mathcal{R}^+_0\!/(r_1\otimes r_2)\\\operatorname{codeg}\delta'=4}} \square[s']\times(\mathcal{R}^+_0\!/(r_1\otimes r_2))_{\le\neg\delta'} \\
\rightrightarrows
\raisedunderop\coprod{\substack{(\delta:s\to r_1\otimes r_2)\in\mathcal{R}^+_0\!/(r_1\otimes r_2)\\\operatorname{codeg}\delta=2}} \square[s]\times(\mathcal{R}^+_0\!/(r_1\otimes r_2))_{\le\neg\delta}\,.
$}
\]
We have a natural isomorphism $\mathcal{R}^+_0\!/(r_1\otimes r_2)\simeq (\mathcal{R}^+_0\!/r_1)\times(\mathcal{R}^+_0\!/r_2)$.
Moreover, for $\gamma:t\to r_1\otimes r_2$, there uniquely exists $\gamma_1:t_1\to r_1$ and $\gamma_2:t_2\to r_2$ such that $\gamma=\gamma_1\otimes\gamma_2$.
In this case, we have $\operatorname{codeg}\gamma = (\operatorname{codeg}\gamma_1)(\operatorname{codeg}\gamma_2)$.
Thus $\partial\square[r_1\otimes r_2]$ is also given as the colimit of the solid part of the following diagram:
\[
\let\objectstyle=\displaystyle
\resizebox{\hsize}{!}{$\displaystyle%
\xymatrix{
  \left(\coprod_{\delta'_1:s'_1\to r_1} \square[s'_1]\times[1]^2\right)\otimes\left(\coprod_{\delta'_2:s'_2\to r_2} \square[s'_2]\times[1]^2\right) \ar@<.5ex>@{.>}[r] \ar@<-.5ex>@{.>}[r] \ar@<.5ex>@{.>}[d] \ar@<-.5ex>@{.>}[d] &
  \left(\coprod_{\delta_1:s_1\to r_1} \square[s_1]\times[1]\right)\otimes\left(\coprod_{\delta'_2:s'_2\to r_2} \square[s'_2]\times[1]^2\right) \ar@{.>}[r] \ar@<.5ex>@{.>}[d] \ar@<-.5ex>@{.>}[d] &
  \square[r_1]\otimes\left(\coprod_{\delta'_2:s'_2\to r_2}\square[s'_2]\times[1]^2\right) \ar@<.5ex>[d] \ar@<-.5ex>[d] \\
  \left(\coprod_{\delta'_1:s'_1\to r_1} \square[s'_1]\times[1]^2\right)\otimes\left(\coprod_{\delta_2:s_2\to r_2} \square[s_2]\times[1]\right) \ar@<.5ex>@{.>}[r] \ar@<-.5ex>@{.>}[r] \ar@{.>}[d] &
  \left(\coprod_{\delta_1:s_1\to r_1} \square[s_1]\times[1]\right)\otimes\left(\coprod_{\delta_2:s_2\to r_2} \square[s_2]\times[1]\right) \ar[r] \ar[d] &
  \square[r_1]\otimes\left(\coprod_{\delta_2:s_2\to r_2}\square[s_2]\times[1]\right) \\
  \left(\coprod_{\delta'_1:s'_1\to r_1} \square[s'_1]\times[1]^2\right)\otimes\square[r_2] \ar@<.5ex>[r] \ar@<-.5ex>[r] &
  \left(\coprod_{\delta_1:s_1\to r_1} \square[s_1]\times[1]\right)\otimes\square[r_2] &
  {} }
$}
\]
It is final in the whole diagram.
However, since $\otimes$ commutes with colimits, the colimit of the whole diagram is clearly isomorphic to the pushout of the square
\[
\xymatrix{
  \partial\square[r_1]\otimes\partial\square[r_2] \ar[r] \ar[d] & \square[r_1]\otimes\partial\square[r_2] \\
  \partial\square[r_1]\otimes\square[r_2] & {} }
\]
Finally, we obtain the required isomorphism.
\end{proof}

For $f_1:A_1\to B_1,f_2:A_2\to B_2\in\square(\mathcal{R})^\wedge$, we write
\[
f_1\odot f_2:B_1\otimes A_2\mathop{\amalg}_{A_1\otimes A_2} A_1\otimes B_2 \to B_1\otimes B_2\,.
\]
Notice that $\odot$ induces a closed monoidal structure on the category $(\square(\mathcal{R})^\wedge)^{[1]}$.
We denote by $\beta_r:\partial\square[r]\hookrightarrow\square[r]$ the natural inclusion.
Then lemma \ref{lem:bound-tensor} asserts that $\beta_{r_1}\odot\beta_{r_2}\simeq \beta_{r_1\otimes r_2}$.
Moreover, for $r\in\mathcal{R}$ and $H<\mathrm{Aut}_{\square(\mathcal{R})}(r)$, we write
\[
\beta^H_r:H\backslash\partial\square[r] \hookrightarrow H\backslash\square[r]\,.
\]

\begin{corol}\label{cor:bound-act-tensor}
In the above situation, we have an isomorphism
\[
\beta^{H_1}_{r_1}\odot \beta^{H_2}_{r_2}
\simeq \beta^{H_1\times H_2}_{r_1\otimes r_2}\,.
\]
\end{corol}
\begin{proof}
First notice that $H_i$ acts on $\beta_{r_i}$ in the category $(\square(\mathcal{R})^\wedge)^{[1]}$ for $i=1,2$.
Hence we obtain an isomorphism $\beta^{H_i}_{r_i}\simeq H_i\backslash\beta_{r_i}$.
Now we have an embedding $\mathrm{Aut}_{\square(\mathcal{R})}(r_1)\times\mathrm{Aut}_{\square(\mathcal{R})}(r_2)\to \mathrm{Aut}_{\square(\mathcal{R})}(r_1\otimes r_2)$, so that we can identify $H_1\times H_2$ with a subgroup of $\mathrm{Aut}_{\square(\mathcal{R})}(r_1\otimes r_2)$.
In the case, $H_1\times H_2$ acts on $\beta_1\odot\beta_2$ componentwisely.
Hence, by lemma \ref{lem:bound-tensor}, we have isomorphisms
\[
\begin{split}
\beta^{H_1}_{r_1}\odot \beta^{H_2}_{r_2}
&\simeq (H_1\backslash\beta_{r_1})\odot(H_2\backslash\beta_{r_2}) \\
&\simeq (H_1\times H_2)\backslash(\beta_{r_1}\odot\beta_{r_2}) \\
&\simeq (H_1\times H_2)\backslash\beta_{r_1\otimes r_2} \\
&\simeq \beta_{r_1\otimes r_2}^{H_1\times H_2}\,.
\end{split}
\]
Composing them, and we obtain the required isomorphism.
\end{proof}

\begin{proof}[Proof of proposition \ref{prop:cube-model-monoidal}]
Recall that a monoidal structure on $\square(\mathcal{R})$ is said to be compatible with the model structure if for every cofibrations $f_1,f_2$, the morphism $f_1\odot f_2$ is also a cofibration which is moreover trivial if either $f_1$ or $f_2$ is so.

We first show that if $f_1,f_2$ are cofibrations, or monomorphisms, then so is $f_1\odot f_2$.
It follows from corollary \ref{cor:ez-cellular} and corollary \ref{cor:bound-act-tensor}.
Indeed, by corollary \ref{cor:ez-cellular}, if we set
\[
S_0:=\{\beta^H_r:H\backslash\partial\square[r]\hookrightarrow H\backslash\square[r]\mid r\in\square(\mathcal{R}),\,H<\mathrm{Aut}_{\square(\mathcal{R})}(r)\}\,,
\]
then the class of monomorphisms in $\square(\mathcal{R})^\wedge$ coincides with the saturated class $\mathrm{Sat}(S_0)$ generated by $S_0$.
Since the category $\square(\mathcal{R})^\wedge$ is locally presentable, the small object argument implies that $\mathrm{Sat}(S_0)={}^\pitchfork(S_0^\pitchfork)$.
Now $\odot$ is a closed monoidal structure on $(\square(\mathcal{R}))^\wedge$, for a class $\mathcal{M}$ of morphisms, we have
\[
\begin{split}
&(S_0\odot S_0)\pitchfork\mathcal{M} \\
&\iff S_0 \pitchfork \mathrm{Map}^\odot_{\mathrm{r}}(S_0,\mathcal{M}) \\
&\iff \mathrm{Sat}(S_0) \pitchfork \mathrm{Map}^\odot_{\mathrm{r}}(S_0,\mathcal{M}) \\
&\iff (\mathrm{Sat}(S_0)\odot S_0) \pitchfork \mathcal{M} \\
&\iff S_0 \pitchfork \mathrm{Map}^\odot_{\mathrm{l}}(S_0,\mathcal{M}) \\
&\iff \mathrm{Sat}(S_0) \pitchfork \mathrm{Map}^\odot_{\mathrm{l}}(S_0,\mathcal{M}) \\
&\iff (\mathrm{Sat}(S_0)\odot\mathrm{Sat}(S_0))\pitchfork\mathcal{M}\,.
\end{split}
\]
One can verify $\mathrm{Sat}(S_0)\odot\mathrm{Sat}(S_0)\subset\mathrm{Sat}(S_0)$ by setting $\mathcal{M}$ to be the class of trivial fibrations, and using corollary \ref{cor:bound-act-tensor}.

Finally, suppose $f_1:A_1\to B_1$ and $f_2:A_2\to B_2$ are monomorphisms such that either $f_1$ or $f_2$ is trivial.
Since the simplicial realization is monoidal and preserves colimits, we have
\[
|f_1\odot f_2|\simeq |f_1|\odot |f_2|:(|B_1|\times|A_2|)\mathop\amalg_{|A_1|\times |A_2|} (|A_1|\times |B_2|) \to |B_1|\times|B_2|\,.
\]
By corollary \ref{cor:real-mono}, $|f_1|$ and $|f_2|$ are cofibrations and either $|f_1|$ or $|f_2|$ is trivial in $\mathbf{SSet}$.
Then since the cartesian product on $\mathbf{SSet}$ is compatible with the Quillen model structure, $|f_1|\odot|f_2|$ is a trivial cofibration.
Thus we conclude that $f_1\odot f_2$ is a trivial cofibration in $\square(\mathcal{R})^\wedge$ in the model structure defined in theorem \ref{theo:cube-model}.
\end{proof}

\subsection{Regularity}
The notion of regularity is really important when one consider model structures on a presheaf category.
We here recall the definition:

\begin{defin}
Let $\mathcal{A}$ be a small category.
Then a model structure on $\mathcal{A}^\wedge$ is said to be regular if for each object $X\in\mathcal{A}^\wedge$, the natural morphism
\[
\hocolim_{(\mathcal{A}[a]\to X)\in\mathcal{A}/X} \mathcal{A}[a] \,\to X
\]
is a weak equivalence.
\end{defin}

Recall that, in the case that $\mathcal{C}$ is a cofibrantly generated model category, a homotopy colimit on $\mathcal{C}$ over a small category $\mathcal{I}$ is defined to be the derived functor of the left Quillen functor $\colim:\mathcal{C}^{\mathcal{I}}_{\text{proj}}\to\mathcal{C}$, where $\mathcal{C}^{\mathcal{I}}_{\text{proj}}$ is the functor category together with the projective model structure.
However, the homotopy colimit of a small diagram $D:\mathcal{I}\to\mathcal{C}$ is also computed as follows:
First, choose a model structure on $\mathcal{C}^I$ such that the adjunction
\[
\colim:\mathcal{C}^I\xrightleftarrows{} \mathcal{C}:\mathrm{const}
\]
forms a Quillen adjunction; e.g. consider the projective model structure.
Choose also a cofibrant replacement $\widetilde D\to D$ of $D$ in the model structure.
Then we set
\[
\hocolim_{i\in I}D(i) := \colim_{i\in I} \widetilde D(i)\,.
\]
It is a consequence of Ken Brown's lemma and the two-out-of-six property that the homotopy types of the homotopy colimit is independent of the choice of the model structure and the fibrant replacement.
Notice that although homotopy colimits are classically defined in a functorial way (e.g. \cite{BK} or \cite{Hir}), the homotopy colimits computed in the way above may not be functorial, that is, the above computation of homotopy colimits may not define any functor $\hocolim:\mathbf{Cat}/\mathcal{C}\to\mathcal{C}$.
Indeed, the homotopy colimits above are actually defined in the homotopy category.
For the equivalence to the classical one, see \cite{Hir} or \cite{Jar}.

For example, let $\mathcal{C}$ be a cofibrantly generated model category and $C$ be a small category.
Regard an object $K\in\mathcal{C}$ as a constant functor $K:C\to\mathcal{C}$.
Then we define an object $\mathcal{N}_{\mathsf{H}}(C;K)\in\mathcal{C}$ by the following homotopy colimit:
\[
\mathcal{N}_{\mathsf{H}}(C;K) := \hocolim_{a\in C} K\,.
\]
We call it the internal nerve of $C$ in $\mathcal{C}$ at $K$.
For a terminal object $\ast\in\mathcal{C}$, we write $\mathcal{N}_{\mathsf{H}}(C):=\mathcal{N}_{\mathsf{H}}(C;\ast)$.

The internal nerve catches a homotopical information of a category in some sense.
In fact, we have

\begin{lemma}\label{lem:intnerve-contract}
Let $\mathcal{C}$ be a cofibrantly generated model category and $K\in\mathcal{C}$.
If $D$ is a small category with a terminal object, then the natural morphism
\[
\mathcal{N}_{\mathsf{H}}(D;K) \to K
\]
is a weak equivalence in $\mathcal{C}$.
In particular, the unique morphism $\mathcal{N}_{\mathsf{H}}(D)\to\ast$ is a weak equivalence.
\end{lemma}
\begin{proof}
Let $\widetilde K\to K$ be a cofibrant replacement in $\mathcal{C}^D$ in the projective model structure.
If $t\in D$ is a terminal object, then we have
\[
\hocolim_{D} K \simeq \colim_{d\in D} \widetilde K(d) \simeq \widetilde K(t)\,.
\]
Since the natural transformation $\widetilde K\to K$ is a pointwise weak equivalence, we obtain a sequence of weak equivalences:
\[
\mathcal{N}_{\mathsf{H}}(D;K)
\simeq \hocolim_D K
\simeq \widetilde K(t)
\simeq K(t) = K\,.
\]
\end{proof}

Notice that since the homotopy colimit is only defined up to homotopy, $\mathcal{N}_{\mathsf{H}}(C;K)$ is not functorial with respect $C\in\mathbf{Cat}$ in general.
However, if we have a functor $\Phi:\mathcal{A}\to\mathbf{Cat}$ together with a pointwise fibration $\varphi:\Phi\to C$, see \cite{Bor} for the definition, to a constant functor at $C\in\mathbf{Cat}$, then we can assume $\mathcal{N}_{\mathsf{H}}(\Phi(\blankdot);K):\mathcal{A}\to\mathcal{C}$ is a functor; indeed, since $\varphi_a:\Phi(a)\to C$ is a fibration, the adjoint pair
\[
\varphi_a^\ast:\mathcal{C}^{C} \rightleftarrows \mathcal{C}^{\Phi(a)}:(\varphi_a)_\ast
\]
arising from the right Kan extension is a Quillen pair.
It follows that if one choose a cofibrant replacement $\widetilde K\to K$ in $\mathcal{C}^C$, for each $a\in\mathcal{A}$, $\varphi_a^\ast\widetilde K\to K$ is also a cofibrant replacement in $\mathcal{C}^{\Phi(a)}$.
Hence we have
\[
\hocolim_{\Phi(a)} K \simeq \colim_{c\in\Phi(a)} \widetilde K(\varphi_a(c))
\]
which is clearly functorial with respect to $a\in\mathcal{A}$.
In particular, thanks to the forgetful functor $\mathcal{A}/X\to\mathcal{A}$, we have a functor
\[
\mathcal{N}_{\mathsf{H}}(\mathcal{A}/(\blankdot);K):\mathcal{A}^\wedge\to\mathcal{C}\,.
\]
It is a special case of Grothendieck construction.

Let go back to the discussion about regularity.
Generally, for a fixed model structure on $\mathcal{A}$, we will say $X\in\mathcal{A}^\wedge$ is regular if the natural morphism
\[
\hocolim_{(\mathcal{A}[a]\xrightarrow{x}X)\in\mathcal{A}/X}\mathcal{A}[a] \,\to X
\]
is a weak equivalence.
Then some calculations of homotopy colimits gives rise to good criteria:

\begin{prop}\label{prop:regular-criteria}
Let $\mathcal{A}$ be a small category.
Suppose $\mathcal{A}^\wedge$ admits a model structure such that
\begin{enumerate}[label={\rm(\alph*)}]
  \item every object in $\mathcal{A}^\wedge$ is cofibrant;
  \item for each $a\in\mathcal{A}$, the morphism $\mathcal{A}[a]\to\ast$ is a weak equivalence.
\end{enumerate}
Then the following hold:
\begin{enumerate}[label={\rm(\arabic*)}]
  \item An arbitrary coproduct of regular objects is regular.
  \item If $X_0\leftarrow X_1\xrightarrow{i} X_2$ is a diagram of regular objects in $\mathcal{A}^\wedge$ with $i$ a cofibration, then $X_0\amalg_{X_1}X_2$ is regular.
  \item If $X_0\to X_1\to\dots$ is a $\lambda$-sequence of cofibrations for an ordinal $\lambda$ such that each term $X_\alpha$ is regular, then its colimit $\colim_{\alpha<\lambda} X_\alpha$ is also regular.
\end{enumerate}
\end{prop}

For the proof, see for example \cite{Jar}.

We now ready to prove the result:

\begin{theo}
Let $\mathcal{R}$ be an EZ cubicalizable category with an enlargement $c:\mathcal{R}\to\mathcal{R}$.
Suppose $X\in\square(\mathcal{R})^\wedge$ is an object such that each natural morphism $\operatorname{sk}_{n-1}X\to\operatorname{sk}_n X$ is a pushouts of a coproduct of morphisms of the form $\partial\square[r]\hookrightarrow\square[r]$ for $r\in\square(\mathcal{R})$.
Then $X$ is regular with respect to the model structure given in theorem \ref{theo:cube-model}.
In particular, if $\mathcal{R}$ has no non-trivial isomorphism, the model structure is regular.
\end{theo}
\begin{proof}
Since $\mathcal{R}$ is an EZ cubicalizable, the category $\square(\mathcal{R})$ is an EZ category.
Then by corollary \ref{cor:ez-cellular}, every monomorphism is a transfinite composition of pushouts of morphism of the form
\[
H\backslash\partial\square[r] \hookrightarrow H\backslash\square[r]
\]
for $r\in\square(\mathcal{R})$ and $H<\mathrm{Aut}_{\square(\mathcal{R})}(r)$.
In particular, the natural morphism $\operatorname{sk}_{n-1}X\to\operatorname{sk}_nX$ is a pushout of a coproduct of such morphisms with $\deg r=n$.
Notice that every isomorphism in $\square(\mathcal{R})$ is an image of an isomorphism in $\mathcal{R}$ by the canonical embedding $\mathcal{R}\hookrightarrow\square(\mathcal{R})$, so if $\mathcal{R}$ has no non-trivial isomorphism, so does $\square(\mathcal{R})$.
Thus, the last statement follows from the first.

We show the first statement.
First, we show that every representable object is regular; i.e. the natural morphism
\[
\hocolim_{(s\xrightarrow{f}r)\in\square(\mathcal{R})/r} \square[s]\,\to\square[r]
\]
is a weak equivalence.
Notice that the simplicial realizations of representables are products of the interval $\Delta[1]$ in $\mathbf{SSet}$ so that the unique morphism $\square[s]\to\ast$ is a weak equivalence.
Since homotopy colimits preserve weak equivalence, the morphism above is a weak equivalence if and only if the natural morphism
\[
\hocolim_{(s\xrightarrow{f}r)\in\square(\mathcal{R})/r}\square[s] \simeq \hocolim_{\square(\mathcal{R})/r} \ast = \mathcal{N}_{\mathsf{H}}(\square(\mathcal{R})/r) \to \ast
\]
is a weak equivalence.
Since $\square(\mathcal{R})/r$ has a terminal object, this follows from lemma \ref{lem:intnerve-contract}.

Next, we show it for $n$-skeletal $X\in\square(\mathcal{R})^\wedge$ by induction on $n$.
Suppose the statement holds for $k$-skeletal $K\in\square(\mathcal{R})^\wedge$ for $k<n$, and suppose $X\in\square(\mathcal{R})^\wedge$ is an $n$-skeletal object such that there is a pushout square below:
\[
\renewcommand{\objectstyle}{\displaystyle}
\xymatrix{
  \raisedunderop\coprod{\substack{\square[r]\to X\cr\deg r=n}} \partial\square[r] \ar[r] \ar@{^(->}[d] \ar@{}[dr]|(.6){\pocorner}& \operatorname{sk}_{n-1} X \ar[d] \\
  \raisedunderop\coprod{\substack{\square[r]\to X\cr\deg r=n}} \square[r] \ar[r] & X }
\]
By the induction hypothesis and proposition \ref{prop:regular-criteria}, the top two objects are regular.
We have already shown that the bottom left corner is regular.
Since moreover the left edge is a cofibration, by proposition \ref{prop:regular-criteria}, we conclude that $X$ is a regular.

Finally, suppose $X\in\square(\mathcal{R})^\wedge$ is an object satisfying the assumption.
Notice that $X$ is a sequencial colimit $\operatorname{sk}_0X\hookrightarrow\operatorname{sk}_1X\hookrightarrow\dots$.
As shown above, each skeleton $\operatorname{sk}_nX$ is regular.
Since $\operatorname{sk}_{n-1}X\to\operatorname{sk}_nX$ is a cofibration, by proposition \ref{prop:regular-criteria}, it follows that $X$ is regular as required.
\end{proof}

\begin{corol}
A category $\square(\widetilde\Delta)^\wedge$ is regular with respect to the model structure given in theorem \ref{theo:cube-model}.
\end{corol}

Note that the category $\square(\widetilde\Delta)$ is known by the name of the cubical category with connections and often denoted by $\square^{\mathsf{c}}$.

\subsection{Equivalence of model structures}
As a result of the regularity, we can verify the equivalence of model structures.
To do this, we use Thomason's homotopy colimit theorem, which is proved in \cite{Tho79}.
Recall that for a small category $C$ and a functor $\Phi:C\to\mathbf{Cat}$, the Grothendieck construction $C\int F$ on $\Phi$ is the category whose morphisms are pairs $(c,x)$ of $c\in C$ and $x\in\Phi(c)$, and whose morphisms $(c,x)\to(c',x')$ are $2$-cells of the form
\[
\xymatrix{
  \ast \ar[r]^{x} \ar[dr]_{x'} |-{}="b" & \Phi(c) \ar@{=>} "b"+(1,1)  \ar[d]^{\Phi(f)} \\
  {} & \Phi(c') }
\]
with $f:c\to c'\in C$.
The composition is given by the obvious composition of $2$-cells.

\begin{theo}[Thomason's homotopy colimit theorem]
\label{theo:Thomason-hocolim}
If $C$ is a small category and $\Phi:C\to\mathbf{Cat}$ is a functor, then there is a weak equivalence
\[
\hocolim_{c\in C} N\Phi(c) \to N\left(C\int \Phi\right)
\]
in the category $\mathbf{SSet}$, where $N:\mathbf{Cat}\to\mathbf{SSet}$ is the usual nerve functor.
\end{theo}

For the proof, see the original article \cite{Tho79}.
The following lemma is a special case of the theorem:

\begin{lemma}\label{lem:slice-hocolim}
Let $\mathcal{A}$ be a small category.
Then for each $X\in\mathcal{A}^\wedge$, there is a weak equivalence
\[
\hocolim_{(\mathcal{A}[a]\to X)\in\mathcal{A}/X} N(\mathcal{A}/a) \,\to N(\mathcal{A}/X)
\]
in the category $\mathbf{SSet}$.
\end{lemma}
\begin{proof}
Set $\Phi:\mathcal{A}/X\to\mathbf{Cat}$ to be the composition
\[
\mathcal{A}/X\xrightarrow{\text{forget}} \mathcal{A} \xrightarrow{\mathcal{A}/(\blankdot)} \mathbf{Cat}
\]
and apply theorem \ref{theo:Thomason-hocolim}.
Then we have a weak equivalence
\[
\hocolim_{(\mathcal{A}[a]\to X)\in\mathcal{A}/X} N(\mathcal{A}/a)\,\to N\left((\mathcal{A}/X)\int\Phi\right)\,.
\]
Hence, to obtain a required weak equivalence, it suffices to show that the forgetful functor
\[
\beta:\left((\mathcal{A}/X)\int\Phi\right) \to \mathcal{A}/X
\]
is a homotopy equivalence in $\mathbf{Cat}$.

Notice that the category $(\mathcal{A}/X)\int\Phi$ is the category whose objects are sequences of the form
\[
\mathcal{A}[a_0] \xrightarrow{f} \mathcal{A}[a] \xrightarrow{x} X
\]
in $\mathcal{A}^\wedge$, and whose morphisms are diagrams
\[
\xymatrix@R=2ex{
  \mathcal{A}[a_0] \ar[r]^{f} \ar[dd] & \mathcal{A}[a] \ar[dr]^{x} \ar[dd] & {} \\
  && X \\
  \mathcal{A}[a'_0] \ar[r]^{f'} & \mathcal{A}[a'] \ar[ur]_{x'} & {} }
\]
in $\mathcal{A}^\wedge$.
We define a functor $\gamma:\mathcal{A}/X\to\left((\mathcal{A}/X)\int\Phi\right)$ which sends each object $(\mathcal{A}[a]\xrightarrow{x} X)\in\mathcal{A}/X$ to the sequence
\[
\mathcal{A}[a] =\joinrel= \mathcal{A}[a] \xrightarrow{x} X
\]
in $(\mathcal{A}/X)\int\Phi$.
Then we clearly have $\beta\gamma=\mathrm{Id}$.
Moreover, there is a natural transformation $\mathrm{Id}\to\gamma\beta$ depicted as follows:
\[
\xymatrix@R=2ex{
  \mathcal{A}[a_0] \ar[r]^{f} \ar[dd]_{f} & \mathcal{A}[a] \ar[dr]^{x} \ar@{=}[dd] & {} \\
  && X \\
  \mathcal{A}[a] \ar@{=}[r] & \mathcal{A}[a] \ar[ur]_{x} & {} }
\]
Therefore $\gamma$ is a homotopy inverse of $\beta$, and the result follows.
\end{proof}

\begin{theo}\label{theo:regular-infty-eq}
Let $\mathcal{R}$ be an EZ cubicalizable category with an enlargement.
Suppose moreover that the model structure on $\square(\mathcal{R})^\wedge$ defined in theorem \ref{theo:cube-model} is regular.
Then a morphism $f:X\to Y\in\square(\mathcal{R})^\wedge$ is an $\infty$-equivalence if and only if the simplicial realization $|f|:|X|\to |Y|$ is a weak equivalence in $\mathbf{SSet}$.
In other words, the model structure coincides with the standard model structure on the test category $\square(\mathcal{R})$.
\end{theo}
\begin{proof}
Recall that we have two functors $\square(\mathcal{R})/(\blankdot),\mathcal{R}^+_0\!/(\blankdot):\square(\mathcal{R})\to\mathbf{Cat}$.
We define a natural transformation $\beta:\square(\mathcal{R})/(\blankdot)\to\mathcal{R}^+_0\!/(\blankdot)$ as follows:
We denote by $1_r$ the greatest element of the Boolean lattice $\mathcal{R}^+_0\!/r$ for each $r\in\square(\mathcal{R})$.
Then for $(s\xrightarrow{f} r)\in\square(\mathcal{R})/r$, we set
\[
\beta(f) = f_\ast(1_s) \in \mathcal{R}^+_0\!/r\,.
\]
Clearly it extends to a functor and is natural with respect to $r\in\square(\mathcal{R})$.
Moreover, since both categories $\square(\mathcal{R})/r$ and $\mathcal{R}^+_0\!/r$ have a terminal object so contractible, $\beta$ is a pointwise Thomason equivalence.
Hence for each $X\in\square(\mathcal{R})^\wedge$, we obtain a weak equivalence
\begin{equation}
\label{eq:slice-trans}
\hocolim_{(\square[r]\to X)\in\square(\mathcal{R})/X} N(\square(\mathcal{R})/r)
\to \hocolim_{(\square[r]\to X)\in\square(\mathcal{R})/X} N(\mathcal{R}^+_0\!/r)\,.
\end{equation}
On the other hand, since the functor $|\blankdot|:\square(\mathcal{R})^\wedge\to\mathbf{SSet}$ is a left Quillen functor, it commutes with homotopy colimits.
Thus, we have weak equivalences
\[
\hocolim_{(\square[r]\to X)\in\square(\mathcal{R})/X} N(\mathcal{R}^+_0\!/r)
\simeq\hocolim_{(\square[r]\to X)\in\square(\mathcal{R})/X} |\square[r]|
\to\left| \hocolim_{(\square[r]\to X)\in\square(\mathcal{R})/X} \square[r]\right|\,.
\]
By the definition of the model structure, the simplicial realization preserves weak equivalences.
Since we assumed the model structure is regular, we obtain a weak equivalence
\begin{equation}
\label{eq:hocolim-realization}
\hocolim_{(\square[r]\to X)\in\square(\mathcal{R})/X} N(\mathcal{R}^+_0\!/r) \to |X|\,.
\end{equation}
Combining \eqref{eq:slice-trans}, \eqref{eq:hocolim-realization} and lemma \ref{lem:slice-hocolim}, we finally obtain the following zig-zag of weak equivalences:
\[
\begin{split}
|X|
&\leftarrow \hocolim_{(\square[r]\to X)\in\square(\mathcal{R})/X} N(\mathcal{R}^+_0\!/r) \\
&\leftarrow \hocolim_{(\square[r]\to X)\in\square(\mathcal{R})/X} N(\square(\mathcal{R})/r) \\
&\rightarrow N(\square(\mathcal{R})/X)\,.
\end{split}
\]
Therefore the result follows.
\end{proof}

\begin{corol}\label{cor:skeletal-unique-homotopy}
Let $\mathcal{R}$ be a EZ cubicalizable category with an enlargement.
If $\mathcal{R}$ has no non-trivial isomorphism, then the model structure on $\square(\mathcal{R})^\wedge$ defined in theorem \ref{theo:cube-model} coincides with the standard model structure.
\end{corol}

For example, in the category $\square(\widetilde\Delta)^\wedge$ of cubical sets with connections, a morphism $f:X\to Y$ is an $\infty$-equivalence if and only if its simplicial realization $|f|:|X|\to |Y|$ is a weak equivalence.
Moreover, since the monoidal structure on $\square(\widetilde\Delta)^\wedge$ is compatible with the model structure, for any object $X\in\square(\widetilde\Delta)^\wedge$, the functor $X\otimes(\blankdot)$ preserves $\infty$-equivalences by Ken Brown's lemma.

\subsection{Comparison of $\square(\mathcal{R})$ with $\square(\mathcal{R}^+_0)$}
Let $\mathcal{R}$ be an EZ cubicalizable category with thin-powered structure $\mathcal{R}^+_0$ and an enlargement.
In this section, to distinguish the notations, we denote by $\square_{\mathcal{R}}[r]$ the representable presheaf over $\square(\mathcal{R})$ and by $\square_{\mathcal{R}^+_0}[r]$ the representable presheaf over $\square(\mathcal{R}^+_0)$.
We consider the following condition to $\mathcal{R}$:
\begin{enumerate}[label={\rm$(\clubsuit)$}]
  \item\label{cond:club} For each $r\in\square(\mathcal{R})$, there is a zigzag of homotopy equivalences connecting $\square_{\mathcal{R}}[r]$ to the terminal object.
\end{enumerate}
In section \ref{subsec:homotopy-cyl}, we proved that, for the inclusion $i:\square(\mathcal{R}^+_0)\hookrightarrow\square(\mathcal{R})$, the restriction functor $i^\ast:\square(\mathcal{R})\to\square(\mathcal{R}^+_0)$ preserves homotopy equivalence (corollary \ref{cor:res-h-pres}).
On the other hand, by corollary \ref{cor:skeletal-unique-homotopy}, a homotopy equivalence $f:X\to Y$ in $\square(\mathcal{R}^+_0)$ is an $\infty$-equivalence.
Combining these result, we obtain the following result (see also section 5 in \cite{Isa}):

\begin{prop}\label{prop:res-Thom}
Let $\mathcal{R}$ be an EZ cubicalizable category with thin-powered structure $\mathcal{R}^+_0$ and an enlargement.
Suppose $\mathcal{R}$ satisfies the condition \ref{cond:club}.
Then for each object $X\in\square(\mathcal{R})^\wedge$, the canonical functor
\[
\square(\mathcal{R}^+_0)/i^\ast X \to \square(\mathcal{R})/X
\]
is a Thomason equivalence.
\end{prop}
\begin{proof}
By Quillen's theorem A, it suffices to show that for each cell $x:\square_{\mathcal{R}}[r]\to X$, the comma category $\commacat{(\square(\mathcal{R}^+_0)/i^\ast X)}{x}$ is aspherical.
By the adjunction $i_!\dashv i^\ast$, the objects of the category $\commacat{(\square(\mathcal{R}^+_0)/i^\ast X)}{x}$ are identified with diagrams in $\square(\mathcal{R}^+_0)^\wedge$ of the form
\[
\xymatrix@C=1.5em{
  \square_{\mathcal{R}^+_0}[r'] \ar[rr] \ar[dr] && i^\ast\square_{\mathcal{R}}[r] \ar[dl]^{i^\ast(x)} \\
  & i^\ast X & }
\]
Hence we have an isomorphism $\commacat{(\square(\mathcal{R}^+_0)/i^\ast X)}{x}\simeq \square(\mathcal{R}^+_0)/i^\ast\square_{\mathcal{R}}[r]$ of categories.
The condition \ref{cond:club} and corollary \ref{cor:res-h-pres} imply that there is a zigzag of homotopy equivalences in $\square(\mathcal{R}^+_0)^\wedge$ connecting $i^\ast \square_{\mathcal{R}}[r]$ to the terminal object $\ast$.
It follows from corollary \ref{cor:skeletal-unique-homotopy} that $i^\ast\square_{\mathcal{R}}[r]\to\ast$ is an $\infty$-equivalence.
Since the category $\square(\mathcal{R}^+_0)$ has the terminal object by \ref{sublem:cube-initial} in lemma \ref{lem:cube-r0}, we obtain the following sequence of Thomason equivalences:
\[
\commacat{(\square(\mathcal{R}^+_0)/i^\ast X)}{x}
\simeq \square(\mathcal{R}^+_0)/i^\ast\square_{\mathcal{R}}[r]
\to \square(\mathcal{R}^+_0)
\to \ast\,.
\]
\end{proof}

\begin{corol}\label{cor:res-infty-eq}
In the situation above, the restriction $i^\ast:\square(\mathcal{R})^\wedge\to\square(\mathcal{R}^+_0)$ preserves and reflects $\infty$-equivalence.
\end{corol}

An example of $\mathcal{R}$ satisfying the condition in proposition \ref{prop:res-Thom} is the category $\widetilde\Delta$.
Indeed, define $\delta_n:\underline{n-1}\to\underline{n}\in\widetilde\Delta$ and $\sigma_n:\underline{n+1}\to\underline{n}\in\widetilde\Delta$ by
\[
\delta_n(k) = k\,,
\]
and
\[
\sigma_n(k) = \begin{cases} k &\quad\text{if $0\le k\le n-1$} \\ n-1 &\quad\text{if $k=n$} \end{cases}\,.
\]
Note that we have an enlargement $c:\widetilde\Delta\to\widetilde\Delta$ defined as $c(\underline{n})=\underline{n+1}$ and $\iota = \delta_{n+1}$.
Then the morphism
\[
\sigma_n:c(\underline{n}) = \underline{n+1}\to \underline{n}
\]
in $\square(\widetilde\Delta)$ gives rise to a homotopy from $\mathrm{id}_{\underline{n}}$ to $\delta_n\delta_n^\dagger$.
Hence $\delta_n^\dagger:\square_{\widetilde\Delta}[n]\to\square_{\widetilde\Delta}[n-1]$ is a homotopy equivalence, and we obtain the sequence
\[
\square_{\widetilde\Delta}[n]
\xrightarrow{\delta_n^\dagger}\square_{\widetilde\Delta}[n-1]
\xrightarrow{\delta_{n-1}^\dagger}\dots
\xrightarrow{\delta_1^\dagger}\square_{\widetilde\Delta}[0] = \ast
\]
of homotopy equivalences.
More generally, for any crossed $\widetilde\Delta$-group $G$ compatible with $\widetilde\Delta^+$ and such that $G(0)=\{\mathrm{pt}\}$, the discussion above makes sense, so that $\widetilde\Delta G$ satisfies the condition \ref{cond:club}.
It follows from proposition \ref{prop:res-Thom} that the restriction functor $i^\ast:\square(\widetilde\Delta G)\to\square(\widetilde\Delta^+)=\square$ preserves and reflects $\infty$-equivalences.

\begin{prop}\label{prop:res-unit}
Let $\mathcal{R}$ be an EZ cubicalizable category with thin-powered structure $\mathcal{R}^+_0$ and an enlargement.
Suppose $\mathcal{R}$ satisfies the condition \ref{cond:club}.
Then the unit $\eta:\mathrm{Id}\to i^\ast i_!$ of the adjunction $i_!:\square(\mathcal{R}^+_0)^\wedge\leftrightarrows\square(\mathcal{R})^\wedge:i^\ast$ is an $\infty$-equivalence.
\end{prop}
\begin{proof}
We first prove that $\eta:X\to i^\ast i_! X$ is an $\infty$-equivalence for every $n$-skeletal object $X\in\square(\mathcal{R}^+_0)^\wedge$ by induction on $n$.
Notice that since every $0$-skeletal object is a coproduct of the terminal object $\ast$, $\eta$ is an isomorpshim on $0$-skeletal objects.
Suppose $\eta$ is an $\infty$-equivalence on $k$-skeletal objects for $0\le k\le n$, and $X\in\square(\mathcal{R}^+_0)^\wedge$ is $(n+1)$-skeletal.
By proposition \ref{prop:sk-cell-extend} and proposition \ref{prop:cell-stab}, we have the following pushout square:
\[
\renewcommand{\objectstyle}{\displaystyle}
\xymatrix@R-1ex{
  \raisedunderop{\coprod}{\substack{x:\square_{\mathcal{R}^+_0}[r]\to X\\ \deg r=n}} \partial\square_{\mathcal{R}^+_0}[r] \ar[r] \ar[d] \ar@{}[dr]|(.6){\pocorner} & \operatorname{sk}_nX \ar[d] \\
  \raisedunderop{\coprod}{\substack{x:\square_{\mathcal{R}^+_0}[r]\to X\\ \deg r=n}} \square_{\mathcal{R}^+_0}[r] \ar[r] & X }
\]
Since each objects in the square is cofibrant and each $\partial\square_{\mathcal{R}^+_0}[r]\hookrightarrow\square_{\mathcal{R}^+_0}[r]$ is a cofibration in the standard model structure, the square is also a homotopy pushout.
Moreover, since functors $i^\ast$ and $i_!$ preserve colimits and cofibrations, the functor $i^\ast i_!$ preserves homotopy colimits.
Now, $\eta:\square_{\mathcal{R}^+_0}[r]\to i^\ast i_!\square_{\mathcal{R}^+_0}[r]\simeq i^\ast\square_{\mathcal{R}}[r]$ is an $\infty$-equivalence by the assumption, and $\eta:\partial\square_{\mathcal{R}^+_0}[r]\to i^\ast i_!\partial\square_{\mathcal{R}^+_0}[r]$ and $\eta:\operatorname{sk}_n X\to i^\ast i_!\operatorname{sk}_n X$ are $\infty$-equivalence by the induction hypothesis.
It follows that the unit $\eta:X\to i^\ast i_! X$ is a homotopy pushout of $\infty$-equivalences, so $\eta:X\to i^\ast i_!X$ is itself $\infty$-equivalence.

Finally, we show $\eta:X\to i^\ast i_! X$ is an $\infty$-equivalence for an arbitrary $X\in\square(\mathcal{R}^+_0)^\wedge$.
Since $\operatorname{sk}_n X\to\operatorname{sk}_{n+1} X$ is a cofibration in the standard model structure, the sequence
\[
\operatorname{sk}_0 X \to \operatorname{sk}_1 X \to \dots
\]
is projective cofibrant in $(\square(\mathcal{R}^+_0)^\wedge)^{\mathbb{N}}$.
Hence the natural morphism
\[
\hocolim_{n\to\infty} \operatorname{sk}_n X \to X
\]
is an $\infty$-equivalence.
Similarly, the natural morphism
\[
\hocolim_{n\to\infty} i^\ast i_!\operatorname{sk}_n X \to i^\ast i_! X
\]
is also an $\infty$-equivalence.
As shown above, each $\eta:\operatorname{sk}_n X\to i^\ast i_!\operatorname{sk}_n X$ is an $\infty$-equivalence.
It follows that $\eta:X\to i^\ast i_!X$ is an $\infty$-equivalence as required.
\end{proof}

\begin{corol}\label{cor:leftKan-leftQuil}
In the situation above, $i_!:\square(\mathcal{R}^+_0)^\wedge\to\square(\mathcal{R})$ preserves $\infty$-equivalence.
\end{corol}
\begin{proof}
For a morphism $f:X\to Y\in\square(\mathcal{R}^+_0)^\wedge$, by corollary \ref{cor:res-infty-eq}, the morphism $i_! f:i_! X\to i_!Y$ is an $\infty$-equivalence in $\square(\mathcal{R})^\wedge$ if and only if $i^\ast i_!f:i^\ast i_!X\to i^\ast i_!Y$ is an $\infty$-equivalence in $\square(\mathcal{R}^+_0)^\wedge$.
Since the unit $\eta:X\to i^\ast i_!X$ is an $\infty$-equivalence by proposition \ref{prop:res-unit}, $f$ is an $\infty$-equivalence if and only if $i^\ast i_!f$ is so.
Hence the result follows.
\end{proof}

\begin{corol}
In the situation above, the counit $\varepsilon:i_!i^\ast\to\mathrm{Id}$ of the adjunction $i_!\dashv i^\ast$ is an $\infty$-equivalence.
\end{corol}
\begin{proof}
The triangle identity of the adjunction $i_!\dashv i^\ast$ implies $(i^\ast\varepsilon)\circ(\eta i^\ast) = \mathrm{id}_{i^\ast}$.
Hence the result follows from proposition \ref{prop:res-unit} and the two-out-of-three property of $\infty$-equivalences.
\end{proof}

We finally obtain a Quillen equivalence of standard model structures:

\begin{corol}
Let $\mathcal{R}$ be an EZ cubicalizable category with thin-powered structure $\mathcal{R}^+_0$ and an enlargement.
Then if $\mathcal{R}$ satisfies the condition \ref{cond:club}, the adjunction
\[
i_!:\square(\mathcal{R}^+_0)^\wedge \leftrightarrows \square(\mathcal{R})^\wedge:i^\ast
\]
gives rise to a Quillen equivalence between standard model structures.
\end{corol}
\begin{proof}
First of all, notice that $i_!$ preserves cofibrations since it preserves cofibrations $\partial\square_{\mathcal{R}^+_0}[r]\hookrightarrow\square_{\mathcal{R}^+_0}[r]$ in the generating set given in corollary \ref{cor:ez-cellular}.
By corollary \ref{cor:leftKan-leftQuil}, $i_!$ also preserves $\infty$-equivalences, so that $i_!$ is a left Quillen functor.

Now, recall that a Quillen adjunction $F:\mathcal{C}\leftrightarrows\mathcal{D}:U$ between model categories is a Quillen equivalence if and only if $U$ reflects weak equivalences between fibrant objects, and the composition
\[
X \xrightarrow{\eta} UFX \xrightarrow{Uj} UPFX
\]
is a weak equivalence for every cofibrant $X\in\mathcal{C}$ (see \cite{Hov}), where $\eta$ is the unit of the adjunction and $j:\mathrm{Id}\to P$ is a fibrant replacement in $\mathcal{D}$.
By corollary \ref{cor:res-infty-eq}, the restriction $i^\ast$ reflects all $\infty$-equivalences.
Moreover, for a fibrant replacement $j:\mathrm{Id}\to P$ in $\square(\mathcal{R})^\wedge$, by proposition \ref{prop:res-unit} and corollary \ref{cor:res-infty-eq}, the morphism
\[
X \xrightarrow{\eta} i^\ast i_!X \xrightarrow{i^\ast(j)} i^\ast Pi_! X
\]
is a composition of $\infty$-equivalences.
Thus, the Quillen adjunction $i_!\dashv i^\ast$ is a Quillen equivalence.
\end{proof}

\subsection{The spatial model structure}
Using the results in the previous section, we can give another model structure on $\square(\mathcal{R})^\wedge$.

\begin{theo}\label{theo:model-spatial}
Let $\mathcal{R}$ be an EZ cubicalizable category with an enlargement.
If $\mathcal{R}$ satisfies the condition \ref{cond:club}, then there is a cofibrantly generated model structure on $\square(\mathcal{R})^\wedge$ such that
\begin{itemize}
  \item $f:X\to Y$ is a weak equivalence if and only if it is an $\infty$-equivalence;
  \item the class of cofibrations is the saturated class generated by the set
\[
S_0:=\left\{\partial\square_{\mathcal{R}}[r]\hookrightarrow\square_{\mathcal{R}}[r]\,:\,r\in\square(\mathcal{R})\right\}\,;
\]
  \item the class of fibrations is the right orthogonal class of trivial cofibrations.
\end{itemize}
Moreover, if $\mathcal{R}^+_0$ is the thin-powered structure on $\mathcal{R}$, then the inclusion $i:\square(\mathcal{R}^+_0)\hookrightarrow\square(\mathcal{R})$ induces a Quillen equivalence
\[
i_!:\square(\mathcal{R}^+_0)^\wedge\leftrightarrows\square(\mathcal{R})^\wedge:i^\ast
\]
between the standard model structure on $\square(\mathcal{R}^+_0)^\wedge$ and the model structure on $\square(\mathcal{R})^\wedge$ defined above.
\end{theo}
\begin{proof}
The first statement is a consequence of a standard argument on lifting model structures along adjuncitons (see Theorem 3.3 in \cite{Cra95}).
Let $J$ be the generating set of trivial cofibrations in $\square(\mathcal{R}^+_0)^\wedge$, which in fact exists since the standard model structure is cofibrantly generated.
We denote by $\mathcal{W}_\infty$ the class of $\infty$-equivalences in $\square(\mathcal{R})^\wedge$.
Then it suffices to verify that the triple $(\mathcal{W}_\infty,S_0,i_!J)$ satisfies the conditions in \cite{Hov} for generating a model structure.
More precisely, if we denote by $\operatorname{Sat}(T)$ the saturated class of morphisms generated by $T$, we will show the following:
\begin{itemize}
  \item $\operatorname{Sat}(i_!J)\subset\mathcal{W}_\infty\cap\operatorname{Sat}(S_0)$.
  \item $S_0\pitchfork f\iff i_!J\pitchfork f \text{ and } f\in\mathcal{W}_\infty$.
\end{itemize}

Notice that since $\mathcal{R}^+_0$ has no non-trivial isomorphism, the class of cofibration in $\square(\mathcal{R}^+_0)^\wedge$ is generated by morphisms of the form
\[
S'_0=\partial\square_{\mathcal{R}^+_0}[r] \hookrightarrow \square_{\mathcal{R}^+_0}[r]\,,
\]
and we have $S_0=i_!S'_0$.
Since every element of $J$ is a cofibration in $\square(\mathcal{R}^+_0)$ and $i_!$ commutes with colimits, it follows that we have  $i_!J\subset\operatorname{Sat}(S_0)$ so that $\operatorname{Sat}(i_!J)\subset\operatorname{Sat}(S_0)$.

Next, since $S_0=i_! S'_0$, by the adjointness, we have $S_0\pitchfork f\iff S'_0\pitchfork i^\ast f$ for morphisms $f:X\to Y\in\square(\mathcal{R})^\wedge$.
The standard model structure on $\square(\mathcal{R}^+_0)^\wedge$ is generated by $S'_0$ and $J$, we also have $S'_0\pitchfork f'\iff J\pitchfork f' \text{ and } f'\in\mathcal{W}'_\infty$ for the class $\mathcal{W}'_\infty$ of $\infty$-equivalences in $\square(\mathcal{R}^+_0)^\wedge$.
Moreover, by corollary \ref{cor:res-infty-eq}, $i^\ast f\in\mathcal{W}'_\infty$ if and only if $f\in\mathcal{W}_\infty$.
Combining these eqiuvalences, we obtain the required equivalence $S_0\pitchfork f\iff i_!J\pitchfork f \text{ and } f\in\mathcal{W}_\infty$.

Finally, notice that we have an obvious Quillen equivalence
\[
\square(\mathcal{R})^\wedge\xrightarrow{\mathrm{Id}} \square(\mathcal{R})^\wedge_{\mathrm{st}}\,,
\]
where the left $\square(\mathcal{R})^\wedge$ is equipped with the model structure defined above while the right is equipped with the standard model structure.
By corollary \ref{cor:leftKan-leftQuil}, we also have a Quillen equivalence
\[
\square(\mathcal{R}^+_0)^\wedge\xrightarrow{i_!} \square(\mathcal{R})^\wedge_{\mathrm{st}}\,.
\]
Since the Quillen equivalences satisfy the two-out-of-three property, it follows that
\[
\square(\mathcal{R}^+_0)^\wedge\xrightarrow{i_!} \square(\mathcal{R})^\wedge
\]
is a Quillen equivalence.
\end{proof}

We call the model structure define in theorem \ref{theo:model-spatial} the spatial model structure.
It is a generalization of the spatial model structure on $\square_\Sigma=\square(\widetilde\Delta\Sigma)$ discussed in \cite{Isa}.
Note that Berger and Moerdijk pointed out in \cite{BM} that, for an EZ category $\mathcal{S}$, the saturated class of morphisms generated by
\[
\left\{\partial\mathcal{S}[s]\to\mathcal{S}[s]\,:\,s\in\mathcal{S}\right\}
\]
has a really natural action on the generalized Reedy model structure.

\begin{rem}
Let $\mathcal{R}$ be an EZ cubicalizable category with an enlargement.
Suppose $\mathcal{R}$ satisfies the condition \ref{cond:club} and has no non-trivial isomorphism.
Then the set
\[
S_0:=\left\{\partial\square_{\mathcal{R}}[r]\hookrightarrow\square_{\mathcal{R}}[r]\,:\,r\in\square(\mathcal{R})\right\}
\]
generates the whole class of monomorphisms.
Hence, in this case, the spatial model structure coincides with the standard model structure on $\square(\mathcal{R})^\wedge$.
\end{rem}

In the spatial model structure, we obtain the following criterion (cf. theorem \ref{theo:regular-infty-eq} and corollary \ref{cor:skeletal-unique-homotopy}):

\begin{prop}
Let $\mathcal{R}$ be an EZ cubicalizable category with an enlargement.
Suppose $\square(\mathcal{R})$ satisfies the condition \ref{cond:club}.
Then the geometric realization $|\blankdot|:\square(\mathcal{R})^\wedge\to\mathbf{SSet}$ is a left Quillen functor, where we regard $\square(\mathcal{R})$ as equipped with the spatial model structure.
Moreover, if $X,Y\in\square(\mathcal{R})^\wedge$ are cofibrant, $f:X\to Y$ is an $\infty$-equivalence if and only if the geometric realization $|f|:|X|\to |Y|$ is a weak equivalence in $\mathbf{SSet}$.
\end{prop}
\begin{proof}
Let $\mathcal{R}^+_0$ be the thin-powered structure on $\mathcal{R}$, and let $J$ be a generating set of trivial cofibration in the standard model structure on $\square(\mathcal{R}^+_0)^\wedge$.
As in the proof of theorem \ref{theo:model-spatial}, $i_!J$ is a generating set of trivial cofibrations in the spatial model structure on $\square(\mathcal{R})$.
Since $|\blankdot|:\square(\mathcal{R}^+_0)^\wedge\to\mathbf{SSet}$ is a left Quillen functor, $|\gamma|$ is a trivial cofibration in $\mathbf{SSet}$ for each $\gamma\in J$.
Now, by the uniqueness of the left Kan extension, we have $|\gamma|\simeq |i_!\gamma|$, so that $|\blankdot|:\square(\mathcal{R})^\wedge\to\mathbf{SSet}$ sends each generating trivial cofibration to a trivial cofibration.
Hence, by corollary \ref{cor:real-mono}, the geometric realization $|\blankdot|$ is a left Quillen functor.

Next, suppose $f:X\to Y$ is a morphism between cofibrant objects.
If $f$ is an $\infty$-equivalence, then $|f|:|X|\to |Y|$ is a weak equivalence since every left Quillen functor preserves weak equivalences between cofibrant objects.
Conversely, suppose $|f|$ is a weak equivalence.
We have the following diagram:
\[
\xymatrix{
  |X| \ar[r]^{|f|} \ar[d]_{|\varepsilon|} & |Y| \ar[d]^{|\varepsilon|} \\
  |i_!i^\ast X| \ar[r]^{|i_!i^\ast f|} & |i_!i^\ast Y| }
\]
Notice that the counit $\varepsilon:X\to i_!i^\ast X$ is an $\infty$-equivalence between cofibrant objects, so $|\varepsilon|:|X|\to |i_!i^\ast X|$ is a weak equivalence.
Hence, by the two-out-of-three property, $|i_!i^\ast f|$ is a weak equivalence.
The uniqueness of the left Kan extension again implies $|i_!i^\ast f| = |i^\ast f|$.
Thus, it follows from corollary \ref{cor:skeletal-unique-homotopy} and corollary \ref{cor:res-infty-eq} that $f$ is an $\infty$-equivalence.
\end{proof}

We next discuss monoidal structures.
If $\mathcal{R}$ is a monoidal cubicalizable category, then as in section \ref{subsec:monoidal-cube}, for morphisms $f:X\to Y$ and $g:Z\to W\in\square(\mathcal{R})^\wedge$, we denote by
\[
f\odot g:X\otimes W\mathop{\amalg}_{X\otimes Z} Y\otimes Z \to Y\otimes W
\]
the induced morphism.

\begin{theo}
Let $\mathcal{R}$ be a monoidal EZ cubicalizable category satisfying the condition \ref{cond:club}.
Then the convolution product on $\square(\mathcal{R})^\wedge$, which is given by the formula
\[
X\otimes Y \simeq \int^{p,q\in\square(\mathcal{R})} \square_{\mathcal{R}}[p\otimes q]\times X(p)\times Y(q)\,,
\]
is compatible with the spatial model structure.
\end{theo}
\begin{proof}
Let $\mathcal{R}^+_0$ be the thin-powered structure on $\mathcal{R}$.
Since the monoidal structure on $\square(\mathcal{R})$ restricts to $\square(\mathcal{R}^+_0)$, $\square(\mathcal{R}^+_0)$ is also a monoidal cubicalizable, and the inclusion $i:\square(\mathcal{R}^+_0)\hookrightarrow\square(\mathcal{R})$ is strongly monoidal.
It follows that $\square(\mathcal{R}^+_0)^\wedge$ the left Kan extension $i_!:\square(\mathcal{R}^+_0)^\wedge\to\square(\mathcal{R})$ is strongly monoidal.
Moreover, by corollary \ref{cor:skeletal-unique-homotopy} and proposition \ref{prop:cube-model-monoidal}, the convolution product on $\square(\mathcal{R}^+_0)^\wedge$ is compatible with the standard model structure.

Let $J$ be a generating set of trivial cofibrations of the standard model structure on $\square(\mathcal{R}^+_0)$.
As in the proof of theorem \ref{theo:model-spatial}, $i_!J$ is a generating set of trivial cofibrations in the spatial model structure.
For each $r\in\square(\mathcal{R})$, we set
\[
\beta_r:\partial\square_{\mathcal{R}}[r]\hookrightarrow\square_{\mathcal{R}}[r]
\]
to be the canonical inclusion in $\square(\mathcal{R})^\wedge$.
We already have $\beta_r\odot\beta_s\simeq\beta_{r\otimes s}$ by corollary \ref{cor:bound-act-tensor}, so that for cofibrations $f,g$ in the spatial model structure, $f\odot g$ is also a cofibration.

On the other hand, we denote by $\beta'_r:\partial\square_{\mathcal{R}^+_0}[r]\hookrightarrow\square_{\mathcal{R}^+_0}[r]$ the canonical inclusion in $\square(\mathcal{R}^+_0)^\wedge$.
Since $i_!:\square(\mathcal{R}^+_0)^\wedge\to\square(\mathcal{R})$ is strongly monoidal, for each $\gamma'\in J$, we have
\[
\beta_r\odot i_!\gamma' \simeq i_!\beta'_r\odot i_!\gamma' \simeq i_!(\beta'_r\odot \gamma')\,.
\]
As mentioned above, the convolution product on $\square(\mathcal{R}^+_0)$ is compatible with the standard model structure, so $\beta'_r\odot \gamma'$ is a trivial cofibration.
By theorem \ref{theo:model-spatial}, $i_!$ is a left Quillen functor.
This implies that $\beta_r\odot i_!\gamma'$ is a trivial cofibration in the spatial model structure.
It follows that if $f$ is a cofibration and $g$ is a trivial cofibration in the spatial model structure, $f\odot g$ is a trivial cofibration.
Similarly, $g\odot f$ is also a trivial cofibration.
\end{proof}

\begin{exam}
Recall that $\widetilde\Delta$ has a canonical model structure defined by 
\[
\underline{m}\otimes\underline{n} = \underline{m+n}\,.
\]
If $G$ is a group operad (see example \ref{ex:group-operad} for the definition), the monoidal structure above extends to $\widetilde\Delta G$ as
\[
g_1\otimes g_2 = \gamma(e_2;g_1,g_2)
\]
for $g_1\in G(m)$ and $g_2\in G(n)$, where $\gamma$ is the composition operator of the operad $G$ and $e_2\in G(2)$ is the unit.
Then since $\widetilde\Delta G$ is a monoidal EZ cubicalizable category satisfying the condition \ref{cond:club}, the category $\square(\widetilde\Delta G)^\wedge$ is a monoidal model category with the spatial model structure.
\end{exam}

\begin{exam}
Let $\mathcal{B}=\{B_n\}_n$ be the group operad consisting of the braid groups $B_n$.
We fix a generator $t\in B_2$, and define a braiding
\[
\tau_{m,n}:\underline{m}\otimes\underline{n}\to\underline{n}\otimes\underline{m}
\]
by
\[
\tau_{m,n} = \gamma(t;e_m,e_n)\,.
\]
Then $\{\tau_{m,n}\}$ defines a natural isomorphism $\tau$; for example, for $g_1\in G(m)$ and $g_2\in G(n)$, we have
\[
\tau_{m,n}\circ(g_1\otimes g_2) = \gamma(t;e_m,e_n)\gamma(e_2,g_1,g_2) = \gamma(t;g_1,g_2)\,.
\]
On the other hand, we also have
\[
(g_2\otimes g_1)\circ\tau_{m,n} = \gamma(e_2;g_2,g_1)\gamma(t;e_m,e_n) = \gamma(t;g_1,g_2)\,.
\]
Hence $\tau_{m,n}\circ(g_1\otimes g_2)=(g_2\otimes g_1)\circ\tau_{m,n}$.
Moreover, by the direct calculus of the braid group $B_3$, we have
\[
\gamma(t;e_1,e_2) = \gamma(e_2;e_1,t)\gamma(e_2;t,e_1)\,.
\]
Hence we obtain the following formula:
\[
\begin{split}
\tau_{l,m+n}
&= \gamma(t;e_l,e_{m+n})
= \gamma(t;e_l,\gamma(e_2;e_m,e_n)) \\
&= \gamma(\gamma(t;e_1,e_2);e_l,e_m,e_n) \\
&= \gamma(\gamma(e_2;e_1,t)\gamma(e_2;t,e_1);e_l,e_m,e_n) \\
&= \gamma(\gamma(e_2;e_1,t);e_m,e_l,e_n)\gamma(\gamma(e_2;t,e_1);e_l,e_m,e_n) \\
&= \gamma(e_2;e_m,\gamma(t;e_l,e_n))\gamma(e_2;\gamma(t;e_l,e_m),e_n) \\
&= (e_m\otimes\tau_{l,n})\circ(\tau_{l,m}\otimes e_n)
\end{split}
\]
Similarly, we also obtain
\[
\tau_{l+m,n} = (\tau_{l,n}\otimes e_m)\circ (e_l\otimes\tau_{m,n})\,.
\]
These imply that $\tau$ is a braiding on the monoidal category $\widetilde\Delta\mathcal{B}$.
It is easily verified that the braiding $\tau$ extends to a braiding of the convolution product on $\square(\widetilde\Delta\mathcal{B})^\wedge$, so that the category $\square(\widetilde\Delta\mathcal{B})^\wedge$ is a braided monoidal model category with the spatial model structure.
\end{exam}

\bibliographystyle{plain}
\bibliography{newcube}
\end{document}